\documentclass[a4paper,11pt]{amsart}
\usepackage{amssymb}
\usepackage{amscd}
\usepackage{amsmath,amsthm}
\usepackage[colorlinks=true]{hyperref}
\usepackage{enumerate}
\usepackage{booktabs,multirow}
\usepackage{tikz}
\usepackage{tikz-cd}
\usepackage{rotating}
\usepackage{ytableau}
\usetikzlibrary{patterns}
\usetikzlibrary{decorations.pathreplacing}
\usetikzlibrary{calc,through}
\usetikzlibrary{cd}
\usetikzlibrary{arrows.meta}

\allowdisplaybreaks[1]
\numberwithin{figure}{section}
\numberwithin{equation}{section}

\title{Fuss--Catalan algebras on generalized Dyck paths via non-crossing partitions}
\author[K.~Shigechi]{Keiichi~Shigechi}	
\email{k1.shigechi AT gmail.com}
\date{\today}

\newcommand\tikzpic[2]{
\raisebox{#1\totalheight}{
\begin{tikzpicture}
#2
\end{tikzpicture}
}}

\newtheorem{theorem}[figure]{Theorem}
\newtheorem{example}[figure]{Example}
\newtheorem{lemma}[figure]{Lemma}
\newtheorem{defn}[figure]{Definition}
\newtheorem{prop}[figure]{Proposition}
\newtheorem{cor}[figure]{Corollary}

\newtheorem{remark}[figure]{Remark}
\begin{document}

\begin{abstract}
We study the Fuss--Catalan algebras, which are generalizations of the Temperley--Lieb 
algebra and act on generalized Dyck paths, through non-crossing partitions.
First, the Temperley--Lieb algebra is defined on non-crossing partitions, and 
a bijection between a Dyck path and a non-crossing partition is shown 
to be compatible with the Temperley--Lieb algebra on Dyck paths, or 
equivalently chord diagrams. We show that the Kreweras endomorphism on non-crossing partitions
is equivalent to the rotation of chord diagrams under the bijection.
Secondly, by considering an increasing $r$-chain in the graded lattice of non-crossing partitions,
we define the Fuss--Catalan algebras on increasing $r$-chains. Through a bijection between 
an increasing $r$-chain and a generalized Dyck path, one naturally obtains the Fuss--Catalan algebra
on generalized Dyck paths.
As generalizations of the Fuss--Catalan algebra, we introduce the one- and two-boundary 
Fuss--Catalan algebras. 
Increasing $r$-chains of symmetric non-crossing partitions give symmetric generalized Dyck 
paths by the bijection, and the boundary Fuss--Catalan algebras naturally act on them.
We show that these representations are compatible with the diagrammatic representations of the 
algebras by use of generalized chord diagrams. 
Thirdly, we discuss the integrability of the Fuss--Catalan algebras.
For the Fuss--Catalan algebras with boundaries, we obtain a new solution of the reflection 
equation in the case of $r=2$.
\end{abstract}

\maketitle

\setcounter{tocdepth}{1}
\tableofcontents

\section{Introduction}
The Dyck paths of size $n$ are one of the combinatorial objects whose total 
number is given by the Catalan number 
$C_{n}:=\genfrac{}{}{}{}{1}{n+1}\genfrac{(}{)}{0pt}{}{2n}{n}=1,2,5,14,42,\ldots$.
The diagrammatic representation  \cite{Kau87} of the Temperley--Lieb algebra $\mathbb{TL}_{2n}$ \cite{TL71}
allows us to act $\mathbb{TL}_{2n}$ on the set of the Dyck paths of size $n$ by
identifying a Dyck path with a diagram called a chord diagram.
A chord diagram is a non-crossing complete matchings on $2n$ vertices
(see e.g. item 61 in Section 2 in \cite{Sta15}).
The Fuss--Catalan algebra introduced by D.~Bisch and V.~Jones in \cite{BisJon97} is a 
generalization of the Temperley--Lieb algebra, which can be viewed as an $r$-color 
generalization.
Since then, the Fuss--Catalan algebra has been studied in mathematics and mathematical
physics (see for example \cite{BabGep07,BisJon97b,Dif98,Hus19,Lan01}).
The Fuss--Catalan algebra acts on a generalization of chord diagrams.
The generalization of chord diagrams can naturally be identified with 
a generalization of Dyck paths called the $r$-Dyck paths.
The total number of $r$-Dyck paths of size $n$ is given by the Fuss--Catalan number, 
which contains the Catalan number as a special case $r=1$.
Therefore, the Fuss--Catalan algebra naturally acts on $r$-Dyck paths in the same way that the 
Temperley--Lieb algebra acts on Dyck paths.
Like Catalan combinatorics  for the Temperley--Lieb algebra, it is natural 
to ask what kind of combinatorics behind the Fuss--Catalan algebra is.
A non-crossing partition introduced below plays a central role to answer this question.

A non-crossing partition $\pi$ of the set $[1,n]:=\{1,2,\ldots,n\}$ is another 
combinatorial object whose total number is again given by the Catalan number \cite{Sta15}.
The systematic study of non-crossing partitions was initiated by G.~Kreweras in \cite{Kre72} 
(see also \cite{Pou72}).
The non-crossing partitions are extensively studied in combinatorics 
(see for example \cite{Arm09,Bia97,Sim94,Sim00,Sta97}).
By establishing a bijection between a Dyck path and a non-crossing partition,
we show that the Temperley--Lieb algebra naturally acts on non-crossing partitions.
Since the Fuss--Catalan algebra is a generalization of Temperley--Lieb algebra,
the set which the Fuss--Catalan algebra acts on supposed to be a generalization 
of non-crossing partitions.
This is achieved by considering increasing $r$-chains of non-crossing partitions.
More precisely, recall that the partially ordered set of non-crossing partitions 
is a graded lattice \cite{Kre72,Sim00}. 
Then, we have an increasing $r$-chain $\pi^{(r)}:=(\pi_1,\ldots,\pi_{r})$ such that
$\pi_1\le\ldots\le\pi_{r}$.
By the bijection between an $r+1$-ary tree and an increasing $r$-chains given by P.~H.~Edelman 
in \cite{Edel80,Edel82}, the number of increasing $r$-chains is equal to that of $r+1$-ary trees,
equivalently, the Fuss--Catalan number (see also \cite{EdelSim94}).
This suggests that a natural combinatorial object for the Fuss--Catalan algebra is 
an increasing $r$-chain of non-crossing partitions.
\begin{table}[ht]
\begin{tabular}{c||c}
Algebra &  Combinatorial objects \\ \hline
Temperley--Lieb algebra & Dyck paths, chord diagrams, non-crossing partitions \\ \hline
Fuss--Catalan algebra & $r$-Dyck paths, generalized chord diagrams, increasing $r$-chains
\end{tabular}
\caption{}
\label{table:Alcom}
\end{table}
In Table \ref{table:Alcom}, we summarize the algebras and corresponding combinatorial 
objects.

In this paper, we study the Fuss--Catalan and Temperley--Lieb algebras from a combinatorial view 
by use of non-crossing partitions.
In the case of Temperely--Lieb algebra, we introduce a bijection $\Psi$ between a Dyck path
and a non-crossing partition. Since we have a natural bijection between a Dyck path
and a chord diagram, we have a bijection between a non-crossing partition and 
a chord diagram, and denote this bijection by $\Psi$ by abuse of notation.
On the one hand, a chord diagram possesses a rotation $\sigma$, which is equivalent 
to the jeu de taquin operation given by M.-P.~Sch\"utzenberger \cite{Sch77}. 
On the other hand, a non-crossing partition possesses another rotation $\rho$ which is called 
Kreweras endomorphism \cite{Kre72}.
We show that $\sigma$ on a chord diagram is compatible with $\rho$ on a non-crossing partition
under the bijection $\Psi$ (Proposition \ref{prop:Fe}).
By introducing generators acting on a non-crossing partition, we show that 
the set of the generators generates the Temperley--Lieb algebra on non-crossing partitions.

Let $e_{i}$, $1\le i\le 2n-1$, be generators of $\mathbb{TL}_{2n}$, 
and $C_{j}$ and $\pi_{j}$ a chord diagram and a non-crossing partition  such that $C_{j}=\Psi(\pi_{j})$
for $j\in\{1,2\}$.
As a consequence of $\Psi$, we prove that $C_{2}=e_{i}C_{1}$ if and only if 
$\pi_{1}$ covers $\pi_{2}$, or $\pi_2$ covers $\pi_1$. 
In some cases, there may be no $i$ such that $C_{2}=e_{i}C_{1}$ for given $C_1$ and $C_2$
(Proposition \ref{prop:NCchord}). 
In this sense, the cover relation in the poset of non-crossing partitions reflects 
the action of a generator on chord diagrams.

By considering an increasing $r$-chain, the results on the Temperley--Lieb algebra
can be applied to the case of the Fuss--Catalan algebras.
We first establish a bijection between an increasing $r$-chain and an $r$-Dyck 
path, which is a generalization of the bijection between a non-crossing partition
and a Dyck path given in \cite{Stu13} by C.~Stump (Section \ref{sec:GDPic}).
Especially, we give two bijections $\Psi$ and $\Phi$ between an increasing $r$-chain and  
a generalized chord diagram.
We show that these two bijections are related by the rotation $\sigma^{(r)}$ on a generalized 
chord diagram (Proposition \ref{prop:PhiPsi}).
By introducing the notion of another combinatorial object called cover-exclusive Dyck tilings,
we give a map from an increasing $r$-chain to a generalized chord diagram.
This description is essentially the same as $\Phi$ (Proposition \ref{prop:Philambda}).

Each non-crossing partition in an $r$-chain gives a chord diagram by $\Psi$.
Since the $r$-chain is increasing, i.e., we have $\pi_i\le\pi_{i+1}$ for $1\le i\le r-1$,
we can prove that a superposition of $r$ chord diagrams corresponding to $\pi_{i}$, $1\le i\le r$, 
is admissible as a generalized chord diagram. More precisely, the chord diagrams $C_{i}:=\Psi(\pi_i)$ and $C_{j}:=\Psi(\pi_{j})$
with $i<j$ do not intersect (Proposition \ref{prop:chainadd}).
Thus, we have independence of $r$ chord diagrams, and this is the reason why 
the Fuss--Catalan algebra is viewed as an $r$-color generalization of the Temperley--Lieb 
algebra.
We define the Fuss--Catalan algebra on an increasing $r$-chain by use of the property that 
$r$ chord diagrams are independent and non-crossing.

After reviewing the diagrammatic representation of the Fuss--Catalan algebra which 
acts on generalized chord diagrams as in \cite{BisJon97,Dif98}, we show that the 
Fuss--Catalan algebra acting on increasing $r$-chains is isomorphic to that on 
generalized chord diagrams (Proposition \ref{prop:isoNCTL}).
The bijection $\Psi$ between an increasing $r$-chain and a generalized 
chord diagram plays a central role.

The one- and two-boundary Temperley--Lieb algebras  are generalizations 
of the Temperley--Lieb algebra. 
The algebras with boundaries naturally appear when we consider a physical system 
with boundaries.
Especially, they arise by considering the six-vertex model with the addition 
of integrable boundary terms \cite{deGNicPyaRit05,deGPya04}. 
These algebras are well studied in mathematical physics to 
describe a physical system with one or two boundaries
\cite{deGNic09,deGNicPyaRit05,deGPya04,MarSal93,MarSal94,MarWoo00,MarWoo03}.

Since the one-boundary Temperley--Lieb algebra acts on symmetric Dyck paths,
it is natural to consider ``symmetric" non-crossing partitions on which 
the algebra acts.
In fact, a symmetric non-crossing partition turns out to be a non-crossing partition 
which is symmetric along the vertical line in the middle in the pictorial
representation such that the point labeled $1$ is on the vertical line (Proposition \ref{prop:SNCtoSC}).
The bijection $\Psi$ behaves nicely since it gives symmetric Dyck paths or symmetric chord diagrams
from symmetric non-crossing partitions.
The existence of the boundary can be reflected by the existence of symmetric chords in a chord diagram.
The symmetry along the vertical line captures the boundary.
As in the case of the Fuss--Catalan algebra, we can define the one-boundary 
Fuss--Catalan algebra on increasing $r$-chains of symmetric non-crossing partitions.

To define the two-boundary Temperley--Lieb algebra on symmetric non-crossing partitions,
we introduce primed integers in symmetric non-crossing partitions. 
In other words, we attach more information on symmetric non-crossing partitions
by primes.
As we have already seen, the existence of a boundary is captured by the 
symmetry along the vertical line.
Since we have two boundaries, the symmetry along the vertical line itself is not enough to capture 
the second boundary.
The primed integers in a symmetric non-crossing partition correspond to 
symmetric chords in the corresponding chord diagram, and this extra 
information captures the existence of the second boundary.
This allows us to define the two-boundary Temperley--Lieb algebra 
on symmetric non-crossing partitions with primed integers.
By a similar way, we define two-boundary Fuss--Catalan algebra on 
increasing $r$-chains of symmetric non-crossing partitions with 
primed integers.

We introduce the one- and two-boundary Fuss--Catalan algebras 
on generalized chord diagrams by use of diagrammatic presentations.
A generalized diagram may have right-end points for the one-boundary Fuss--Catalan
algebra, and both left- and right-end points for the two-boundary Fuss--Catalan 
algebra.
We show that the diagrammatic one- and two-boundary Fuss--Catalan
algebras are isomorphic to those on symmetric non-crossing 
partitions without or with primed integers respectively 
(Theorems \ref{thrm:isoSNC1BTL} and \ref{thrm:2SNC2BTL}).

We discuss the integrability of the one- and two-boundary 
Fuss--Catalan algebras for $r=2$ as a special case.
In \cite{Dif98}, P. Di Francesco gave a solution of the Yang--Baxter 
equation \cite{Bax82} for the Fuss--Catalan algebra in the case of 
general $r\ge1$ (see also \cite{BabGep07}).
If the system with a boundary is integrable, 
we have the reflection equation \cite{Skl88} at the boundary.
In fact, in the case of one-boundary Fuss--Catalan algebra 
for $r=2$, we obtain a new solution of the reflection 
equation (Proposition \ref{prop:RE}).

The paper is organized as follows.
In Section \ref{sec:GDPTL}, we introduce the notion of $r$-Dyck paths,
chord diagrams, and a rotation on them. We review the Temperley--Lieb algebra 
acting on chord diagrams.
In Section \ref{sec:NCP}, we introduce non-crossing partitions, the 
Kreweras endomorphism, and the Temperley--Lieb algebra on non-crossing partitions.
The bijection $\Psi$ from a non-crossing partition to a chord diagram 
is given.
In Section \ref{sec:FCGDP}, we define the Fuss--Catalan algebra on an increasing 
$r$-chain by use of the results in previous sections.
We give a bijection from an increasing $r$-chain to an $r$-Dyck path, and 
two bijections from an increasing $r$-chain to a generalized chord diagram.
In Section \ref{sec:Pic}, we summarize the diagrammatic representation of 
the Fuss--Catalan algebra, and show that the algebra is isomorphic to 
the Fuss--Catalan algebra on increasing $r$-chains.
In Section \ref{sec:SNC}, we introduce the one- and two-boundary 
Fuss--Catalan algebras on symmetric non-crossing partitions possibly
with primed integers.
In Section \ref{sec:1bFC}, we introduce the diagrammatic representation 
of the one-boundary Fuss--Catalan algebra, and show that this algebra 
is isomorphic to the one-boundary Fuss--Catalan algebra on symmetric 
non-crossing partitions.
In Section \ref{sec:2BFC}, the diagrammatic representation of the two-boundary
Fuss--Catalan algebra is given, and we show that this algebra is isomorphic 
to the two-boundary Fuss--Catalan algebra on symmetric non-crossing 
partitions with primed integers.
Finally, we discuss the integrability of the one- or two-boundary 
Fuss--Catalan algebra in Section \ref{sec:RE}. Especially, we obtain a new solution 
of the reflection equation for $r=2$.

\section{Generalized Dyck paths and Temperley--Lieb algebra}
\label{sec:GDPTL}
\subsection{Definition}
Fix a positive integer $r\ge1$.
A {\it generalized Dyck path} of size $n$ is an up-right lattice path from 
$(0,0)$ to $(rn,n)$ such that it never goes below the line $y=x/r$.
We denote the set of generalized Dyck paths of size $n$ by $\mathcal{P}_{n}^{(r)}$.
A generalized Dyck path is also called an $r$-Dyck path for short.
We represent a generalized Dyck path of size $n$ by a word consisting of two 
alphabets $U$ and $R$, where $U$ (resp. $R$) stands for an up (resp. right) step 
in a generalized Dyck path.
For example, $\mathcal{P}_{2}^{(3)}$ consists of four $3$-Dyck paths:
\begin{align*}
URRRURRR \qquad URRURRRR \qquad URURRRRR \qquad UURRRRRR
\end{align*}

The number of generalized Dyck paths in $\mathcal{P}_{n}^{(r)}$ is given by 
the well-known Fuss--Catalan numbers, {\it i.e.}, 
\begin{align*}
|\mathcal{P}_{n}^{(r)}|=\genfrac{}{}{}{}{1}{nr+1}\genfrac{(}{)}{0pt}{}{n(r+1)}{n}.
\end{align*}
When $r=1$, a $1$-Dyck path in $\mathcal{P}_{n}^{(1)}$ coincides with 
the standard notion of a Dyck path of size $n$.
The total number of Dyck paths of size $n$ is given by the Catalan number.

An {\it $r$-Young diagram} of size $n$ is a diagram consisting of $2rn$ cells 
in such a way that there are $n$ rectangles of shape $1\times r$ in the first row 
and $rn$ rectangles of shape $1\times 1$ in the second row.
An {\it $r$-Young tableau} of size $n$ is a tableau such that each rectangle 
is filled by a positive integer. 
The integers in $r$-Young tableau are increasing from left to right and from top to bottom.

We have a bijection between an $r$-Dyck path of size $n$ and an $r$-Young tableau of 
size $n$.
Recall that a $r$-Dyck path $P$ is a word of alphabets $\{U,R\}$.
Let $x_i$, $1\le i\le n$, be the position of the $i$-th up step in $P$, and 
$y_{i}$, $1\le i\le rn$, be the position of the $i$-th right step in $P$. 
We obtain an $r$-Young tableau $Y(P)$ by putting integers in $\{x_i:1\le i\le n\}$ 
(resp. $\{y_i:1\le i\le rn\}$) from left to right in the first (resp. second) 
row of $Y(P)$.
This map from $P$ to $Y(P)$ is obviously invertible, and we obtain a bijection
between the two sets.

\begin{example}
Let $(n,r)=(2,2)$. We have three $2$-Young tableaux as shown in Figure \ref{fig:YT}.
\begin{figure}[ht]
\tikzpic{-0.5}{[scale=0.6]
\draw(0,0)--(4,0)--(4,-2)--(0,-2)--(0,0)(2,0)--(2,-2)(0,-1)--(4,-1)(1,-1)--(1,-2)
(3,-1)--(3,-2);
\draw(1,-0.5)node{$1$}(3,-0.5)node{$2$};
\draw(0.5,-1.5)node{$3$}(1.5,-1.5)node{$4$}(2.5,-1.5)node{$5$}(3.5,-1.5)node{$6$};
}
\tikzpic{-0.5}{[scale=0.6]
\draw(0,0)--(4,0)--(4,-2)--(0,-2)--(0,0)(2,0)--(2,-2)(0,-1)--(4,-1)(1,-1)--(1,-2)
(3,-1)--(3,-2);
\draw(1,-0.5)node{$1$}(3,-0.5)node{$3$};
\draw(0.5,-1.5)node{$2$}(1.5,-1.5)node{$4$}(2.5,-1.5)node{$5$}(3.5,-1.5)node{$6$};
}
\tikzpic{-0.5}{[scale=0.6]
\draw(0,0)--(4,0)--(4,-2)--(0,-2)--(0,0)(2,0)--(2,-2)(0,-1)--(4,-1)(1,-1)--(1,-2)
(3,-1)--(3,-2);
\draw(1,-0.5)node{$1$}(3,-0.5)node{$4$};
\draw(0.5,-1.5)node{$2$}(1.5,-1.5)node{$3$}(2.5,-1.5)node{$5$}(3.5,-1.5)node{$6$};
}
\caption{Three $2$-Young tableaux of size $2$}
\label{fig:YT}
\end{figure}
The tableaux correspond to the $2$-Dyck paths, $U^2R^4$, $URUR^3$ and $UR^2UR^2$ 
from left to right respectively.
\end{example}

We define a rotation on $r$-Dyck paths by use of $r$-Young tableaux and the modified 
operation on a two-row Young tableau called {\it jeu de taquin}.
The jeu de taquin operation was introduced in \cite{Sch77} by M.-P.~Sch\"utzenberger. 
In our setup, the jeu de taquin is equivalent to the promotion studied in \cite{Schu72,Schu76}.

Let $Y$ be an $r$-Young tableau of size $n$. The tableau $Y$ is filled by integers in $[1,(r+1)n]$.
The modified jeu de taquin on $Y$ is defined as follows.
\begin{enumerate}[($\heartsuit$1)]
\item Delete the integer $(r+1)n$ from $Y$.
\item Let $c_0$ be the rectangle without an integer. We consider two cases:
\begin{enumerate}
\item
Suppose that $c_0$ is in the first row or in the first column of $Y$. 
We move the integer in the cell left to $c_{0}$, or above $c_0$ to the cell $c_{0}$ respectively.
\item
Suppose that $c_{0}$ is in the second row of $Y$. Let $c_1$ and $c_2$
be the cell left to and above the cell $c_{0}$. We denote by $l(c_1)$ (resp. $l(c_2)$) the integer
in the cell $c_1$ (resp. $c_2$).
We consider two cases:
\begin{enumerate}
\item $l(c_1)>l(c_2)$. We move the integer $l(c_1)$ from $c_1$ to $c_0$.
\item $l(c_1)<l(c_2)$. We move the integer $l(c_2)$ from $c_2$ to $c_0$.
\end{enumerate}
\end{enumerate}
As a result, the cell $c_1$ or $c_2$ becomes an empty cell.
\item We repeat ($\heartsuit$2) until the left-most rectangle in the first row 
becomes empty. 
\item Increase all integers by one.
\item Since the left-most cell in the first row is empty, we put the integer $1$ in this cell.
\end{enumerate}

\begin{defn}
We denote by $\xi$ the modified jeu de taquin operation on $r$-Young tableaux.
\end{defn}

\begin{example}
We consider the action of $\xi$ on the second $2$-Young tableau in Figure \ref{fig:YT}.
\begin{align*}
\xi:\tikzpic{-0.5}{[scale=0.6]
\draw(0,0)--(4,0)--(4,-2)--(0,-2)--(0,0)(2,0)--(2,-2)(0,-1)--(4,-1)(1,-1)--(1,-2)
(3,-1)--(3,-2);
\draw(1,-0.5)node{$1$}(3,-0.5)node{$3$};
\draw(0.5,-1.5)node{$2$}(1.5,-1.5)node{$4$}(2.5,-1.5)node{$5$}(3.5,-1.5)node{$6$};
}
&\leadsto
\tikzpic{-0.5}{[scale=0.6]
\draw(0,0)--(4,0)--(4,-2)--(0,-2)--(0,0)(2,0)--(2,-2)(0,-1)--(4,-1)(1,-1)--(1,-2)
(3,-1)--(3,-2);
\draw(1,-0.5)node{$1$}(3,-0.5)node{$3$};
\draw(0.5,-1.5)node{$2$}(1.5,-1.5)node{$4$}(2.5,-1.5)node{$5$}(3.5,-1.5)node{};
}\leadsto
\tikzpic{-0.5}{[scale=0.6]
\draw(0,0)--(4,0)--(4,-2)--(0,-2)--(0,0)(2,0)--(2,-2)(0,-1)--(4,-1)(1,-1)--(1,-2)
(3,-1)--(3,-2);
\draw(1,-0.5)node{$1$}(3,-0.5)node{$3$};
\draw(0.5,-1.5)node{}(1.5,-1.5)node{$2$}(2.5,-1.5)node{$4$}(3.5,-1.5)node{$5$};
}\leadsto
\tikzpic{-0.5}{[scale=0.6]
\draw(0,0)--(4,0)--(4,-2)--(0,-2)--(0,0)(2,0)--(2,-2)(0,-1)--(4,-1)(1,-1)--(1,-2)
(3,-1)--(3,-2);
\draw(1,-0.5)node{}(3,-0.5)node{$3$};
\draw(0.5,-1.5)node{$1$}(1.5,-1.5)node{$2$}(2.5,-1.5)node{$4$}(3.5,-1.5)node{$5$};
}\\[12pt]
&\leadsto
\tikzpic{-0.5}{[scale=0.6]
\draw(0,0)--(4,0)--(4,-2)--(0,-2)--(0,0)(2,0)--(2,-2)(0,-1)--(4,-1)(1,-1)--(1,-2)
(3,-1)--(3,-2);
\draw(1,-0.5)node{}(3,-0.5)node{$4$};
\draw(0.5,-1.5)node{$2$}(1.5,-1.5)node{$3$}(2.5,-1.5)node{$5$}(3.5,-1.5)node{$6$};
}\leadsto
\tikzpic{-0.5}{[scale=0.6]
\draw(0,0)--(4,0)--(4,-2)--(0,-2)--(0,0)(2,0)--(2,-2)(0,-1)--(4,-1)(1,-1)--(1,-2)
(3,-1)--(3,-2);
\draw(1,-0.5)node{$1$}(3,-0.5)node{$4$};
\draw(0.5,-1.5)node{$2$}(1.5,-1.5)node{$3$}(2.5,-1.5)node{$5$}(3.5,-1.5)node{$6$};
}
\end{align*}
In terms of a word, we have $\xi(URURRR)=URRURR$.
\end{example}

\subsection{Chord diagrams}
\label{sec:chord}
A {\it chord diagram} of size $2n$ is a visualization of a Dyck paths of size $n$ 
by a set of $n$ arches.
This is the same notion as a non-crossing complete 
matchings on $2n$ (see item 61 in Section 2 of \cite{Sta15}).
A chord diagram consists of $2n$ labeled points and $n$ non-crossing arches which 
connect two labeled points.
We denote an arch connecting the point $i$ and $j$ by $(i,j)$ with $i<j$. 
A chord diagram is a set of arches $(i,j)$ and we obtain a Dyck path in 
terms of a word $\{U,R\}^{2n}$ by 
replacing $i$ (resp. $j$) by $U$ (resp. $R$) in an arch $(i,j)$.

For example, two Dyck paths of size $3$ correspond to the following 
chord diagrams:
\begin{align*}
URUURR \leftrightarrow
\tikzpic{-0.3}{[scale=0.8]
\draw(0,0)node[anchor=north]{$1$}to[out=90,in=90](1,0)node[anchor=north]{$2$};
\draw(2,0)node[anchor=north]{$3$}to[out=90,in=90](5,0)node[anchor=north]{$6$};
\draw(3,0)node[anchor=north]{$4$}to[out=90,in=90](4,0)node[anchor=north]{$5$};
} 
\qquad
UURURR \leftrightarrow
\tikzpic{-0.3}{[scale=0.8]
\draw(0,0)node[anchor=north]{$1$}to[out=90,in=90](5,0)node[anchor=north]{$6$};
\draw(1,0)node[anchor=north]{$2$}to[out=90,in=90](2,0)node[anchor=north]{$3$};
\draw(3,0)node[anchor=north]{$4$}to[out=90,in=90](4,0)node[anchor=north]{$5$};
} 
\end{align*}

We denote by $\mathcal{C}_{n}$ the set of chord diagram of size $2n$.

For later purpose, we introduce the set $\mathcal{C}_{n}^{(r)}\subseteq\mathcal{C}_{nr}$ of chord diagrams of size $2rn$
satisfying the following conditions:
\begin{enumerate}[({A}1)]
\item Let $(i,j)$, $1\le i<j\le 2rn$, be an arch in a chord diagram of size $2nr$. Then, we 
impose a condition on $i$ and $j$:
\begin{align}
\label{eq:condA}
i+j-1\equiv 0 \pmod{2r}.
\end{align}
\end{enumerate}
A chord diagram in $\mathcal{C}_{n}^{(r)}$ has $2nr$ points.
We bundle the $2nr$ points into groups of $r$ points from left to right. 
Then, we label $2n$ bundled points by $1,1', 2,2',\cdots,n,n'$ from left to right. 

The condition (\ref{eq:condA}) implies that an arch $(i,j)$ connects a bundled point 
with another bundled point such that one of the points is primed and the other 
is not primed.

\begin{example}
An example of a chord diagram in $\mathcal{C}_{3}^{(2)}$ is 
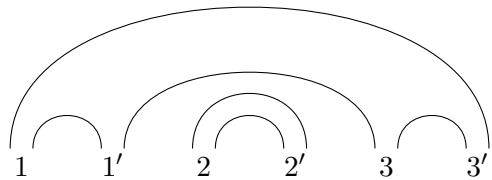
\begin{align*}
\tikzpic{-0.5}{[xscale=0.6]
\draw(0,0)..controls(0,2.5)and(10.5,2.5)..(10.5,0);
\draw(0.5,0)to[out=90,in=90](2,0);
\draw(2.5,0)..controls(2.5,1.35)and(8,1.35)..(8,0);
\draw(4,0)to[out=90,in=90](6.5,0);
\draw(4.5,0)to[out=90,in=90](6,0);
\draw(8.5,0)to[out=90,in=90](10,0);
\draw(0.25,-0.5)node[anchor=south]{$1$};
\draw(2.25,-0.5)node[anchor=south]{$1'$};
\draw(4.25,-0.5)node[anchor=south]{$2$};
\draw(6.25,-0.5)node[anchor=south]{$2'$};
\draw(8.25,-0.5)node[anchor=south]{$3$};
\draw(10.25,-0.5)node[anchor=south]{$3'$};
}
\end{align*}
The following chord diagram is not in $\mathcal{C}_{3}^{(2)}$ since it violates 
the condition (A1). There exist chords connecting $1$ and $2$, and $2'$ and $3'$.
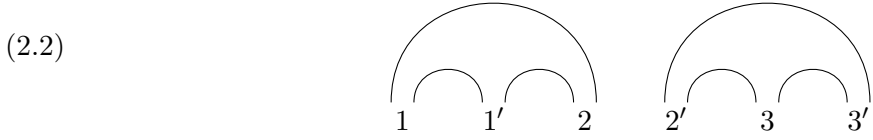
\begin{align}
\label{eq:dnadm}
\tikzpic{-0.5}{[xscale=0.6]
\draw(0,0)to[out=90,in=90](4.5,0);
\draw(0.5,0)to[out=90,in=90](2,0);
\draw(2.5,0)to[out=90,in=90](4,0);	
\draw(6,0)to[out=90,in=90](10.5,0);
\draw(6.5,0)to[out=90,in=90](8,0);
\draw(8.5,0)to[out=90,in=90](10,0);	
\draw(0.25,-.5)node[anchor=south]{$1$};
\draw(2.25,-.5)node[anchor=south]{$1'$};
\draw(4.25,-.5)node[anchor=south]{$2$};
\draw(6.25,-.5)node[anchor=south]{$2'$};
\draw(8.25,-.5)node[anchor=south]{$3$};
\draw(10.25,-.5)node[anchor=south]{$3'$};
}
\end{align}

\end{example}

We introduce another generalization of a chord diagram.
Fix a positive integer $r\ge1$.
A generalized chord diagram $\widetilde{C}$ of size $(r+1)n$ is a visualization 
of generalized Dyck path by $n$ arches.
The diagram $\widetilde{C}$ consists of $(r+1)n$ points and $n$ non-crossing arches which 
connect $r+1$ points.
We denote by $\widetilde{\mathcal{C}}_{n}^{(r)}$ the set of generalized chord diagram 
of size $(r+1)n$.

\begin{example}
Let $(n,r)=(2,2)$. We have three generalized chord diagrams in $\widetilde{\mathcal{C}}_{2}^{(2)}$:
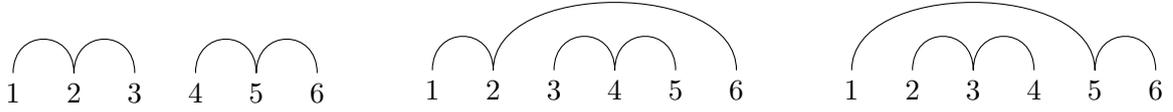
\begin{align*}
\tikzpic{-0.8}{
\draw(0,0)node[anchor=north]{$1$};
\draw(0.8,0)node[anchor=north]{$2$};
\draw(1.6,0)node[anchor=north]{$3$};
\draw(2.4,0)node[anchor=north]{$4$};
\draw(3.2,0)node[anchor=north]{$5$};
\draw(4,0)node[anchor=north]{$6$};
\draw(0,0)..controls(0,0.6)and(0.8,0.6)..(0.8,0)..controls(0.8,0.6)and(1.6,0.6)..(1.6,0);
\draw(2.4,0)..controls(2.4,0.6)and(3.2,0.6)..(3.2,0)..controls(3.2,0.6)and(4,0.6)..(4,0);
}\qquad
\tikzpic{-0.5}{
\draw(0,0)node[anchor=north]{$1$};
\draw(0.8,0)node[anchor=north]{$2$};
\draw(1.6,0)node[anchor=north]{$3$};
\draw(2.4,0)node[anchor=north]{$4$};
\draw(3.2,0)node[anchor=north]{$5$};
\draw(4,0)node[anchor=north]{$6$};
\draw(0,0)..controls(0,0.6)and(0.8,0.6)..(0.8,0)..controls(0.8,1.2)and(4,1.2)..(4,0);
\draw(1.6,0)..controls(1.6,0.6)and(2.4,0.6)..(2.4,0)..controls(2.4,0.6)and(3.2,0.6)..(3.2,0);
}\qquad
\tikzpic{-0.5}{
\draw(0,0)node[anchor=north]{$1$};
\draw(0.8,0)node[anchor=north]{$2$};
\draw(1.6,0)node[anchor=north]{$3$};
\draw(2.4,0)node[anchor=north]{$4$};
\draw(3.2,0)node[anchor=north]{$5$};
\draw(4,0)node[anchor=north]{$6$};
\draw(0,0)..controls(0,1.2)and(3.2,1.2)..(3.2,0)..controls(3.2,0.6)and(4,0.6)..(4,0);
\draw(0.8,0)..controls(0.8,0.6)and(1.6,0.6)..(1.6,0)..controls(1.6,0.6)and(2.4,0.6)..(2.4,0);
}
\end{align*}
\end{example}

\begin{prop}
The set $\widetilde{\mathcal{C}}_{n}^{(r)}$ is bijective to the set $\mathcal{P}_{n}^{(r)}$.
\end{prop}
\begin{proof}
We construct a bijection between the two sets.
Let $\widetilde{C}\in\widetilde{\mathcal{C}}_{n}^{(r)}$ be a generalized chord diagram
and $A_{i}$, $1\le i\le n$, be its $n$ non-crossing arches.
Let $a_{i}:=\min A_{i}$ be the minimum integer in the arch $A_{i}$.
We have a word $w:=w(\widetilde{C})$ in $\{U,R\}^{(r+1)n}$ such that 
the $a_{i}$-th letter in $w$ is $U$, and other letters are $R$.
We will show that $w\in\mathcal{P}_{n}^{(r)}$.
First, we have $a_1=1$ by definition of a generalized chord diagram.
Further, each $U$ letter has at least $r$ $R$ letters right to it in $w$.
This implies that $w$ is above $y=x/r$.
From these, $w$ is a generalized Dyck path.

Conversely, let $w$ be a word of $U$ and $R$ corresponding to a generalized Dyck path.
Let $a_{i}$ be the position of the letter $U$ in $w$. By definition of a generalized 
Dyck path, we have $a_1$=1, and $|\{a_{i}\}|=n$.
The integers $a_{i}$, $1\le i\le n$, are strictly increasing.
We construct a generalized chord diagram $\widetilde{C}:=\widetilde{C}(w)$ 
in $\widetilde{\mathcal{C}}_{n}^{(n)}$ from $\{a_{i}:1\le i\le n\}$.
 We recursively give an arch connecting $(r+1)$ points 
as follows:
\begin{enumerate}
\item Set $i=n$ and $S:=\{1,2,\ldots, (r+1)n\}$.
\item We take $r+1$ successive increasing integers, which starts from $a_{i}$, in $S$.
Let $A_{i}$ be the set of such $r+1$ integers.
\item Decrease $i$ by one, and replace $S$ by $S\setminus A_{i}$. Then, go to (2).
The algorithm stops when $i=1$.
\end{enumerate}
The sets $A_{i}$, $1\le i\le n$, contain $r+1$ integers, and these sets give 
$n$ arches consisting of $r+1$ integers.
Recall that a generalized Dyck path is above $y=x/r$. This property insures that we can always take 
$r+1$ integers as in (2).

It is easy to see that the above two maps are inverse of each other. This means that 
the map is a bijection between the two sets $\widetilde{\mathcal{C}}_{n}^{(r)}$ 
and $\mathcal{P}_{n}^{(r)}$. This completes the proof.
\end{proof}

\subsection{Rotation on chord diagrams}
\label{sec:rotchord}
We define a rotation $\sigma: \mathcal{C}_{n}\rightarrow\mathcal{C}_{n}$ 
as follows.
Let $C$ be a chord diagram in $\mathcal{C}_{n}$. 
We relabel the integers $i$ by $i'$ and $j'$ by $j+1$ modulo $n$, 
and move the points labeled $1$ to the left by an isotropic move.
Here, we keep the connectivity of the points, and the operation 
gives a unique chord diagram $\sigma(C)$. 

For example, we have 
\begin{align*}
\sigma: 
\tikzpic{-0.3}{[scale=0.8]
\draw(0,0)to[out=90,in=90](1,0);
\draw(2,0)to[out=90,in=90](5,0);
\draw(3,0)to[out=90,in=90](4,0);
\foreach \x/\y in {0/1,1/1',2/2,3/2',4/3,5/3'}
\draw(\x,-0.8)node[anchor=south]{$\y$};
}\leadsto
\tikzpic{-0.3}{[scale=0.8]
\draw(0,0)to[out=90,in=90](1,0);
\draw(2,0)to[out=90,in=90](5,0);
\draw(3,0)to[out=90,in=90](4,0);
\foreach \x/\y in {0/1',1/2,2/2',3/3,4/3',5/1}
\draw(\x,-0.8)node[anchor=south]{$\y$};
}\leadsto
\tikzpic{-0.3}{[scale=0.8]
\draw(-1,0)to[out=90,in=90](2,0);
\draw(0,0)to[out=90,in=90](1,0);
\draw(3,0)to[out=90,in=90](4,0);
\foreach \x/\y in {-1/1,0/1',1/2,2/2',3/3,4/3'}
\draw(\x,-0.8)node[anchor=south]{$\y$};
}
\end{align*}

We define the rotation on $\mathcal{C}_{n}^{(r)}$ in a similar manner.
We regard $C\in\mathcal{C}_{n}^{(r)}$ as a chord diagram of $2nr$ points.
Then, one can regard $\sigma^{r}$ on $\mathcal{C}_{nr}$ as a rotation 
on $\mathcal{C}_{n}^{(r)}$.
We denote the rotation on $\mathcal{C}_{n}^{(r)}$ by $\sigma^{(r)}:=\sigma^{r}$.

We can also define a rotation 
$\widetilde{\sigma}:\widetilde{\mathcal{C}}_{n}^{(r)}\rightarrow\widetilde{\mathcal{C}}_{n}^{(r)}$ 
on a generalized chord diagram $\widetilde{C}$ in $\widetilde{\mathcal{C}}_{n}^{(r)}$.
In this case, the connectivity of points is preserved under the rotation.
The diagram $\widetilde{C}$ has $(r+1)n$ points and $n$ arches.
We relabel the integers $i$ by $i+1$ modulo $(r+1)n$, and move the point labeled $1$ 
to the left by keeping its connectivity.
For example, we have 
\begin{align*}
\widetilde{\sigma}: \tikzpic{-0.8}{
\draw(0,0)node[anchor=north]{$1$};
\draw(0.8,0)node[anchor=north]{$2$};
\draw(1.6,0)node[anchor=north]{$3$};
\draw(2.4,0)node[anchor=north]{$4$};
\draw(3.2,0)node[anchor=north]{$5$};
\draw(4,0)node[anchor=north]{$6$};
\draw(0,0)..controls(0,0.6)and(0.8,0.6)..(0.8,0)..controls(0.8,0.6)and(1.6,0.6)..(1.6,0);
\draw(2.4,0)..controls(2.4,0.6)and(3.2,0.6)..(3.2,0)..controls(3.2,0.6)and(4,0.6)..(4,0);
}\leadsto
\tikzpic{-0.8}{
\draw(0,0)node[anchor=north]{$2$};
\draw(0.8,0)node[anchor=north]{$3$};
\draw(1.6,0)node[anchor=north]{$4$};
\draw(2.4,0)node[anchor=north]{$5$};
\draw(3.2,0)node[anchor=north]{$6$};
\draw(4,0)node[anchor=north]{$1$};
\draw(0,0)..controls(0,0.6)and(0.8,0.6)..(0.8,0)..controls(0.8,0.6)and(1.6,0.6)..(1.6,0);
\draw(2.4,0)..controls(2.4,0.6)and(3.2,0.6)..(3.2,0)..controls(3.2,0.6)and(4,0.6)..(4,0);
}\leadsto
\tikzpic{-0.5}{
\draw(0,0)node[anchor=north]{$1$};
\draw(0.8,0)node[anchor=north]{$2$};
\draw(1.6,0)node[anchor=north]{$3$};
\draw(2.4,0)node[anchor=north]{$4$};
\draw(3.2,0)node[anchor=north]{$5$};
\draw(4,0)node[anchor=north]{$6$};
\draw(0,0)..controls(0,1.2)and(3.2,1.2)..(3.2,0)..controls(3.2,0.6)and(4,0.6)..(4,0);
\draw(0.8,0)..controls(0.8,0.6)and(1.6,0.6)..(1.6,0)..controls(1.6,0.6)and(2.4,0.6)..(2.4,0);
}
\end{align*}

Let $P\in\mathcal{P}_{n}^{(r)}$, and $\widetilde{C}:=\widetilde{C}(P)$ be a generalized chord 
diagram corresponding to $P$ in $\widetilde{\mathcal{C}}_{n}^{(r)}$.
The next proposition shows that the rotation on $P$ is equivalent to the rotation of $\widetilde{C}$. 
\begin{prop}
\label{prop:xisigma}
We have $\xi(P)=\widetilde{\sigma}(\widetilde{C})$.
\end{prop}
\begin{proof}
Recall that an $r$-Dyck path $P$ has its representation by an $r$-Young tableau $Y$. 
The map $\xi$ moves the integers right or down. 
On the other hand, the action of 
$\widetilde{\sigma}$ on $\widetilde{C}$ implies that the arches are moved rightward 
by one unit modulo $(r+1)n$.
The left-most point in an arch corresponds to the integers in the first row of $Y$.
If an arch does not contain the right-most point, which has a label $(r+1)n$, the map $\xi$
also preserves this arch by increasing the integers by one.
If an arch contains the right-most point, and is rotated by $\widetilde{\sigma}$, 
the rotated arch contains the left-most point. 
This corresponds to the addition of $1$ in the first row of $Y$.
From these, we have $\xi(P)=\widetilde{\sigma}(\widetilde{C})$.
\end{proof}

\subsection{Temperley--Lieb algebra on chord diagrams}
\label{sec:TLchord}
The {\it Temperley--Lieb algebra} $\mathbb{TL}_{n}$ is 
the unital associative $\mathbb{C}[q,q^{-1}]$-algebra 
generated by the set $\{e_1,\ldots,e_{n-1}\}$ satisfying the 
following relations \cite{TL71}:
\begin{align}
\label{eq:ei}
\begin{split}
&e_i^{2}=\tau e_i, \qquad 1\le i\le n-1, \\
&e_{i}e_{i+1}e_{i}=e_i, \qquad 1\le i\le n-2,\\
&e_{i+1}e_{i}e_{i+1}=e_{i+1}, \qquad 1\le i\le n-2,\\
&e_{i}e_{j}=e_{j}e_{i}, \qquad |i-j|>1,
\end{split}
\end{align}
where $\tau:=-(q+q^{-1})$.
The algebra $\mathbb{TL}_{n}$ has a diagrammatic representation \cite{Kau87}.
The generator $e_{i}$ is depicted as 
\begin{align*}
e_{i}=\tikzpic{-0.5}{[xscale=0.8]
\draw(0,0)node[anchor=north]{$1$} to (0,1);
\draw(2,0)node[anchor=north]{$i-1$}to (2,1);
\draw(1,0.5)node{$\cdots$};
\draw(3.2,0)node[anchor=north]{$i$}..controls(3.2,0.5)and(4.2,0.5)..(4.2,0)node[anchor=north]{$i+1$};
\draw(3.2,1)..controls(3.2,0.5)and(4.2,0.5)..(4.2,1);
\draw(5.4,0)node[anchor=north]{$i+2$}--(5.4,1);
\draw(7.4,0)node[anchor=north]{$n$}--(7.4,1);
\draw(6.4,0.5)node{$\cdots$};
}
\end{align*}
The unit $\mathbf{1}$ is depicted as the diagram consisting of 
$n$ vertical strands without cap-cup pairs.
The product $XY$ of two elements $X,Y\in\mathbb{TL}_{n}$ is calculated by
placing the diagram of $Y$ on top of the diagram of $X$. 
We regard two diagrams are equivalent if they are isotropic to each other.
If we have a closed loop, we remove it and give a factor $\tau$.

The action of the generator $e_{i}$, $1\le i\le n-1$, on 
a chord diagram $C$ in $\mathcal{C}_{n}$ is given by placing 
the diagram of $C$ on top of the diagram of $e_{i}$.
If closed loops appear in the diagram, we remove each of 
which and give a factor $\tau$.

\begin{example}
The action of $e_{2}$ on a Dyck path $URUURR$ gives 
the new Dyck path $UURURR$.
Diagrammatically, we have
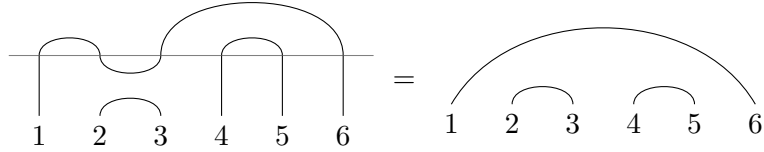
\begin{align*}
\tikzpic{-0.4}{[scale=0.8]
\draw(0,0)to[out=90,in=90](1,0);
\draw(2,0)to[out=90,in=90](5,0);
\draw(3,0)to[out=90,in=90](4,0);
\draw[gray](-0.5,0)--(5.5,0);
\draw(0,0)--(0,-1)(1,0)to[out=-90,in=-90](2,0)
(1,-1)to[out=90,in=90](2,-1)
(3,0)to(3,-1)(4,0)to(4,-1)(5,0)to(5,-1);
\draw(0,-1)node[anchor=north]{$1$}(1,-1)node[anchor=north]{$2$}
(2,-1)node[anchor=north]{$3$}(3,-1)node[anchor=north]{$4$}
(4,-1)node[anchor=north]{$5$}(5,-1)node[anchor=north]{$6$};
}
=
\tikzpic{-0.4}{[scale=0.8]
\draw(0,0)node[anchor=north]{$1$}to[out=60,in=120](5,0)node[anchor=north]{$6$};
\draw(1,0)node[anchor=north]{$2$}to[out=90,in=90](2,0)node[anchor=north]{$3$};
\draw(3,0)node[anchor=north]{$4$}to[out=90,in=90](4,0)node[anchor=north]{$5$};
}
\end{align*}
Therefore, we have $e_{2}(URU^2R^2)=U^2RUR^2$.
\end{example}

\section{Non-crossing partitions}
\label{sec:NCP}
\subsection{Definition}
A {\it non-crossing partition} of the set $[n]:=\{1,2,\ldots,n\}$ 
is a partition $\pi$ of $[n]$ such that 
if four integers satisfy $i<j<k<l$, a block $B_1$ contains $i$ and $k$
and another block $B_2$ contains $j$ and $l$, then two blocks $B_1$ 
and $B_2$ coincide with each other \cite{Kre72,Pou72}.

We denote by $\mathcal{NC}_{n}$ the set of non-crossing partitions 
of the set $[n]$.
It is well-known that the number of non-crossing partitions in $\mathcal{NC}_{n}$
is given by the $n$-th Catalan number $C_{n}:=\genfrac{}{}{}{}{1}{n+1}\genfrac{(}{)}{0pt}{}{2n}{n}$.
Recall that the number of Dyck paths of size $n$ is also given by the $n$-th Catalan number.
We will see a bijection between a non-crossing partition in $\mathcal{NC}_{n}$ 
and a Dyck path in $\mathcal{P}^{(1)}_{n}$.

When a non-crossing partition $\pi$ has $m$ blocks, we write 
$\pi=\pi_1/\pi_2/\ldots/\pi_{m}$ 
where each $\pi_{i}$, $1\le i\le m$, is an increasing integer 
sequence.
Since we consider the partition of $[n]$, the order of $\pi_1,\ldots,\pi_{m}$
is not relevant.

\begin{example}
The non-crossing partition $\pi=134/2/56$ is a non-crossing partition of $[6]$ 
such that $\pi$ consists of three blocks $\{1,3,4\}$, $\{2\}$ and $\{5,6\}$.

The partitions $14/23$ and $12/34$ is in $\mathcal{NC}_{4}$. However,
the partition $13/24$ is not in $\mathcal{NC}_{4}$ since it is crossing.
\end{example}

Suppose $\pi\in\mathcal{NC}_{n}$ has $m$ blocks. 
Then, we define the rank $\mathtt{rk}(\pi)$ of $\pi$ as 
$\mathtt{rk}(\pi):=n-m$.

Let $\pi,\nu\in\mathcal{NC}_{n}$ be two non-crossing partitions.
We say that $\nu$ covers $\pi$ if and only if 
$\mathtt{rk}(\nu)=\mathtt{rk}(\pi)+1$ and there exists a unique 
block $B$ in $\nu$ such that $B=B_{i}\cup B_{j}$ where 
$B_i$ and $B_{j}$ are two blocks in $\pi$.
The other blocks in $\nu$ and $\pi$ coincide with each other.

When $\nu$ covers $\pi$, we write $\pi\lessdot\nu$.
We write $\pi\le\nu$ if there exists a positive integer $r$ 
such that $\pi=x_1\lessdot x_2\lessdot \ldots \lessdot x_{r}=\nu$.
In other words, $\pi\le \nu$ if $\pi$ is a refinement of $\nu$, and 
equivalently $\nu$ is a coarsement of $\pi$.

The Hasse diagram of non-crossing partitions in $\mathcal{NC}_3$ 
is depicted in Figure \ref{fig:NC3}.
Two elements $x,y\in\mathcal{NC}_3$ are connected by an edge if and only if 
$x\lessdot y$ in the Hasse diagram. 
\begin{figure}[ht]
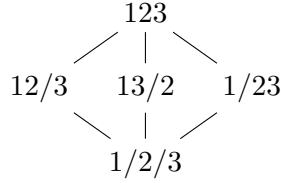

\tikzpic{-0.5}{[xscale=1.4]
\node(0) at (0,0){$123$};
\node(1) at (-1,-1){$12/3$};
\node(2) at (0,-1){$13/2$};
\node(3) at (1,-1){$1/23$};
\node(4) at (0,-2){$1/2/3$};
\draw(0)to(1)(0)to(2)(0)to(3)(1)to(4)(2)to(4)(3)to(4);
 }
\caption{The Hasse diagram of non-crossing partitions in $\mathcal{NC}_{3}$}
\label{fig:NC3}
\end{figure}

\subsection{Kreweras endomorphism}
\label{sec:Kreendo}
To define a rotation on non-crossing partitions, we introduce 
the Kreweras endomorphism on $\mathcal{NC}_{n}$ following \cite{Kre72} (see also \cite{SimUll91}).

Let $\pi\in\mathcal{NC}_{n}$ be a non-crossing partition 
consisting of $m$ blocks $B_{i}$, $1\le i\le m$.
We introduce a pictorial representation of $\pi$ as follows.
Let $S$ be a circle with $n$ points. We label these $n$ points 
from $1$ to $n$ clockwise.
Suppose that the block is $B_{i}=n_1n_2\ldots n_{r}$. 
Then, we connect $r$ points on $S$ by arches. 
Since $\pi$ is non-crossing, the arches in the circle $S$ are also 
non-crossing. 
When the size of the block $B_{i}$ is one, {\it i.e.}, $B_{i}$ consists 
of a single integer, we do not add an arch on $S$.

We append new $n$ points on $S$ by dividing the interval between two points 
labeled $i$ and $i+1$ for $1\le i\le n-1$ or $n$ and $1$.
We put a dashed label $1'$ on the point between $n$ and $1$, and 
put a dashed label $i'$ on the point between $i-1$ and $i$ clockwise.
Since $\pi$ is non-crossing which implies that arches are non-crossing,
arches divide the inside of $S$ into several regions.
We connect the points labeled dashed integers by dotted arches 
if they belong to the same region.
Then, if we focus on the dashed points and dotted arches, we obtain 
another non-crossing partition $\pi'$ from $\pi$.

\begin{defn}
The endomorphism $\rho:\mathcal{NC}_{n}\rightarrow \mathcal{NC}_{n}$, $\pi\mapsto\pi'$, 
is called the Kreweras endomorphism.
\end{defn}
 
\begin{example}
Let $\pi=136/2/4/5/78$. Then, we have $\rho(\pi)=17/23/456/8$.
In Figure \ref{fig:Kre}, we depict $\pi$ by solid lines and 
$\rho(\pi)$ by dotted lines.

\begin{figure}[ht]
\begin{tikzpicture}[scale=0.8]
\draw circle(3cm);
\foreach \a in {0,45,90,135,...,360}
\filldraw [black] (\a:3cm)circle(1.5pt);
\draw (112.5:3cm) node[anchor=south east] {$1'$};
\draw (67.5:3cm) node[anchor=south west]{$2'$};
\draw (22.5:3cm) node[anchor=south west]{$3'$};
\draw (-22.5:3cm) node[anchor=north west]{$4'$};
\draw (-67.5:3cm) node[anchor=north west]{$5'$};
\draw (-112.5:3cm) node[anchor=north east]{$6'$};
\draw (-157.5:3cm) node[anchor=north east]{$7'$};
\draw (157.5:3cm) node[anchor=south east]{$8'$};
\draw (90:3cm) node[anchor=south]{$1$};
\draw (45:3cm) node[anchor=south west]{$2$};
\draw (0:3cm)node[anchor=west]{$3$};
\draw (-45:3cm)node[anchor=north west]{$4$};
\draw (-90:3cm)node[anchor=north]{$5$};
\draw (-135:3cm)node[anchor=north east]{$6$};
\draw (180:3cm)node[anchor=east]{$7$};
\draw (135:3cm)node[anchor=south east]{$8$};
\draw(90:3cm)to[bend right=30](0:3cm)to[bend right=15](-135:3cm)to[bend right=15](90:3cm);
\draw (180:3cm) to [bend right=60] (135:3cm);
\draw[dashed] (67.5:3cm) to[bend right=40] (22.5:3cm);
\draw[dashed] (-22.5:3cm)to[bend right=40](-67.5:3cm)to[bend right=40](-112.5:3cm)to[bend left=40](-22.5:3cm);
\draw[dashed] (-157.5:3cm) to [bend right=40] (112.5:3cm);
\end{tikzpicture}

\caption{An example of the Kreweras endomorphism}
\label{fig:Kre}
\end{figure}
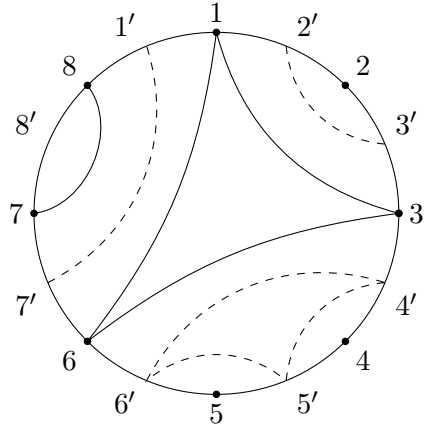
\end{example}

We summarize the properties of the map $\rho$.
\begin{prop}
\label{prop:Kre}
The Kreweras endomorphism $\rho$ satisfies 
\begin{enumerate}
\item $\rho^{2n}$ is the identity.
\item $\rho^2$ is a rotation on $\mathcal{NC}_{n}$. In other words, $\pi'=\rho^{2}(\pi)$ 
is obtained from $\pi$ by replacing $i$ by $i+1$ for $2\le i\le n$ and $n$ by $1$.
\item The rank function $\mathtt{rk}$ satisfies $\mathtt{rk}(\rho(\pi))+\mathtt{rk}(\pi)=n-1$. 
\item If $\pi\lessdot\nu$, then $\rho(\nu)\lessdot\rho(\pi)$.
\end{enumerate}
\end{prop}
\begin{proof}
(1) and (2) are obvious from the definition of $\rho$.

We show (3) by induction on $n$. For $n=1,2$, (3) holds.
Suppose that $\pi$ contains the block consisting of 
$(k_1,k_2,\ldots,k_{p})$ where $k_1=1$ and $p\ge1$.
In the graphical presentation, we have $p$ diagonals 
which connect $k_{i}$ and $k_{i+1}$ for $1\le i\le p-1$, and $k_{p}$ and $k_1$.
Suppose that $b_{i}$ is the number of blocks between $k_{i}$ and $k_{i+1}$ for 
$1\le i\le p-1$ and $b_{p}$ is the number of blocks between $k_p$ and $k_1$.
The number of blocks in $\pi$ is $1+\sum_{i=1}^{p}b_{i}$.
By definition of rank, the rank of $\pi$ is given by $\mathtt{rk}(\pi)=n-1-\sum_{i=1}^{p}b_{i}$.
We consider the action of $\rho$ on $\pi$.
In the region between $[k_i,k_{i+1}]$, we have $k_{i+1}-k_{i}-b_{i}$ blocks 
in $\rho(\pi)$ by induction hypothesis.
Therefore, $\rho(\pi)$ has $\sum_{i=1}^{p}(k_{i+1}-k_{i}-b_{i})=n-\sum_{i=1}^{p}b_{i}$ blocks.
The rank of $\rho(\pi)$ is given by $\mathtt{rk}(\rho(\pi))=\sum_{i=1}^{p}b_{i}$.
From these, we have $\mathtt{rk}(\rho(\pi))+\mathtt{rk}(\pi)=n-1$.

(4) Suppose that $\pi\lessdot\nu$. It is enough to prove that $\rho(\nu)\lessdot\rho(\pi)$.
The cover relation implies that there are two blocks $B_1$ and $B_2$ in $\pi$ such that $B_1$ and $B_2$ 
are distinct blocks in $\pi$ and $B':=B_{1}\cup B_{2}$ is a block in $\nu$.
Without loss of generality, we assume that $\min B_1<\min B_2$.
We have two cases: a) $\max B_1<\min B_2$, and b) $\max B_1>\max B_2$.
\paragraph{Case a)}
Since $\pi$ is non-crossing, there is no diagonal $(i,j)$ in the pictorial representation of $\nu$ 
such that $\max B_1<i<\min B_2<\max B_2<j$.
We consider the circular presentation of $\pi$ and $\nu$.
Let $X(\pi)$ be the set of primed points between $\max B_2$ and $\min B_1$, and $Y(\pi)$ the set of 
primed points between $\max B_1$ and $\min B_{2}$.
Since $B'$ is a single block in $\nu$, there is no diagonal connecting the primed points in $X(\pi)$
and $Y(\pi)$ in $\nu$. In $\pi$, since the points $X(\pi)$ and $Y(\pi)$ belong to the same region, we have 
several diagonals connecting the primed points in $X(\pi)$ and $Y(\pi)$.
Note that we have at least one such diagonal. 
Let $B'$ be the block which consists of the primed points $X(\pi)$ and $Y(\pi)$ in $\pi$.
By the Kreweras endomorphism, the block $B'$ in $\pi$ is divided into two smaller blocks $B'_1$ and 
$B'_{2}$ such that $B'_{1}$ (resp. $B'_2$) consists of primed points in $X(\pi)$ (resp. $Y(\pi)$).
Since $\pi\lessdot\nu$, the other blocks in $\rho(\pi)$ and $\rho(\nu)$ coincide with each other.
As a result, the blocks of $\rho(\pi)$ are obtained from $\rho(\nu)$ by merging the two blocks 
$B'_1$ and $B'_2$ in $\nu$ into a larger block $B'$.
These observations imply that $\rho(\nu)\lessdot\rho(\pi)$.

\paragraph{Case b)} 
Let $i_0$ (resp. $j_0$) be the largest (resp. smallest) integer in $B_1$ which is smaller (resp. larger)
than $\min B_2$ (resp. $\max B_2$).
In $\pi$, the two primed integers $(i_{0}+1)'$ and $j'_0$ belong to the same block.
On the other hand, in $\nu$, these two primed integers belong to distinct blocks since 
we merge the blocks $B_1$ and $B_2$ into a larger block.
This means that two primed blocks containing $(i_0+1)'$ and $j'_0$ are merged into 
a larger block. These imply that $\rho(\nu)\lessdot\rho(\pi)$.

From these considerations, we have $\rho(\nu)\lessdot\rho(\pi)$ if $\pi\lessdot\nu$.
\end{proof}

Let $\pi\in\mathcal{NC}_{n}$ and $\pi':=\rho(\pi)$.
The non-crossing partition $\pi'$ is obtained from $\pi$ by the following operations.
Suppose that $\pi$ consists of $m$ blocks, {\it i.e.}, $\pi=B_1/B_2/\ldots /B_{m}$ 
such that $\min B_i<\min B_{i+1}$ for $1\le i\le m-1$.
By definition $\min B_1=1$. 
Suppose that $j\in B_{i}$ for some $i$ such that $1\le i\le m$, and the integer $j-1$ (modulo $n$) belongs to 
a block $B_{k}$. 
If $k=i$, then we define an integer $t:=j$. 
Otherwise, $t$ is defined as follows.
Let $B_{k}$ be the block consisting of integers $\{b_1,b_2,\ldots,b_s\}$.
If $j-1=b_{r}$ with $1\le r\le s-1$, then we define $t:=b_{r+1}$.
If $j-1=b_{s}$, then we define $t:=b_{1}$.
Further, if $s=1$, then we define $t:=b_1$.

\begin{prop}
\label{prop:jt}
Let $j$ and $t$ be two integers as above. 
Then, $j$ and $t$ belong to the same block in $\pi'$.
\end{prop}
\begin{proof}
Consider the pictorial representation of $\pi$ on the circle $S$ with $n$ points. 
Suppose that $k=i$. In this case, the integers $j$ and $j-1$ belong to the same block 
$B_{i}$. By the action of $\rho$, the primed integer $j'$ forms a block by itself.
Since $t=j$, we are done in this case.
Suppose that $k\neq i$.
Then, by definition of $t$, the integers $j-1$ and $t$ belong to the same block $B_{i}$,
and there is no integer $u$ such that $j-1<u<t$ and $u\in B_{i}$. 
Then, it is obvious that the primed integers $j'$ and $t'$ belong to
the same region in $S$. 
This means that $j$ and $t$ belong to the same block in $\pi'$.
\end{proof}

If $j=t$ in $(j,t)$, then a block containing $j$ consists of only $j$. The size of 
this block is one.

For example, take $\pi=12/35/4/69/78$. Then, the set $I(\pi)$ of all pairs $(j,t)$ for $\pi$
is given by $I(\pi)=\{(1,6),(2,2),(3,1),(4,5),(5,4),(6,3),(7,9),(8,8),(9,7)\}$. 
This implies $\pi'=136/2/45/79/8$ since $j$ and $t$ in $(j,t)\in I(\pi)$ belong to 
the same block in $\rho(\pi)$ by Proposition \ref{prop:jt}.
In fact, this $\pi'$ is nothing but the non-crossing partition obtained by the pictorial operation, i.e., 
the Kreweras endomorphism.

\subsection{Temperley--Lieb algebra on non-crossing partitions}
\label{sec:TLNC}
We define a map $f_{i}$, $1\le i\le n-1$ on $\mathcal{NC}_{n}$.
Let $\pi\in\mathcal{NC}_{n}$.
Let $B_{i}$ and $B_{i+1}$ be the blocks in $\pi$ such that 
the integer $i$ (resp. $i+1$) belongs to $B_i$ (resp. $B_{i+1}$).
Then, the action of $f_i$ on $\pi$ is given by
\begin{align}
\label{def:f}
f_{i}\pi:=\begin{cases}
\tau\pi, & \text{ if } B_{i}=B_{i+1}, \\
\pi', & \text{ if } B_{i}\neq B_{i+1},
\end{cases}
\end{align}
where $\tau=-(q+q^{-1})$ and  $\pi'$ is a non-crossing partition obtained from $\pi$ by merging the 
two blocks $B_i$ and $B_{i+1}$ into a larger block.

For example, let $\pi=13/2/456/78$. Then, we have $f_{1}\pi=f_2\pi=123/456/78$, 
$f_{4}\pi=f_{5}\pi=f_7\pi=\tau \pi$ and $f_{6}\pi=13/2/45678$.

We have $n-1$ operators $f_{i}$, $1\le i\le n-1$. 
In Section \ref{sec:TLchord}, we define the action of Temperley--Lieb algebra 
$\mathbb{TL}_{2n}$ on Dyck paths of size $n$.
The algebra $\mathbb{TL}_{2n}$ has $2n-1$ generators.
To define the action of $\mathbb{TL}_{2n}$ on $\mathcal{NC}_{n}$, we need to 
define $2n-1$ operators which act on a non-crossing partition.
This is achieved by making use of the generators $f_{i}$ and the 
Kreweras endomorphism $\rho$.
Recall that $\rho$ satisfies  $\rho^{2n}=1$ by Proposition \ref{prop:Kre}.

We define $2n-1$ generators $\{F_{i}: 1\le i\le 2n-1\}$ which act on $\mathcal{NC}_{n}$ by 
\begin{align}
\label{def:Fi}
F_{i+1}=\rho F_{i} \rho^{-1},
\end{align}
where $\rho$ is the Kreweras endomorphism and $F_1=f_1$.

\begin{remark}
Two remarks are in order.
\begin{enumerate}
\item
Since $\rho$ is the Kreweras endomorphism, we have $f_2\neq F_2$. 
However, $\rho^2$ is the rotation on $\mathcal{NC}_{n}$ by Proposition \ref{prop:Kre}, we have 
$F_3=f_{2}$ on $\mathcal{NC}_{n}$.
\item 
By defining $F_{2n}=\rho F_{2n-1}\rho^{-1}$, we have an affine 
Temperley--Lieb algebra. 
This is well-defined since we have $\rho^{2n}=1$.
The set of generators $\{F_1,\ldots, F_{2n}\}$ satisfies 
the relations $F_{i}F_{i\pm1}F_{i}=F_{i}$ and $F_{i}F_{j}=0$ if $|i-j|\ge2$,
where we set $F_{2n+1}=F_{1}$.
Thus, we have the affine Termperley--Lieb algebra of type $A$ acting 
on non-crossing partitions.
\end{enumerate}
\end{remark}

One of the main results in this paper is the following theorem.
\begin{theorem}
\label{thrm:TLF}
The set of generators $\{F_i: 1\le i\le 2n-1 \}$ generates the 
Temperley--Lieb algebra $\mathbb{TL}_{2n}$ on $\mathcal{NC}_{n}$. 
\end{theorem}

\begin{proof}
Since we have $F_1=f_1$, $F_1^2=\tau F_1$ is equivalent to 
$f_1^2=\tau f_{1}$. 
However, it is obvious that the action of $f_1$ in Eq. (\ref{def:f}) implies that 
$f_1^2=\tau f_1$, i.e., $F_{1}^{2}=\tau F_{1}$.

We show that $F_1F_2F_1=F_1$.
Let $\pi\in\mathcal{NC}_{n}$. We consider the two cases: 1) the two integer $1$ and $2$ 
belong to the same block in $\pi$, and 2) otherwise.

Case 1) Let $B_{1\cup 2}$ be a block of $\pi$ where $1$ and $2$ belong to.
Since $1,2\in B_{1\cup2}$, we have $F_{1}\pi=\tau\pi$.
We consider the action of $F_2$ on $\pi$. Since $F_{2}=\rho F_1\rho^{-1}$,
we consider the action of $F_1$ on $\pi':=\rho^{-1}\pi$.
By definition of $\rho$, the two integers $1$ and $2$ belong to different blocks in $\pi'$.
The action of $F_1$ on $\pi'$ yields a non-crossing partition $\pi''$ where the two integers $1$ and 
$2$ belong to the same block.
Then, the action of $\rho$ on $\pi''$ gives a non-crossing partition such that  
$B_{1\cup 2}$ is divided into two blocks $\{2\}$ and $B_{1\cup2}\setminus\{2\}$, and all other 
blocks are the same as those of $\pi$.
Finally, the action of $F_1$ on $\rho\pi''$ gives $\pi$. 
As a summary, we have $F_1F_2F_1\pi=\tau\pi=F_1\pi$.

Case 2) Let $B_{1}$ (resp. $B_{2}$) be a block in $\pi$ such that the integer $1$ (resp. $2$)
belongs to $B_{1}$ (resp. $B_2$). We have $B_1\neq B_{2}$.
The action of $F_1$ merges $B_1$ and $B_2$ into a larger new block $B_{1\cup2}:=B_{1}\cup B_{2}$.
By the action of $\rho^{-1}$ on $B_{1\cup2}$, we have at least two blocks $B'_1$ and $B'_{2}$ where 
$B'_{i}$ contain $i$ in $\rho^{-1}F_{1}\pi$.  
Further action of $F_1$ merges the two blocks $B'_1$ and $B'_2$ into a larger block, and other blocks
remain the same.
As a consequence, $F_2F_1\pi$ has two blocks $\{2\}$ and $B_{1\cup2}\setminus\{2\}$, and other blocks 
are the same as those of $\pi$. Then, $F_{1}F_{2}F_{1}\pi$ has a large block $B_{1\cup2}$.
From these, we have $F_{1}F_2F_1\pi=F_1\pi$.

By combining the cases 1) and 2), we have $F_1F_2F_1=F_1$ on $\mathcal{NC}_{n}$.
One can also show in a similar manner that $F_2F_1F_2=F_2$. 

We show that $F_{i}F_{j}=F_{j}F_{i}$ for $|i-j|\ge2$.
By definition, we have $F_{j}=\rho^{j-i}F_{i}\rho^{-(j-i)}$.
The $F_i$ merges the two blocks $B_{i}$ and $B_{i+1}$ where $B_i$ contains the integer $i$.
Since $|i-j|\ge2$, it is obvious that the actions of $F_{i}$ and $F_{j}$ commute with each other.
Thus, we have $F_{i}F_{j}=F_{j}F_{i}$.

By combining the above observations together, the set $\{F_1,\ldots,F_{2n-1}\}$ generates 
the Temperley--Lieb algebra $\mathbb{TL}_{2n}$ on acting on $\mathcal{NC}_{n}$.
\end{proof}

In what follows, 
we will give a bijection $\Psi$ between a non-crossing partition in $\mathcal{NC}_{n}$ 
and a chord diagram in $\mathcal{C}_{n}$.
The bijection $\Psi$ plays a central role when we study the relations between the various Fuss--Catalan 
algebras on non-crossing partitions and those on generalized chord diagrams in the later section.

We will construct a chord diagram $C(\pi)\in\mathcal{C}_{n}$ from 
a non-crossing partitions $\pi\in\mathcal{NC}_{n}$.
Recall that a chord diagram consists of $n$ arches. To fix the positions 
of arches, we need to extract $n$ directed edges from $\pi$.
Suppose that $\pi$ consists of $m$ blocks $B_{i}$, $1\le i\le m$.
Since each block $B_{i}$ is an increasing sequence, 
we denote it by $B_{i}=(b_1,\ldots,b_{p})$ with some $p\ge1$.  
Let 
\begin{align*}
\mathcal{E}(B_{i}):=\{(b_i,b_{i+1}) :1\le i\le p-1 \}\cup\{(b_{p},b_{1})\},
\end{align*}
be the set of pairs of integers. When $p=1$, $\mathcal{E}(B_{i})=\{(b_1,b_1)\}$ 
by definition.
We define the set of pairs of integers by
\begin{align}
\label{eq:calE}
\mathcal{E}(\pi):=\bigcup_{1\le i\le m}\mathcal{E}(B_{i}).
\end{align}
Note that we have $|\mathcal{E}(\pi)|=n$, which means that we have $n$ elements 
which will be identified with $n$ arches in $C(\pi)$.
We associate an arch in $C(\pi)$  to an element in $\mathcal{E}(\pi)$ as follows.
Recall that we have $2n$ points labeled $1,1',2,2',\ldots,n,n'$ in $C(\pi)$.
Let $(i,j)\in\mathcal{E}(\pi)$ be a pair of integers. 
We connect the integer $i$ and primed integer $(j-1)'$ modulo $n$ by an arch. Here, we consider 
the label modulo $n$.

For example, we consider $\pi=12/3/4$. Then, we have 
\begin{align*}
\mathcal{E}(\pi)=\{(1,2),(2,1),(3,3),(4,4)\}.
\end{align*}
In the corresponding chord diagram, we have four arches 
connecting $1$ and $1'$, $2$ and $4'$, $3$ and $2'$, and $4$ and $3'$.
As a result, we have 
\begin{align*}
12/3/4 \qquad\leftrightarrow\qquad 
\tikzpic{-0.5}{[scale=0.5]
\draw(0,-1)node[anchor=south]{$1$}(0,0)..controls(0,1)and(1,1)..(1,0)(1,-1)node[anchor=south]{$1'$};
\draw(2,-1)node[anchor=south]{$2$}(2,0)..controls(2,2.5)and(7,2.5)..(7,0)(7,-1)node[anchor=south]{$4'$};
\draw(3,-1)node[anchor=south]{$2'$}(3,0)..controls(3,1)and(4,1)..(4,0)(4,-1)node[anchor=south]{$3$};
\draw(5,-1)node[anchor=south]{$3'$}(5,0)..controls(5,1)and(6,1)..(6,0)(6,-1)node[anchor=south]{$4$};
}
\qquad\leftrightarrow\qquad URUURURR.
\end{align*}

It is obvious that the construction of $C(\pi)$ from $\pi$ is invertible.
Therefore, we briefly explain the construction of a non-crossing partition $\pi(C)$ 
from $C$.
An arch connecting $i$ and $j'$ implies that the integers $i$ and $j+1$ are 
in the same block, and there is no integer $k$ such that $i<k<j+1$.
Here, we consider the integers modulo $n$.
We have $n$ arches in $C(\pi)$ and these uniquely determine the elements of
a block in $\pi$.
In this way, we have a unique non-crossing partition $\pi(C)$ from a chord diagram $C$.

\begin{defn}
\label{def:Psi}
We define the bijection $\Psi$ from $\mathcal{NC}_{n}$ to $\mathcal{C}_{n}$ given 
above.
\end{defn}

\begin{prop}
\label{prop:Fe}
The bijection $\Psi$ is compatible with the action of $\mathbb{TL}_{2n}$ on $\mathcal{NC}_{n}$.
In other words, we have 
\begin{align}
\label{rel:Fe}
F_{i}=\Psi^{-1} e_{i}\Psi, \qquad \rho=\Psi^{-1}\sigma\Psi,
\end{align}
where $F_{i}$ is defined in Eq. (\ref{def:Fi}) and $e_i$ is defined in Eq. (\ref{eq:ei}).
\end{prop}
\begin{proof}
To show Eq. (\ref{rel:Fe}), it is enough to show that 
$F_1=\Psi^{-1}e_1\Psi$ and $\rho=\Psi^{-1}\sigma\Psi$ 
since we have $F_{i+1}=\rho F_{i}\rho^{-1}$, and $e_{i+1}=\sigma e_{i}\sigma^{-1}$ for $i\ge1$.
Let $\pi\in\mathcal{NC}_{n}$ and $C:=\Psi(\pi)\in\mathcal{C}_{n}$.

We prove $F_1=\Psi^{-1}e_1\Psi$.
First, suppose that $1$ and $2$ belong to the same block in $\pi$.
In $C$, we have an arch connecting the two points $1$ and $1'$.
By definition of $F_1$, we have $F_{1}\pi=\tau\pi$.
Similarly, $e_1C=\tau C$. We have $F_1=\Psi^{-1}e_1\Psi$ in this case.
Secondly, suppose that $1$ and $2$ belong to distinct blocks $B_1$ and $B_2$ in $\pi$, 
where the block $B_i$ contains the integer $i$.
Let $n_1>1$ be the minimal integer in the block $B_1\setminus\{1\}$ if $|B_1|\ge2$, and 
$n_1:=1$ if $|B_{1}|=1$. 
Similarly, let $n_2$ be the maximal integer in the block $B_2$ if $|B_2|\ge2$, and 
$n_2:=2$ if $|B_2|=1$.
The chord diagram $C$ contains two distinguished arches.
One is the arch connecting $1$ and $(n_1-1)'$, and the other is the arch 
connecting $1'$ and $n_2$.
The action of $F_1$ on $\pi$ gives a new block $B_{1\cup 2}:=B_1\cup B_{2}$.
Note that $n_2$ is the maximal integer which is smaller than $n_1$ 
in $B_{1\cup 2}$ if $n_1\neq1$.
Note that if $n_1\neq1$, then there is no $k$ such that $k\in B_{1\cup 2}$ and $n_2<k<n_1$.
The action of $e_1$ on $C$ transforms the two distinguished arches 
into the two arch connecting $1$ and $1'$, and $n_2$ and $(n_1-1)'$.
Other arches remain the same after the actions of $e_1$ on $C$.
From these, it is obvious that we have $\Psi(F_{1}\pi)=e_1C$.
Therefore, we have $F_1=\Psi^{-1}e_1\Psi$.

We will show that $\rho=\Psi^{-1}\sigma\Psi$.
Suppose that the two integers $i$ and $j$ belong to the same block such that $j$ is the 
minimum larger than $i$. 
This means that we have an arch $A_1$ connecting $i$ and $(j-1)'$ in $C$.
By the definition of Kreweras endomorphism, the two integers $i+1$ and $j$ belong 
to the same block in $\rho(\pi)$. Further, $j$ is the largest integer in this block.
In terms of the chord diagram, we have an arch $A_2$ connecting the two points $j$ and $i'$ 
in $\Psi(\rho(\pi))$.
It is easy to see that $\sigma A_1=A_2$ by the definition of $\sigma$. 
From these, we have $\rho=\Psi^{-1}\sigma\Psi$.
This completes the proof.
\end{proof}

Proposition \ref{prop:Fe} can be summarized as in the following commutative 
diagrams:
\begin{center}
\begin{tikzcd}
\pi \arrow[r,"F_{i}"] \arrow[d,"\Psi"] 
& \pi' \arrow["\Psi",d] \\
C \arrow[r,"e_{i}"] & C'
\end{tikzcd}\qquad
\begin{tikzcd}
\pi \arrow[r,"\rho"] \arrow[d,"\Psi"] 
& \pi' \arrow["\Psi",d] \\
C \arrow[r,"\sigma"] & C'
\end{tikzcd}
\end{center}
where $\pi,\pi'\in\mathcal{NC}_{n}$ and $C,C'\in\mathcal{C}_{n}$.

\begin{example}
The actions of $F_i$, $1\le i\le 5$, on the 
non-crossing partitions of size $3$ are given in Figure \ref{fig:NCtoChord}. 
Each chord diagram corresponds to the non-crossing partition 
depicted in Figure \ref{fig:NC3}. Therefore, a chord diagram is obtained 
from a non-crossing partition by $\Psi$.
\begin{figure}[ht]
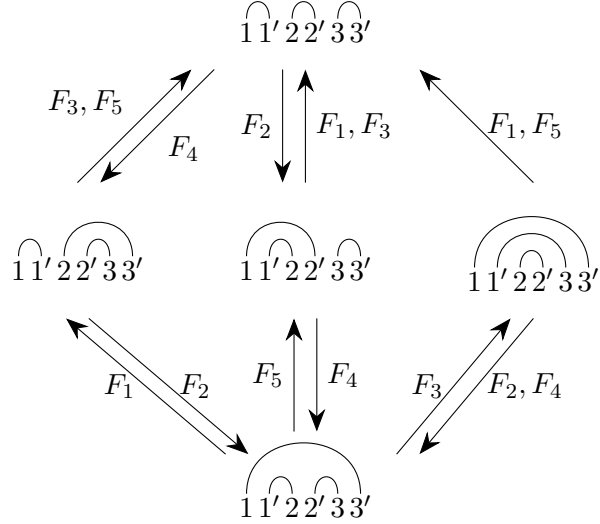

\tikzpic{-0.5}{[xscale=3.0,yscale=3]
\node(0) at (0,0){
\tikzpic{-0.5}{[scale=0.3]
\draw(0,0)..controls(0,1)and(1,1)..(1,0);
\draw(2,0)..controls(2,1)and(3,1)..(3,0);
\draw(4,0)..controls(4,1)and(5,1)..(5,0);
\foreach \x/\y in{0/1,2/2,4/3}
\draw(\x,-1.5)node[anchor=south]{$\y$};
\foreach \x/\y in {1/{1'},3/{2'},5/{3'}}
\draw(\x,-1.5)node[anchor=south]{$\y$};
}};
\node(1) at (-1,-1){
\tikzpic{-0.5}{[scale=0.3]
\draw(0,0)..controls(0,1)and(1,1)..(1,0);
\draw(2,0)..controls(2,2)and(5,2)..(5,0);
\draw(3,0)..controls(3,1)and(4,1)..(4,0);
\foreach \x/\y in{0/1,2/2,4/3}
\draw(\x,-1.5)node[anchor=south]{$\y$};
\foreach \x/\y in {1/{1'},3/{2'},5/{3'}}
\draw(\x,-1.5)node[anchor=south]{$\y$};
}
};
\node(2) at (0,-1){
\tikzpic{-0.5}{[scale=0.3]
\draw(0,0)..controls(0,2)and(3,2)..(3,0);
\draw(1,0)..controls(1,1)and(2,1)..(2,0);
\draw(4,0)..controls(4,1)and(5,1)..(5,0);
\foreach \x/\y in{0/1,2/2,4/3}
\draw(\x,-1.5)node[anchor=south]{$\y$};
\foreach \x/\y in {1/{1'},3/{2'},5/{3'}}
\draw(\x,-1.5)node[anchor=south]{$\y$};
}
};
\node(3) at (1,-1){
\tikzpic{-0.5}{[scale=0.3]
\draw(0,0)..controls(0,3)and(5,3)..(5,0);
\draw(1,0)..controls(1,2)and(4,2)..(4,0);
\draw(2,0)..controls(2,1)and(3,1)..(3,0);
\foreach \x/\y in{0/1,2/2,4/3}
\draw(\x,-1.5)node[anchor=south]{$\y$};
\foreach \x/\y in {1/{1'},3/{2'},5/{3'}}
\draw(\x,-1.5)node[anchor=south]{$\y$};
}
};
\node(4) at (0,-2){
\tikzpic{-0.5}{[scale=0.3]
\draw(0,0)..controls(0,3)and(5,3)..(5,0);
\draw(1,0)..controls(1,1)and(2,1)..(2,0);
\draw(3,0)..controls(3,1)and(4,1)..(4,0);
\foreach \x/\y in{0/1,2/2,4/3}
\draw(\x,-1.5)node[anchor=south]{$\y$};
\foreach \x/\y in {1/{1'},3/{2'},5/{3'}}
\draw(\x,-1.5)node[anchor=south]{$\y$};
}
};
\draw[{Stealth[length=3mm]}-] (-0.5,-0.2)--node[anchor=south east]{$F_3,F_5$}(-1,-0.7);
\draw[-{Stealth[length=3mm]}] (-0.4,-0.2)--node[anchor=north west]{$F_4$}(-0.9,-0.7);
\draw[-{Stealth[length=3mm]}] (-0.1,-0.2)--node[anchor=east]{$F_2$}(-0.1,-0.7);
\draw[{Stealth[length=3mm]}-] (0,-0.2)--node[anchor=west]{$F_1,F_3	$}(0,-0.7);
\draw[-{Stealth[length=3mm]}] (1,-0.7)--node[anchor=west]{$F_1,F_5$}(0.5,-0.2);
\draw[-{Stealth[length=3mm]}] (1,-1.3)--node[anchor=west]{$F_2,F_4$}(0.5,-1.9);
\draw[{Stealth[length=3mm]}-] (0.9,-1.3)--node[anchor=east]{$F_3$}(0.4,-1.9);
\draw[-{Stealth[length=3mm]}] (0.05,-1.3)--node[anchor=west]{$F_4$}(0.05,-1.8);
\draw[{Stealth[length=3mm]}-] (-0.05,-1.3)--node[anchor=east]{$F_5$}(-0.05,-1.8);
\draw[-{Stealth[length=3mm]}] (-0.95,-1.3)--node[anchor=west]{$F_2$}(-0.25,-1.9);
\draw[{Stealth[length=3mm]}-] (-1.05,-1.3)--node[anchor=east]{$F_1$}(-0.35,-1.9);
}
\caption{The actions of the generators $F_{i}$ on non-crossing partitions, or equivalently chord diagrams in $\mathcal{NC}_{3}$}
\label{fig:NCtoChord}
\end{figure}
Here, we do not depict the action of $F_{i}$ if it is obvious. 
Note that the actions of $F_{i}$ non-crossing partitions are the same as 
the actions of $e_{i}$ on chord diagrams.
For example, we have $F_{5}(1/2/3)=13/2$.
There is no $F_{i}$ such that $1/23=F_{i}(123)$.
\end{example}

The next proposition explains the relation between the cover relation $\pi\lessdot \nu$ in $\mathcal{NC}_{n}$
and the action of a generator $F_{i}$ for some $i$.

\begin{prop}
\label{prop:NCchord}
Let $C_1$ and $C_2$ be two distinct chord diagrams, and $\pi_i:=\Psi^{-1}(C_i)$ for $i=1$ and $2$.
We have $C_2=e_{i}C_{1}$ for some $1\le i\le 2n-1$
if and only if two non-crossing partitions $\pi_1$ and $\pi_2$ satisfy $\pi_{1}\lessdot\pi_{2}$
for $i\equiv1\pmod2$ or $\pi_{2}\lessdot\pi_{1}$ for $i\equiv0\pmod2$. 
\end{prop}
\begin{proof}
Suppose we have $C_{2}=e_{i}C_{1}$. By applying $\Psi^{-1}$, we have 
$\pi_2=F_{i}\pi_{1}$ from Proposition \ref{prop:Fe}.
Since $C_1\neq C_{2}$, we have $\pi_1\neq \pi_2$.
From Eq. (\ref{def:Fi}), we have 
\begin{align*}
\rho^{-i+1}\pi_2&=\rho^{-i+1}F_{i}\rho^{i-1}\rho^{-i+1}\pi_{1}, \\
&=F_{1}\rho^{-i+1}\pi_{1}.
\end{align*}
By definition of $F_1$, $\rho^{-i+1}\pi_2$ is obtained from $\rho^{-i+1}\pi_1$
by merging the two blocks $B_1$ and $B_{2}$ such that the block $B_{i}$ contains 
the integer $i$ and $B_1\neq B_{2}$.
Otherwise, we have $\pi_2=\pi_1$ since $B_1=B_2$. 
This implies that $\mathtt{rk}(\rho^{-i+1}\pi_2)=\mathtt{rk}(\rho^{-i+1}\pi_{1})$+1, 
and $\rho^{-i+1}\pi_1\lessdot\rho^{-i+1}\pi_2$.
From Proposition \ref{prop:Kre}, 
we have $\rho^{-i+1}\pi_1\lessdot\rho^{-i+1}\pi_2$ if and only if the two 
non-crossing partitions satisfy
$\pi_1\lessdot\pi_2$ for $i\equiv1\pmod2$ or $\pi_2\lessdot\pi_1$ for $i\equiv0\pmod2$.
This completes the proof.
\end{proof}

\section{Fuss--Catalan algebra on generalized Dyck paths}
\label{sec:FCGDP}
\subsection{An increasing chain in \texorpdfstring{$\mathcal{NC}_{n}$}{NCn}}
To define the action of the Fuss--Catalan algebra, which is a generalization of the Temperley--Lieb algebra, 
on generalized Dyck paths, we introduce an increasing $r$-chain in the poset of $\mathcal{NC}_{n}$ and 
define the set of generators $\{F_{i}^{(s)}: 1\le i\le 2n-1, 1\le s\le r\}$.
Since an increasing $r$-chain will be identified with a generalized Dyck path by a bijection, we will first define 
the Fuss--Catalan algebra on increasing $r$-chains.

Fix a positive integer $r$.
An {\it increasing $r$-chain} $\pi^{(r)}:=(\pi_1,\ldots,\pi_{r})$ in the poset of $\mathcal{NC}_{n}$ is 
a sequence of non-crossing partitions such that 
$\pi_1\le \pi_2\le\ldots\le\pi_{r}$.
We define the set of increasing $r$-chains of $\mathcal{NC}_{n}$ by $\mathcal{NC}_{n}^{(r)}$ for $r\ge1$.
Note that if $r=1$,  then $\mathcal{NC}_{n}^{(r)}=\mathcal{NC}_{n}$.

\begin{example}
When $(n,r)=(3,2)$, we have twelve elements in $\mathcal{NC}_{3}^{(2)}$:
\begin{align*}
&(1/2/3,1/2/3) && (1/2/3,12/3) && (1/2/3,13/2) && (1/2/3,1/23) \\
&(1/2/3,123) && (12/3,12/3)  && (12/3,123) && (13/2,13/2) \\
&(13/2,123) && (1/23,1/23) && (1/23,123) && (123,123)
\end{align*}
\end{example}

As we will prove later in Proposition \ref{prop:NCP}, 
the number of elements in $\mathcal{NC}_{n}^{(r)}$ is given by the 
Fuss--Catalan number.
Therefore, we have $|\mathcal{NC}_{n}^{(r)}|=|\mathcal{P}^{(r)}_{n}|$.
To define a Fuss--Catalan algebra, which is a generalized Temperley--Lieb algebra, on $\mathcal{P}^{(r)}_{n}$, 
it is enough to define the algebra on $\mathcal{NC}_{n}^{(r)}$.
One translates the results on $\mathcal{NC}_{n}^{(r)}$ into the ones on $\mathcal{P}_{n}^{(r)}$
by a bijection between the two sets.
Below, we introduce and study such a bijection.

\subsection{A bijection between a generalized Dyck path and an increasing chain}
\label{sec:GDPic}
We study a bijection $\kappa^{(r)}$ between an $r$-Dyck path in $\mathcal{P}_{n}^{(r)}$ 
and an increasing $r$-chain in $\mathcal{NC}_{n}$.
Let $\pi:=(\pi_1,\ldots,\pi_{r})\in\mathcal{NC}_{n}^{(r)}$ be an increasing $r$-chain. 
Each non-crossing partition $\pi_i$ consists of $h_i\ge1$ blocks, that is, 
$\pi_{i}=B^{(i)}_1/B^{(i)}_2/\ldots/B^{(i)}_{h_{i}}$ such that 
$\min B^{(i)}_{h}<\min B^{(i)}_{h+1}$ for all $1\le h\le h_i-1$.
We first assign $r$-Dyck paths to each block $B^{(1)}_{h}$ by
\begin{align*}
B^{(1)}_{h} \leftrightarrow UR^{r-1}(UR^{r})^{l}R,
\end{align*}
where $1\le h\le h_1$, and $l$ is the number of integers in $B_{h}^{(1)}$ minus one, {\it i.e.}, 
$l:=|B^{(1)}_{h}|-1$.

We construct $h_{i}$ $r$-Dyck paths corresponding to each block $B^{(i)}_{h}$ 
from $h_{i-1}$ $r$-Dyck paths for the blocks $B^{(i-1)}_{h}$.
Note that we always have $h_{i}\le h_{i-1}$ since we have $\pi_{i-1}\le\pi_{i}$.
Suppose that $B^{(i)}_{h}$ is obtained from some blocks $B^{(i-1)}_{j}$ 
by merging them into a larger block.
In other words, $B^{(i)}_{h}$ can be written as 
\begin{align*}
B^{(i)}_{h}=B^{(i-1)}_{j_1}\cup B^{(i-1)}_{j_2}\cup\ldots\cup B^{(i-1)}_{j_m},
\end{align*}
as a set of integers where $j_1<j_2<\ldots<j_{m}$. 
This implies that we have $\min B^{(i-1)}_{j_p}<\min B^{(i-1)}_{j_{p+1}}$.
Let $P_{j_k}$ be an $r$-Dyck path corresponding to $B^{(i-1)}_{j_{k}}$, $1\le k\le m$.
Let $a:=\{q\in B^{(i-1)}_{j_{1}} : q<\min B^{(i-1)}_{j_2}\}$.
We merge $B^{(i-1)}_{j_1}$ and $B^{(i-1)}_{j_2}$ as a larger block 
$B_{j_1\cup j_2}:=B^{(i-1)}_{j_1}\cup B^{(i-1)}_{j_2}$.
An new $r$-Dyck path $P'_{j_1\cup j_2}$ corresponding to $B_{j_1\cup j_2}$ is obtained from 
$P_{j_1}$ and $P_{j_2}$ by inserting $P_{j_2}$ at the $a(r+1)$-th position of $P_{j_1}$ from left. 
Then, we continue this process until we obtain an $r$-Dyck path. 
Namely, we obtain an $r$-Dyck path $P'_{j_1\cup j_2\cup j_3}$ corresponding to 
$B_{j_1\cup j_2\cup j_3}:=B_{j_1\cup j_2}\cup B^{(i-1)}_{j_3}$ by 
inserting $P_{j_3}$ at the $b(r+1)$-th position of $P'_{j_1\cup j_2}$ from left, 
where $b:=\{q\in B_{j_1\cup j_2} : q<\min B^{(i-1)}_{j_3}\}$.
In this way, we obtain an $r$-Dyck path $P'_{j_1\cup j_2\cup\ldots \cup j_{m}}$ corresponding 
to the merged blocks $B_{h}^{(i)}$.
Let $\mathcal{I}$ be the set of positions of $U$ from left in $P'_{j_1\cup\ldots\cup j_{m}}$.
By definition of $r$-Dyck path, the first step is always $U$, which implies that $1\in \mathcal{I}$.
Finally, we define an $r$-Dyck path $P_{j_1\cup\ldots\cup j_m}$ 
such that the set of the positions of $U$ in $P_{j_1\cup\ldots\cup j_{m}}$ 
is given by 
\begin{align*}
\{1\} \cup \{i-1: i\in\mathcal{I}\setminus\{1\}\}.
\end{align*}
In this way we have a correspondence 
\begin{align*}
B^{(i)}_{h}\leftrightarrow P_{j_1\cup\ldots\cup j_{m}}.
\end{align*}
By the algorithm above, we obtain $h_r$ $r$-Dyck paths corresponding 
to the blocks $B^{(r)}_{h}$, $1\le h\le h_r$.
To obtain an $r$-Dyck path for an increasing chain $\pi$, we merge $h_r$ $r$-Dyck paths 
in the same way as we obtain $P'_{j_1\cup\ldots\cup j_{m}}$ from 
the $r$-Dyck paths $P_{j_{k}}$.
The newly obtained $r$-Dyck path corresponds to 
the $r$-chain $\pi$.

The above map from $\pi\in\mathcal{NC}_{n}^{(r)}$ to $P\in\mathcal{P}_{n}^{(r)}$ 
is invertible.
We briefly explain the inverse.
Given an $r$-Dyck path $P$, we decompose $P$ into several smaller $r$-Dyck paths $P_{j_k}$, 
$1\le k\le m$. Here, $m$ stands for the number of decomposed $r$-Dyck paths.
From this decomposition, we have a non-crossing partition $\pi_{r}$.
Then, we push right the $U$ steps except the first one in $P_{j_k}$ and obtain 
a new $r$-Dyck path $P'_{j_k}$.
Again, we decompose $P'_{j_k}$ into smaller $r$-Dyck paths, and obtain a 
non-crossing partition $\pi_{r-1}$.
We continue until we have $r$ non-crossing partitions.
Then, we define the $r$-chain $\pi$ by $\pi:=(\pi_1,\ldots,\pi_r)$. 
By construction, we always have $\pi_{i-1}\le \pi_{i}$ for $2\le i\le r$.
 
\begin{defn}
We define the above bijection from $\mathcal{NC}_{n}^{(r)}$ to $\mathcal{P}_{n}^{(r)}$
by $\kappa^{(r)}$.
\end{defn}

The following proposition is a direct consequence of the existence of a bijection 
between $\mathcal{NC}_{n}^{(r)}$ and $\mathcal{P}_{n}^{(r)}$.
\begin{prop}
\label{prop:NCP}
The number of $r$-chains in $\mathcal{NC}_{n}^{(r)}$ is the Fuss--Catalan number,
{\it i.e.}, 
\begin{align*}
|\mathcal{NC}_n^{(r)}|=\genfrac{}{}{}{}{1}{rn+1}\genfrac{(}{)}{0pt}{}{(r+1)n}{n}.
\end{align*}
\end{prop}

\begin{remark}
It is well-known that an $r+1$-ary tree is bijective to an $r$-Dyck path.
In \cite{Edel80}, Edelman constructed a bijection between $r+1$-ary trees 
and $r$-chains in the lattice of non-crossing partitions.
The bijection $\kappa^{(r)}$ given in Section \ref{sec:GDPic} is different from the bijection 
given by Edelman. 
The bijection $\kappa^{(r)}$ is a generalization of the bijection between a Dyck path and 
a non-crossing partition given by Stump in \cite{Stu13}.
\end{remark}

\begin{example}
Consider the $3$-chain $\pi=(1/2/3/4,14/23, 1234)$.
Let $\pi_1=B^{(1)}_1/B^{(1)}_2/B^{(1)}_{3}/B^{(1)}_{4}=1/2/3/4$, 
$\pi_2=B^{(2)}_1/B^{(2)}_2=14/23$ and $\pi_3=B^{(3)}_1=1234$.
For $\pi_1$, we have four Dyck paths $URRR$ for $B_{i}$, $1\le i\le 4$.
For $\pi_2$, we merge $B^{(1)}_1$ and $B^{(1)}_4$, and $B^{(1)}_2$ and $B^{(1)}_3$
into larger blocks. Therefore, we have two $3$-Dyck paths $UR^2UR^4$ for 
$B^{(2)}_{1}$ and $B^{(2)}_{2}$.
To obtain $3$-Dyck path for $\pi_3$, we merge $B^{(2)}_1$ and $B^{(2)}_2$ into 
a large block $B^{(3)}_1$.
We have $\{i\in B^{(2)}_1 : i<\min B^{(2)}_1\}=\{1\}$ since $B^{(2)}_1=\{1,4\}$ 
and $B^{(2)}_{2}=\{2,3\}$.
We insert the $3$-Dyck path for $B^{(2)}_2$ into the fourth position of the $3$-Dyck 
path for $B^{(2)}_1$, which gives the $3$-Dyck path 
$UR^2U^2R^2UR^8$. Then, the set $\mathcal{I}$ of positions of $U$ is given
by $\mathcal{I}=\{1,4,5,8\}$.
Finally, the set of the positions of $U$ steps in $3$-Dyck path for $\pi$ is
given by $\{1,3,4,7\}$. 
From this, we have the $3$-Dyck path $URU^2R^2UR^9$.
\end{example}

\begin{example}
Consider the $3$-Dyck path $URU^2R^8$ of size $3$. We will construct an increase $3$-chain
$\pi=(\pi_1,\pi_2,\pi_3)$.
First, $URU^2R^8$ can not be decomposed into smaller $3$-Dyck paths, we have 
$\pi_3=123$.
By moving the up steps rightward by one unit, we have $UR^2U^2R^7$.
This $3$-Dyck path can be decomposed into $UR^2UR^4$ and $UR^3$, therefore, we have 
$\pi_2=13/2$. By moving up steps rightward by one unit, we have two $3$-Dyck paths 
$UR^3UR^3$ and $UR^3$. The path $UR^3UR^3$ can be decomposed into two Dyck paths $UR^3$ of size $1$.
We have $\pi_1=1/2/3$.
As a result, we have an increasing chain $(1/2/3,13/2,123)$  for $URU^2R^8$.
\end{example}

\subsection{Extended Kreweras endomorphism}
We extend the action of the Kreweras endomorphism $\rho$ on $\mathcal{NC}_{n}$ to that of $\mathcal{NC}_{n}^{(r)}$ 
as follows.
Let $\pi^{(r)}:=(\pi_1,\ldots,\pi_{r})\in\mathcal{NC}_{n}^{(r)}$. 
We define the Kreweras endomorphism $\rho:\mathcal{NC}_{n}^{(r)}\rightarrow\mathcal{NC}_{n}^{(r)}$ 
by
\begin{align}
\label{eq:FKre}
\rho(\pi^{(r)}):=(\rho(\pi_{r}),\rho(\pi_{r-1}),\ldots,\rho(\pi_1)),
\end{align} 
where $\rho$ in the right hand side of Eq. (\ref{eq:FKre}) is the 
Kreweras endomorphism for $\mathcal{NC}_{n}$.
Recall that the Kreweras endomorphism $\rho$ reverses the rank and the cover relation
by Proposition \ref{prop:Kre}.
This implies that $\rho(\pi^{(r)})$ is also an increasing $r$-chain in $\mathcal{NC}_{n}^{(r)}$, i.e., 
we have $\rho(\pi_{i+1})\le\rho(\pi_{i})$ for $1\le i\le r-1$.

For example, we have 
\begin{align*}
\rho(1/23/4,14/23,1234)=(1/2/3/4,1/24/3,124/3).
\end{align*}
Note that $1/23/4\le 14/23\le 1234$ implies 
$1/2/3/4\le 1/24/3\le 124/3$.

The next proposition is a generalization of Proposition \ref{prop:xisigma}, and
it relates the rotation $\xi$ and $\widetilde{\sigma}$ with the extended Kreweras endomorphism.
\begin{prop}
We have $\xi^{r+1}=\widetilde{\sigma}^{r+1}=\rho^{2}$.
\end{prop}
\begin{proof}
From Proposition \ref{prop:Kre}, the square of $\rho$ is a rotation of an $r$-chain.
The maps $\xi$ and $\widetilde{\sigma}$ are also rotations of an $r$-Dyck path.	

A generalized chord diagram $\widetilde{C}$ in $\widetilde{\mathcal{C}}_{n}^{(r)}$ has $n$ arches 
and an arch connects $r+1$ points.
The map $\widetilde{\sigma}^{r+1}$ moves the position of an arch in $\widetilde{C}$ rightward 
by $r+1$ units. We consider the inverse map of $\kappa^{(r)}$. 
The move by $r+1$ units implies that we replace the integers $i$ by $i+1$ in 
the corresponding non-crossing partition.
Then, this increment of the integers can be realized by the rotation $\rho^{2}$.
From these, $\xi^{r+1}=\widetilde{\sigma}^{r+1}$ gives the rotation equivalent to $\rho^{2}$.
\end{proof}

\subsection{The first bijection between \texorpdfstring{$\mathcal{NC}_{n}^{(r)}$}{NC} and 
\texorpdfstring{$\mathcal{C}_{n}^{(r)}$}{C}}
\label{sec:fbijNCC}
In Section \ref{sec:TLNC}, we have a bijection $\Psi$ between a chord diagram 
in $\mathcal{C}_{n}$ and a non-crossing partition in $\mathcal{NC}_{n}$.
In this section, we generalize this correspondence. Namely, we will have 
a bijection $\Psi:=\Psi^{(r)}$ between a generalized chord diagram in $\mathcal{C}_{n}^{(r)}$
and $\mathcal{NC}_{n}^{(r)}$ by generalizing the bijection $\Psi$ (introduced in
Definition \ref{def:Psi}) in a natural way.

Let $\pi^{(r)}:=(\pi_{1},\ldots,\pi_{r})\in\mathcal{NC}_{n}^{(r)}$, and 
$C_{s}:=\Psi(\pi_{s})$, $1\le s\le r$, be a chord diagram corresponding to $\pi_{i}\in\mathcal{NC}_{n}$.
We construct a generalized chord diagram $C(\pi^{(r)})$ from the set of 
chord diagrams $\{C_{s} : 1\le s\le r\}$.
Recall that each point $i$ or $i'$ in a generalized chord diagram has $r$ points.
Suppose that the two points $i$ and $j'$ are connected by an arch in $C_{s}$.
We connect the $s$-th point from left in $i$ and the $s$-th point from right in $j'$
by an arch. 
This is compatible with the condition (\ref{eq:condA}).
In this way, the set $\{C_{s} : 1\le s\le r\}$ determines $rn$ arches in a generalized 
chord diagram.
We denote the map from $\mathcal{NC}_{n}^{(r)}$ to $\mathcal{C}_{n}^{(r)}$ 
by $\Psi^{(r)}$ where we have used a natural inclusion 
$\mathcal{C}_{n}^{(r)}\hookrightarrow\mathcal{C}_{rn}$.

To show $\Psi^{(r)}$ is a bijection, we construct the inverse of $\Psi^{(r)}$ as follows.
Let $C$ be a generalized chord diagram in $\mathcal{C}_{n}^{(r)}$. 
Recall that a generalized chord diagram has $2n$ bundled points.
Suppose that the $s$-th point from left in $i$ and the $s$-th point from right in $j'$
are connected by an arch in $C$. 
Then, $C_{s}$ has an arch between the points $i$ and $j'$.
In this way, we have the set of $r$ chord diagrams $\{C_s : 1\le s\le r\}$.
We define the non-crossing partitions $\pi_{s}=\Psi^{-1}C_{s}$  for $1\le s\le r$, 
and obtain $\pi^{(r)}=(\pi_1,\ldots,\pi_{r})\in\mathcal{NC}_{n}^{(r)}$.

We first show that $\Psi^{(r)}$ is well-defined.
\begin{prop}
\label{prop:chainadd}
Let $\pi_1,\pi_2,\ldots,\pi_r\in\mathcal{NC}_{n}$ such that $\pi_1\le \pi_{2}\le\ldots\le\pi_r$.
Then, the superposition of the $r$ chord diagrams $\Psi(\pi_s)$ for $1\le s\le r$
gives a generalized chord diagram $(\Psi(\pi_1),\Psi(\pi_2),\ldots,\Psi(\pi_r))$ 
is in $\mathcal{C}_{n}^{(r)}$.
\end{prop}
\begin{proof}
To prove the proposition, it is enough to prove that $(\Psi(\pi_1),\Psi(\pi_2))\in\mathcal{C}_{n}^{(2)}$ for $\pi_1\le \pi_2$.
Since $\pi_{1}\le\pi_{2}$, a block in $B_{j}^{(2)}$ in $\pi_{2}$
is written as 
\begin{align*}
B_{j}^{(2)}=B^{(1)}_{j_1}\cup B_{j_2}^{(1)}\cup\ldots\cup B_{j_m}^{(1)},
\end{align*}
where $B^{(1)}_{j}$ is a block in $\pi_1$.
We first prove the case of $\pi_1\lessdot\pi_2$. 	
Since $\pi_1\lessdot\pi_2$, there exists a unique block $B_{j}^{(2)}$ in $\pi_2$ such 
that $B_{j}^{(2)}=B_{j_1}^{(1)}\cup B_{j_2}^{(1)}$ with $B_{j}^{(1)}\neq B_{j_1}^{(1)}, B_{j_2}^{(1)}$ 
where $B_{i}^{(t)}$ is a block in $\pi_t$ for $t=1$ and $2$.
We consider the two cases: 1) $B_{j_1}^{(1)}$ and $B_{j_2}^{(1)}$ are not nesting, 
and 2) $B_{j_1}^{(1)}$ and $B_{j_2}^{(1)}$ are nesting.

Case 1). Since $B_{j_1}^{(1)}$ and $B_{j_2}^{(1)}$ are not nesting, we have 
\begin{align*}
i_1=\min B_{j_1}^{(1)}\le i_{2}=\max B_{j_1}^{(1)}<i_3=\min B_{j_{2}}^{(1)}\le i_4=\max B_{j_2}^{(1)}.
\end{align*}
Further, since $\pi_1$ and $\pi_2$ are non-crossing and one can obtain $\pi_2$ from $\pi_1$ by merging 
$B_{j_1}^{(1)}$ and $B_{j_2}^{(1)}$, there exist no arch $(k,l)$ such that 
$i_2<k<i_3$ and $l<i_1$ or $l>i_4$.
In $\Psi(\pi_{1})$, we have  two distinguished arches. 
One is an arch which connects $i_2$ and $(i_1-1)'$, and the other is an arch 
which connects $i_4$ and $(i_3-1)'$.
By merging the two blocks in $\pi_1$, the four points $i_1,i_2,i_3$ and $i_4$ belong 
to the same block in $\pi_2$.
The chord diagram $\Psi(\pi_2)$ also has two distinguished arches.
One is an arch which connects $i_2$ and $(i_3-1)'$, and 
the other is an arch which connects $i_{4}$ and $(i_1-1)'$.
Note that $\Psi(\pi_1)$ and $\Psi(\pi_2)$ have the same arches except 
these distinguished arches.
Then, the property that $\pi_i$'s are non-crossing implies that 
the distinguished four arches are non-crossing, and a superposition of the 
two chord diagrams $\Psi(\pi_1)$ and $\Psi(\pi_2)$ is admissible 
as $\mathcal{C}_{n}^{(2)}$.

Case 2). Since $B_{j_1}^{(1)}$ and $B_{j_2}^{(1)}$ are nesting, we define 
six integers 
\begin{align*}
\min B_{j_1}^{(1)}=i_1\le i_2<i_{3}=\min B_{j_2}^{(1)}\le i_{4}=\max B_{j_2}^{(1)}
<i_{5}\le i_6=\max B_{j_1}^{(1)},
\end{align*}
where $i_2,i_5\in B_{j_1}^{(1)}$, and $i_2$ is the maximal integer which is smaller than $i_{3}$
and $i_5$ is the minimal integer which is larger than $i_4$.
Since $\pi_1\lessdot\pi_2$, there is no pair of integers $(j,k)$ such that 
$i_2<j<i_3\le i_4<k<i_{5}$ and $j$ and $k$ belong to the same block in $\pi_1$.
We have two distinguished arches in $\Psi(\pi_1)$.
The first one is an arch which connects $i_2$ and $(i_5-1)'$, and 
the other is an arch which connects $i_4$ and $(i_3-1)'$.
Similarly, $i_2$ and $(i_3-1)'$, and $i_4$ and $(i_5-1)'$ are two 
distinguished arches in $\Psi(\pi_2)$. 
Note that all the arches except these distinguished arches are the same 
in $\Psi(\pi_1)$ and $\Psi(\pi_2)$.
Again, the fact that $\pi_i$ is non-crossing implies that 
$\Psi(\pi_1)$ and $\Psi(\pi_2)$ are non-crossing and admissible
as a chord diagram in $C_{n}^{(2)}$.

From these two cases, we have $(\Psi(\pi_1),\Psi(\pi_2))\in\mathcal{C}_{n}^{(2)}$ for $\pi_1\lessdot\pi_2$.

If $\pi_1\le \pi_2$, we have a sequence of non-crossing partitions
\begin{align}
\label{eq:chain12}
\pi_1=\nu_0\lessdot \nu_1\lessdot \ldots \lessdot \nu_{s}=\pi_{2}.
\end{align}
By the above argument, the chord diagrams for $\nu_{i}$ and $\nu_{i+1}$ are 
non-crossing in $\mathcal{C}_{n}^{(2)}$.
We will show that the chord diagrams for $\pi_{1}$ and $\nu_{i}$, $1\le i\le s$,
are non-crossing in $\mathcal{C}_{n}^{(2)}$.
In the proof of the case $\pi_1\lessdot\pi_2$, we have seen that we change the connectivity of a chord diagram 
when we merge two blocks into a larger block. 
The chord diagram $\nu_2$ is obtained from $\pi_1$ by two changes of the connectivity. 
The property that blocks are non-crossing implies that the chord diagrams for $\nu_2$ and $\pi_1$
are non-crossing in $\mathcal{C}_{n}^{(2)}$.
By repeating a similar argument, one can show that the chord diagrams for $\pi_1$ and $\pi_2$ are 
admissible in $\mathcal{C}_{n}^{(2)}$.

Consider the $r$-chain $\pi_1\le \pi_2\le\ldots\le\pi_{r}$.
Since we have $\pi_{i}\le \pi_{j}$, $i<j$, the chord diagrams for $\pi_i$ and $\pi_{j}$ are 
non-crossing in $\mathcal{C}_{n}^{(2)}$.
Then, it is easy to see that the superposition of $\Psi(\pi_s)$, $1\le s\le r$, gives 
a generalized chord diagram in $\mathcal{C}_{n}^{(r)}$.
Especially, the arches are non-crossing.
This completes the proof.
\end{proof}

\begin{example}
Let $\pi^{(4)}=(1/23/4,1/23/4,14/23,1234)\in\mathcal{NC}_{4}^{(4)}$. 
The generalized chord diagram corresponding to $\pi^{(4)}$ is depicted 
in Figure \ref{fig:pi44}.
\begin{figure}[ht]
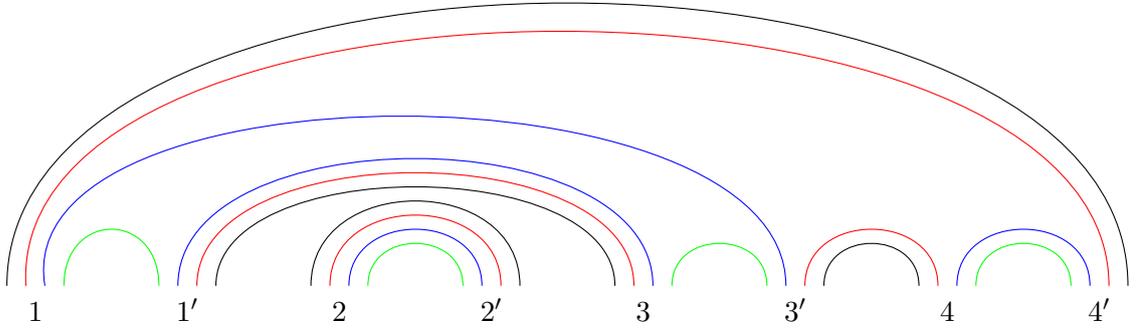

\tikzpic{-0.5}{[scale=0.5]
\draw(0,0)..controls(0,10)and(29.5,10)..(29.5,0);
\draw[red](0.5,0)..controls(0,9)and(29,9)..(29,0);
\draw[blue](1,0)..controls(0,6)and(20.5,6)..(20.5,0);
\draw[green](1.5,0)..controls(1.5,2)and(4,2)..(4,0);
\draw[blue](4.5,0)..controls(4.5,4.5)and(17,4.5)..(17,0);
\draw[red](5,0)..controls(5,4)and(16.5,4)..(16.5,0);
\draw(5.5,0)..controls(5.5,3.5)and(16,3.5)..(16,0);
\draw(8,0)..controls(8,3)and(13.5,3)..(13.5,0);
\draw[red](8.5,0)..controls(8.5,2.5)and(13,2.5)..(13,0);
\draw[blue](9,0)..controls(9,2)and(12.5,2)..(12.5,0);
\draw[green](9.5,0)..controls(9.5,1.5)and(12,1.5)..(12,0);
\draw[green](17.5,0)..controls(17.5,1.5)and(20,1.5)..(20,0);
\draw[red](21,0)..controls(21,2)and(24.5,2)..(24.5,0);
\draw(21.5,0)..controls(21.5,1.5)and(24,1.5)..(24,0);
\draw[blue](25,0)..controls(25,2)and(28.5,2)..(28.5,0);
\draw[green](25.5,0)..controls(25.5,1.5)and(28,1.5)..(28,0);
\draw(0.75,-1.2)node[anchor=south]{$1$};
\draw(4.75,-1.2)node[anchor=south]{$1'$};
\draw(8.75,-1.2)node[anchor=south]{$2$};
\draw(12.75,-1.2)node[anchor=south]{$2'$};
\draw(16.75,-1.2)node[anchor=south]{$3$};
\draw(20.75,-1.2)node[anchor=south]{$3'$};
\draw(24.75,-1.2)node[anchor=south]{$4$};
\draw(28.75,-1.2)node[anchor=south]{$4'$};
}
\caption{A chord diagram corresponding to $\pi^{(4)}=(1/23/4,1/23/4,14/23,1234)$}
\label{fig:pi44}
\end{figure}
The chord diagrams in black and red correspond to $1/23/4$. 
Similarly, the blue and green diagrams correspond to $14/23$ and $1234$ respectively. 
Note that the four chord diagrams are non-crossing.

The action of $\rho^{-1}$ is given by $\rho^{-1}(\pi^{(4)})=(1/2/3/4,13/2/4,134/2,134/2)$.
\end{example}

The following proposition is a direct consequence of the bijection between 
$\mathcal{NC}_{n}^{(r)}$ and $\mathcal{C}_{n}^{(r)}$ and Proposition \ref{prop:NCP}.
\begin{prop}
\label{prop:NCC}
We have 
\begin{align*}
|\mathcal{P}_{n}^{(r)}|=|\mathcal{NC}_n^{(r)}|=|\mathcal{C}_{n}^{(r)}|.
\end{align*}
\end{prop}

The bijections $\kappa^{(1)}$ (introduced in Section \ref{sec:GDPic}), 
$\Psi^{(1)}$ (defined in Definition \ref{def:Psi}) and the Kreweras endomorphism $\rho$ are related as follows.
\begin{prop}
\label{prop:PC}
Let $P_{\pi}:=\kappa^{(1)}(\pi)\in\mathcal{P}_{n}^{(1)}$ and $C_{\pi}:=\Psi^{(1)}(\pi)\in\mathcal{C}_{n}^{(1)}$.
Then, we have $P_{\pi}=C_{\rho(\pi)}$ as words of $\{U,R\}^{\ast}$.
\end{prop}
\begin{proof}
Suppose that $\pi$ consists of $m$ block $B_1,\ldots,B_{m}$ such that $\min B_i<\min B_{i+1}$ 
for $1\le i\le m-1$.
Recall that a block $B_{i}$ is a sequence of increasing integers, and we denote 
the number of integers in $B_i$ by $l_{i}:=|B_{i}|$. 
We assign a Dyck path to $B_{i}$: 
\begin{align}
\label{eq:BDyck}
B_{i} \leftrightarrow U(UR)^{l_i-1}R.
\end{align}
When $\pi=12\ldots n$, it is obvious that we have $P_{\pi}=C_{\rho(\pi)}$ by 
Eq. (\ref{eq:BDyck}) under the identification given in Section \ref{sec:chord}.
To obtain $P_{\pi}$, we insert Dyck paths corresponding to the blocks $B_{i}$, $2\le i\le m$, 
into the Dyck path corresponding to $B_{1}$ one-by-one.
Let $P_{1\cup 2\cup\ldots\cup i}$ be the Dyck path obtained from the blocks $B_{j}$, $1\le j\le i$.
We consider the case where we insert the Dyck path corresponding to $B_{i}$ into 
the Dyck path $P_{1\cup2\cup\ldots \cup (i-1)}$.
The correspondence (\ref{eq:BDyck}) implies that we have an arch connecting to 
the integer $i_{\min}:=\min B_{i}$ and the minimal integer $t$ such that 
$t>i$ and $t$ is in the merged block $B_{1}\cup\ldots\cup B_{i-1}$.
Then, Proposition \ref{prop:jt} implies that the integers $t$ and $j$ belong 
to the same block in $\rho(\pi)$.
From this, we have $P_{\pi}=C_{\rho(\pi)}$.
\end{proof}

Proposition \ref{prop:PC} can be visualized by the following commutative diagram:
\begin{center}
\begin{tikzcd}
\pi \ar[r,"\Psi"] \ar[d,"\kappa"] & C_{\pi} \ar[d,"\rho"] \\
P_{\pi} \ar[r,"\sim"] & C_{\pi'} 
\end{tikzcd}
\end{center}
where $\Psi=\Psi^{(1)}$ and $\kappa=\kappa^{(1)}$.
We identify $P_{\pi}$ and $C_{\pi'}$ by the bijection given in Section \ref{sec:chord}.

The following lemma is a direct consequence of Proposition \ref{prop:Fe}.
\begin{lemma}
\label{lemma:rhosigam}
We have the following commutative diagram: 
\begin{center}
\begin{tikzcd}
\pi \ar[r,"\rho"] \arrow[d,"\Psi"'] & \pi' \ar[d,"\Psi"] \\
C \ar[r,"\sigma"] & C'
\end{tikzcd}
\end{center}
\end{lemma}

The next proposition gives the relation among $\kappa^{(1)}$, $\sigma$ and $\rho$.
\begin{prop}
We have $\sigma\circ\kappa^{(1)}=\kappa^{(1)}\circ\rho$.
\end{prop}
\begin{proof}
From Proposition \ref{prop:PC}, we have $\kappa^{(1)}=\Psi^{(1)}\circ\rho$.
Then, we have 
\begin{align*}
\sigma\circ\kappa^{(1)}&=\sigma\circ\Psi^{(1)}\circ\rho, \\
&=(\Psi^{(1)}\circ\rho)\circ\rho, \\
&=\kappa^{(1)}\circ\rho,
\end{align*}
where we have used Lemma \ref{lemma:rhosigam} to obtain the second equality.
\end{proof}

\subsection{The second bijection between \texorpdfstring{$\mathcal{NC}_{n}^{(r)}$}{NC} and 
\texorpdfstring{$\mathcal{C}_{n}^{(r)}$}{C}}
In this section, we consider another bijection $\Phi$ from $\mathcal{NC}_{n}^{(r)}$ 
to $\mathcal{C}_{n}^{(r)}$.
Recall that an element $\pi^{(r)}:=(\pi_1,\ldots,\pi_{r})\in\mathcal{NC}_{n}^{(r)}$
is an $r$-chain in $\mathcal{NC}_{n}$.
In Section \ref{sec:GDPic}, we have introduced the bijection $\kappa^{(r)}$ from 
$\mathcal{NC}_{n}^{(r)}$ to $\mathcal{P}_{n}^{(r)}$. 
Since the bijection $\kappa^{(1)}$ acts on a non-crossing partition, it has a natural 
action on $\pi^{(r)}$, {\it i.e.}, we define the map 
$\Phi':\mathcal{NC}_{n}^{(r)}\rightarrow(\mathcal{P}_{n}^{(1)})^{r}$ by
\begin{align*}
\Phi'(\pi^{(r)}):=(\kappa^{(1)}(\pi_1),\ldots, \kappa^{(1)}(\pi_{r})).
\end{align*}
By Section \ref{sec:chord}, each Dyck path $\kappa^{(1)}(\pi_{s})$, $1\le s\le r$, can be 
identified with a chord diagram.
Then, we define $\Phi$ as the compositions of the maps 
$\Phi:\mathcal{NC}_{n}^{(r)}\rightarrow(\mathcal{P}_{n}^{(1)})^{n}\rightarrow\mathcal{C}_{n}^{(r)}$, 
$\pi^{(r)}\mapsto\Phi'(\pi^{(r)})\mapsto C^{(r)}$, 
where $C^{(r)}$ is obtained from $\pi^{(r)}$
by the superposition of the $r$ chord diagrams in $\Phi'(\pi^{(r)})$.

The two maps $\Phi$ and $\Psi$ are related as follows.
\begin{prop}
\label{prop:PhiPsi}
We have $\Phi=\sigma^{(r)}\circ\Psi$ where $\sigma^{(r)}$ is the rotation of a chord 
diagram in $\mathcal{C}_{n}^{(r)}$ defined in Section \ref{sec:rotchord}.
\end{prop}
\begin{proof}
Since a generalized chord diagram $C(\pi^{(r)})$ in $\mathcal{C}_{n}^{(r)}$ is obtained by a 
superposition of the chord diagrams for $\pi_{i}$, $1\le i\le r$,
it is enough to prove that $\Phi=\kappa=\sigma\circ\Psi$ for $r=1$.
Let $\pi\in\mathcal{NC}_{n}$, and $B_{i}$, $1\le i\le m$, be the blocks in $\pi$ such 
that $\min B_{i}<\min B_{i+1}$ for $1\le i\le m-1$.
We consider the action $\kappa$ on $\pi$.

Suppose that $\pi=12\ldots n$ be the maximal element in $\mathcal{NC}_{n}$, {\it i.e.}, 
$\pi$ consists of a single block.
We have 
$\kappa(\pi)=U(UR)^{n-1}R$.
On the other hand, we have 
$\Psi(\pi)=(UR)^{n}$.
Then, it is obvious that $\Phi=\sigma\Psi$.

We consider the case where $\pi$ has at least two blocks.
By construction of $\kappa$, each block $B_{i}$ gives a Dyck path 
$U(UR)^{l-1}R$ where $l=|B_{i}|$.
We insert the Dyck path $D_1$ corresponding to $B_{i+1}$ into 
the Dyck path $D_2$ corresponding to $B_{1\cup\cdots\cup i}:=B_{1}/B_{2}/\ldots/B_{i}$.
Recall that we insert $D_1$ right to the $2(\min B_{i+1}-1)$-th letter in $D_{2}$. 
This implies that the Dyck path $\Phi(\pi)$ is a superposition of the Dyck paths 
for the block $B_{i}$.
Similarly, the chord diagram $\Psi(\pi)$ is a superposition of chord diagrams 
corresponding to blocks $B_{i}$ in $\Psi(\pi)$.
As for a block $B_{i}$, we have $\Phi(B_{i})=\sigma\circ\Psi(B_{i})$.
Then, the superposition of the Dyck paths insures that we have 
$\Phi(\pi)=\sigma\circ\Psi(\pi)$ for general $\pi$.
This completes the proof.
\end{proof}

Let $\pi\in\mathcal{NC}_{n}^{(r)}$ be an increasing $r$-chain.
Below, we study the relation between the map $\Phi(\pi)$ and 
a cover-exclusive Dyck tiling.
We first introduce the notion of cover-exclusive Dyck tilings following \cite{Bre02,SZJ12}.
Let $\lambda_{0}$ be the Dyck path $(U^rR^r)^{n}$, and $\lambda$ be
a Dyck path above $\lambda_{0}$.
We consider a cover-exclusive Dyck tiling above $\lambda_0$ and below $\lambda$.

A {\it ribbon} is connected skew shape which does not contain a $2$-by-$2$ 
rectangle. 
A {\it Dyck tile} is a ribbon such that the centers of the unit squares in
the ribbon form a Dyck path. The size of a Dyck tile is defined to be the size 
of the Dyck path which characterizes the Dyck tile.
In pictorial representation of Dyck paths, the region $R(\lambda)$ above $\lambda_{0}$ and 
below $\lambda$ is a skew shape.
We consider a tiling of $R(\lambda)$ by Dyck tiles.
A box $(x,y)$ means that the coordinate of the box is $(x,y)$.
We say that the box $(x+1,y+1)$ (resp. $(x-1,y+1)$) is northeast (resp. northwest) of 
the box $(x,y)$. The box $(x,y+2)$ are said to be just above the box $(x,y)$.
In this paper, we consider only cover-exclusive Dyck tilings which are 
studied as Dyck tilings of type $II$ in \cite{SZJ12}.
Let $D$ and $D'$ be two Dyck tiles. 
We say that a Dyck tiling is {\it cover-exclusive} if 
the pair $(D,D')$ satisfies the following condition:
\begin{enumerate}
\item[($\diamondsuit$)]
If there exists a box of $D'$ just above, northwest, or northeast of a box of $D$, 
then all boxes just above, northwest, and northeast of a box of $D$ belong to $D$
or $D'$.
\end{enumerate}
 
\begin{example}
In Figure \ref{fig:cexDyck}, we depict twelve Dyck tilings above $(U^2R^2)^3$.
Note that the number of cover-exclusive Dyck tilings above $(U^2R^2)^3$ is 
equal to the number of $2$-Dyck paths of size $3$.
\begin{figure}[ht]
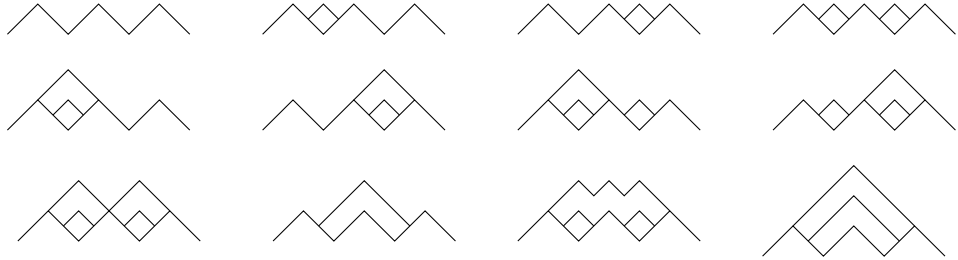

\tikzpic{-0.5}{[scale=0.2]
\draw(0,0)--(2,2)--(4,0)--(6,2)--(8,0)--(10,2)--(12,0);
}
\quad
\tikzpic{-0.5}{[scale=0.2]
\draw(0,0)--(2,2)--(4,0)--(6,2)--(8,0)--(10,2)--(12,0);
\draw(3,1)--(4,2)--(5,1);
}
\quad
\tikzpic{-0.5}{[scale=0.2]
\draw(0,0)--(2,2)--(4,0)--(6,2)--(8,0)--(10,2)--(12,0);
\draw(7,1)--(8,2)--(9,1);
}
\quad
\tikzpic{-0.5}{[scale=0.2]
\draw(0,0)--(2,2)--(4,0)--(6,2)--(8,0)--(10,2)--(12,0);
\draw(7,1)--(8,2)--(9,1)(3,1)--(4,2)--(5,1);
} \\[12pt]
\tikzpic{-0.5}{[scale=0.2]
\draw(0,0)--(2,2)--(4,0)--(6,2)--(8,0)--(10,2)--(12,0);
\draw(3,1)--(4,2)--(5,1)(2,2)--(4,4)--(6,2);
}
\quad
\tikzpic{-0.5}{[scale=0.2]
\draw(0,0)--(2,2)--(4,0)--(6,2)--(8,0)--(10,2)--(12,0);
\draw(7,1)--(8,2)--(9,1)(6,2)--(8,4)--(10,2);
}
\quad
\tikzpic{-0.5}{[scale=0.2]
\draw(0,0)--(2,2)--(4,0)--(6,2)--(8,0)--(10,2)--(12,0);
\draw(3,1)--(4,2)--(5,1);
\draw(2,2)--(4,4)--(6,2)(7,1)--(8,2)--(9,1);
}
\quad
\tikzpic{-0.5}{[scale=0.2]
\draw(0,0)--(2,2)--(4,0)--(6,2)--(8,0)--(10,2)--(12,0);
\draw(3,1)--(4,2)--(5,1);
\draw(6,2)--(8,4)--(10,2)(7,1)--(8,2)--(9,1);
} \\[12pt]
\tikzpic{-0.5}{[scale=0.2]
\draw(0,0)--(2,2)--(4,0)--(6,2)--(8,0)--(10,2)--(12,0);
\draw(3,1)--(4,2)--(5,1)(2,2)--(4,4)--(6,2);
\draw(6,2)--(8,4)--(10,2)(7,1)--(8,2)--(9,1);
}
\quad
\tikzpic{-0.5}{[scale=0.2]
\draw(0,0)--(2,2)--(4,0)--(6,2)--(8,0)--(10,2)--(12,0);
\draw(3,1)--(6,4)--(9,1);
}\quad
\tikzpic{-0.5}{[scale=0.2]
\draw(0,0)--(2,2)--(4,0)--(6,2)--(8,0)--(10,2)--(12,0);
\draw(3,1)--(4,2)--(5,1)(7,1)--(8,2)--(9,1);
\draw(2,2)--(4,4)--(5,3)--(6,4)--(7,3)--(8,4)--(10,2);
}\quad
\tikzpic{-0.5}{[scale=0.2]
\draw(0,0)--(2,2)--(4,0)--(6,2)--(8,0)--(10,2)--(12,0);
\draw(3,1)--(6,4)--(9,1);
\draw(2,2)--(6,6)--(10,2);
}
\caption{Twelve cover-exclusive Dyck tilings above $(U^2R^2)^3$}
\label{fig:cexDyck}
\end{figure}
However, the Dyck tiling above $(U^2R^2)^2$ in Figure \ref{fig:naDyck}
\begin{figure}[ht]
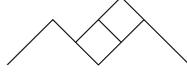

\tikzpic{-0.5}{[scale=0.3]
\draw(0,0)--(2,2)--(4,0)--(6,2)--(8,0);
\draw(3,1)--(4,2)--(5,1);
\draw(4,2)--(5,3)--(6,2);
}
\caption{Non-admissible Dyck tiling}
\label{fig:naDyck}
\end{figure}
is not cover-exclusive since it violates the condition ($\diamondsuit$).
\end{example}

Let $\pi:=(\pi_1,\ldots,\pi_r)\in\mathcal{NC}_{n}^{(r)}$ be an increasing $r$-chain.
We recursively construct a sequence of cover-exclusive Dyck tilings $D_{i}$, $1\le i\le r$:
\begin{align*}
D_{r+1}\xrightarrow{\pi_r}D_{r}\xrightarrow{\pi_{r-1}}D_{r-1}\xrightarrow{\pi_{r-2}}
\cdots \xrightarrow{\pi_{1}}D_{1},
\end{align*}
where $D_{r+1}$ is the cover-exclusive Dyck tiling above $\lambda_{0}$ without Dyck tiles.
Here, $D_{i+1}\xrightarrow{\pi_{i}}D_{i}$ means that $D_{i}$ is obtained from 
$D_{i+1}$ and $\pi_{i}$.

Suppose that $\pi_{i}$ consists of $m$ blocks $B_{t}$, $1\le t\le m$, such that 
$\max B_{t}<\max B_{t+1}$ for $1\le t\le m-1$.
Further, suppose that a block $B_{t}$ consists of $p$ positive increasing integers 
$(q_{1},\ldots,q_{p})$.
By decomposing $\pi_i$ into $m$ blocks $B_{t}$, $1\le t\le m$, we obtain a refined 
sequence of cover-exclusive Dyck tilings
\begin{align*}
D_{i+1}=D'_{0}\xrightarrow{B_1}D'_{1}\xrightarrow{B_2}D'_{2}\cdots \xrightarrow{B_{m}}D'_{m}=D_{i}.
\end{align*}
Given a Dyck tiling $D'_{t-1}$ and a block $B_{t}$ in $\pi_{i}$, $D'_{t}$ is obtained from 
$D'_{t-1}$ by adding Dyck tiles which connect 
$r(q_s-1)+r-i+1$-th down step and $r(q_{s+1}-1)+i$-th up step in $\lambda_{0}$ for all $1\le s\le p-1$.
Since $D'_{t-1}$ and $D'_{t}$ are cover-exclusive Dyck tilings, the added Dyck tiles are uniquely 
fixed by giving two up and down steps which are the end steps of a Dyck tile.

\begin{example}
Let $\pi:=(14/23,1234)$ be a $2$-chain in $\mathcal{NC}_{4}$.
The non-crossing partition $1234$ consists of a single block, and $14/23$ consists of 
two blocks $23$ and $14$.
The cover-exclusive Dyck tiling for $\pi$ is given as follows.
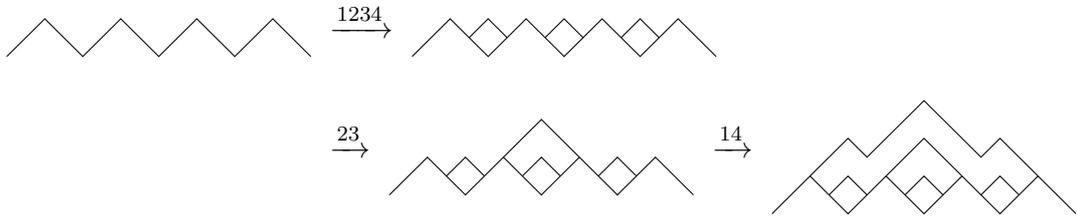
\begin{align*}
\tikzpic{-0.5}{[scale=0.25]
\draw(0,0)--(2,2)--(4,0)--(6,2)--(8,0)--(10,2)--(12,0)--(14,2)--(16,0);
}&\xrightarrow{1234}
\tikzpic{-0.5}{[scale=0.25]
\draw(0,0)--(2,2)--(4,0)--(6,2)--(8,0)--(10,2)--(12,0)--(14,2)--(16,0);
\draw(3,1)--(4,2)--(5,1)(7,1)--(8,2)--(9,1)(11,1)--(12,2)--(13,1);
} \\[12pt]
&\xrightarrow{23}
\tikzpic{-0.5}{[scale=0.25]
\draw(0,0)--(2,2)--(4,0)--(6,2)--(8,0)--(10,2)--(12,0)--(14,2)--(16,0);
\draw(3,1)--(4,2)--(5,1)(7,1)--(8,2)--(9,1)(11,1)--(12,2)--(13,1);
\draw(6,2)--(8,4)--(10,2);
}\xrightarrow{14}
\tikzpic{-0.5}{[scale=0.25]
\draw(0,0)--(2,2)--(4,0)--(6,2)--(8,0)--(10,2)--(12,0)--(14,2)--(16,0);
\draw(3,1)--(4,2)--(5,1)(7,1)--(8,2)--(9,1)(11,1)--(12,2)--(13,1);
\draw(6,2)--(8,4)--(10,2);
\draw(2,2)--(4,4)--(5,3)--(8,6)--(11,3)--(12,4)--(14,2);
}
\end{align*}
\end{example}

Let $\pi\in\mathcal{NC}_{n}^{(r)}$ and $D$ be a cover-exclusive Dyck tiling for $\pi$. 
\begin{defn}
We denote by $\lambda(\pi)$ the top Dyck path of the cover-exclusive Dyck tiling $D$ 
for $\pi\in\mathcal{NC}_{n}^{(r)}$.
\end{defn}

By Section \ref{sec:chord}, we have a correspondence between a chord diagram 
and a Dyck path. Further, we can identify a chord diagram $\mathcal{C}_{n}^{(r)}$ with a chord 
diagram $\mathcal{C}_{rn}$. 
By use of these two correspondences, we identify a Dyck path $\lambda(\pi)$ with a chord 
diagram in $\mathcal{C}_{n}^{(r)}$.
We first show that this identification is well-defined.
\begin{lemma}
The chord diagram $\lambda(\pi)$ in $\mathcal{C}_{rn}$ gives a chord diagram 
in $\mathcal{C}_{n}^{(r)}$ by the inclusion $\mathcal{C}_{n}^{(r)}\hookrightarrow\mathcal{C}_{rn}$.
In other words, the chord diagram $\lambda(\pi)$ satisfies the condition (\ref{eq:condA}).
\end{lemma}
\begin{proof}
Let $\lambda_{0}$ be the Dyck path $(U^rR^r)^{n}$, that is, $\lambda_{0}$ is the lowest Dyck path.
The chord diagram for $\lambda_{0}$ obviously satisfies the 
condition (\ref{eq:condA}). 
By recursive construction of a cover-exclusive Dyck tiling, a Dyck tile connects a down step with 
an up step in $\lambda_{0}$. Note that a Dyck tile corresponds to an arch connecting a primed 
integer and an integer. This means that $\lambda(\pi)$ also satisfies the condition (\ref{eq:condA}). 
This completes the proof.
\end{proof}

\begin{prop}
\label{prop:Philambda}
Let $\pi:=(\pi_1,\ldots,\pi_{r})$ and $\lambda(\pi)$ be as above.
Then, we have $\Phi(\pi)=\lambda(\pi)$ as a chord diagram in $\mathcal{C}_{n}^{(r)}$.
\end{prop}
\begin{proof}
Suppose that $\pi_i$ consists of $m$ blocks $B_1,\ldots, B_{m}$, and a block $B_{t}$,
$1\le t\le m$ consists of $p$ positive integers $(q_1,\ldots,q_{p})$.
By the recursive construction of a Dyck tiling, we have a Dyck tile connecting 
the $r(q_{s}-1)+r-i+1$-th down step with the $r(q_{s+1}-1)+i$-th up step in $\lambda_{0}$.
In terms of a chord diagram, this Dyck tile corresponds to an arch connecting a primed integer 
$q_{s}'$ with an integer $q_{s+1}$.
If we rotate the chord diagram by $(\sigma^{(r)})^{-1}$, the new arch connects the integer 
$q_s$ with a primed integer $(q_{s+1}-1)'$.
We can obtain this new arch by the action of $\Psi$ on $\pi_{i}$.
From Proposition \ref{prop:PhiPsi}, we have $\Phi(\pi)=\lambda(\pi)$.
\end{proof}

\begin{cor}
The number of cover-exclusive Dyck tilings above $(U^rR^r)^{n}$ is given by 
the Fuss--Catalan number.
\end{cor}
\begin{proof}
From Proposition \ref{prop:Philambda}, we have $\Phi(\pi)=\lambda(\pi)$.
The number of generalized chord diagrams in $\mathcal{C}_{n}^{(r)}$ is given by 
the Fuss--Catalan number by Proposition \ref{prop:NCC}. The number of cover-exclusive 
Dyck tilings $\lambda(\pi)$ is equal to the number of generalized chord diagrams.
Thus, it is given by the Fuss--Catalan number.
\end{proof}

\subsection{Definition of the generators \texorpdfstring{$F_{i}^{(s)}$}{Fi}}
Recall that we have an operator $f_1$ on $\mathcal{NC}_{n}$.
We construct the generators $F_{i}^{(s)}$, $1\le i\le 2n-1$, $1\le s\le r$, 
by use of $f_{1}$.

Let $\pi^{(r)}=(\pi_1,\ldots,\pi_{r})\in\mathcal{NC}_{n}^{(r)}$.
We first define $f_{1}^{(s)}$, $1\le s\le r$, on $\mathcal{NC}_{n}^{(r)}$ by 
\begin{align}
\label{def:fis}
f_{1}^{(s)}\pi^{(r)}&:=
(\pi_1,\ldots,\pi_{r-s},f_1(\pi_{r-s+1}),f_1(\pi_{r-s+2}),\ldots,f_{1}(\pi_{r})), 
\end{align}
where the operator $f_1$ in the right hand side of Eq. (\ref{def:fis}) is defined 
in Eq. (\ref{def:f}) on $\mathcal{NC}_{n}$.
Then, we define the generators $F_{i}^{(s)}$, $1\le i\le 2n-1$, $1\le s\le r$, by 
\begin{align}
\label{def:F}
\begin{aligned}
F_{i}^{(s)}&:=\rho^{i-1} F_{1}^{(s)}\rho^{-(i-1)},
\end{aligned}
\end{align}
where $F_{1}^{(s)}:=f_{1}^{(s)}$ and $\rho$ is the extended Kreweras endomorphism
defined by Eq. (\ref{eq:FKre}).

\begin{defn}
The unital associative $\mathbb{C}[q,q^{-1}]$-algebra $\mathbb{NC}_{n}^{(r)}$ acting on $\mathcal{NC}_{n}^{(r)}$ 
is generated by $\{F_{i}^{(s)}: 1\le i\le 2n-1, 1\le s\le r\}$.
We call $\mathbb{NC}_{n}^{(r)}$ the Fuss--Catalan algebra.
\end{defn}

Since we have a bijection $\kappa^{(r)}$ from $\mathcal{NC}_{n}^{(r)}$ to 
$\mathcal{P}_{n}^{(r)}$, 
one can consider the representations of the algebra $\mathbb{NC}_{n}^{(r)}$ on generalized 
Dyck paths.

\section{Diagrammatic representation}
\label{sec:Pic}
In Section \ref{sec:TLchord}, a diagrammatic representation of the Temperley--Lieb 
algebra $\mathbb{TL}_{n}$ is given.
In this section, we introduce a diagrammatic representation of the 
algebra $\mathbb{TL}_{n}^{(r)}$ with $r\ge1$.
This algebra $\mathbb{TL}_{n}^{(r)}$ is called {\it Fuss--Catalan algebra} 
in \cite{BisJon97,Dif98}.
The algebra $\mathbb{TL}_{n}^{(r)}$ acts on the set $\mathcal{C}_{n}^{(r)}$ 
of generalized chord diagrams.

We introduce the following diagrammatic representation of $s$ strands:
\begin{align}
\label{eq:strand}
\tikzpic{-0.5}{
\node(1)at(0,-1){\framebox{$s$}};
\draw(0,0.3)--(0,-0.8)(0,-1.2)--(0,-2.3);
}:=
\tikzpic{-0.42}{
\draw(0,0.3)--(0,-2.3)(0.7,0.3)--(0.7,-2.3);
\draw(0.35,-1)node{$\cdots$};
\draw[decoration={brace},decorate](0,0.35)to[anchor=south]node{$s$}(0.7,0.35);
}
\end{align}
In other words, the left hand side of Eq. (\ref{eq:strand}) represents  
$s$ vertical strands bundled into one.

We define the diagrammatic representation of the generator $E^{(s)}_{i}$ by 
\begin{align}
\label{eq:picFis}
E_{i}^{(s)}:=
\tikzpic{-0.5}{
\node(1)at(0,-1){\framebox{$s'$}};
\draw(0,0.3)--(0,-0.74)(0,-1.26)--(0,-2.3);
\node(2)at(1.2,-0.7){\framebox{$s$}};
\draw(0.45,0.3)..controls(0.45,-0.2)and(0.69,-0.75)..(0.99,-0.75);
\draw(1.4,-0.75)..controls(1.7,-0.75)and(1.95,-0.2)..(1.95,0.3);
\node(3)at(1.2,-1.3){\framebox{$s$}};
\draw(0.45,-2.3)..controls(0.45,-1.8)and(0.7,-1.3)..(1,-1.3);
\draw(1.4,-1.3)..controls(1.7,-1.3)and(1.95,-1.8)..(1.95,-2.3);
\node(4)at(2.4,-1){\framebox{$s'$}};
\draw(2.4,0.3)--(2.4,-0.74)(2.4,-1.26)--(2.4,-2.3);
\draw(0.225,-2.3)node[anchor=north]{$i$};
\draw(2.2,-2.3)node[anchor=north]{$i+1$};
}
\end{align}
where $s'=r-s$. The action of $E_{i}^{(s)}$ is local, we have vertical bundled strands 
at the positions except $i$ and $i+1$.

Let $X$ and $Y$ be diagrams. Then, the product $XY$ is calculated by placing 
the diagram $Y$ on top of the diagram $X$.
If we have a closed loop, we remove it and give a factor $\tau:=-(q+q^{-1})$.
We regard two diagrams equivalent if they are isotropic to each other. 

\begin{defn}
The algebra $\mathbb{TL}_{n}^{(r)}$ is a unital associative $\mathbb{C}[q,q^{-1}]$-algebra
generated by the set of generators $\{E_{i}^{(s)} : 1\le i\le 2n-1, 1\le s\le r\}$. 
\end{defn}

The algebra $\mathbb{TL}_{n}^{(r)}$ acts on the generalized chord diagrams 
in $\mathcal{C}_{n}^{(r)}$.

\begin{prop}
Let $C\in\mathcal{C}_{n}^{(r)}$, and $X$ be an element in $\mathbb{TL}_{n}^{(r)}$.
Then, the action of $X$ on $\mathcal{C}_{n}^{(r)}$ is well-defined, {\it i.e.}, 
$XC\in\mathcal{C}_{n}^{(r)}$.
\end{prop}
\begin{proof}
It is enough to show that $E_{i}^{(s)}C\in\mathcal{C}_{n}^{(r)}$.
Recall that $C$ satisfies the condition (\ref{eq:condA}).
We act $E_{i}^{(s)}$ from the bottom of $C$, and we obtain a new chord 
diagram $C'$.
The diagrammatic representation (\ref{eq:picFis}) of $E_{i}^{(s)}$ insures that 
$C'$ also satisfies the condition (\ref{eq:condA}).
To see this, note that the strands in $E_{i}^{(s)}$ are vertical lines or arches of size one, 
and the arches also satisfy the condition (\ref{eq:condA}).
From these, we have $C'\in\mathcal{C}_{n}^{(r)}$.
\end{proof}

The generators $\{E_{i}^{(s)}\}$ satisfy the following relations.
\begin{theorem}
\label{thrm:relF}
The set of operators $\{E_{i}^{(s)}: 1\le i\le 2n-1, 1\le s\le r\}$ 
generates the following relations of order up to three:
\begin{align*}
&E_{i}^{(s)}E_{i}^{(s')}=\tau^{\min\{s,s'\}}E_{i}^{\max\{s,s'\}}, \qquad 1\le i\le 2n-1,\  1\le s,s'\le r, \\
&E_{i}^{(s)}E_{j}^{(s')}=E_{j}^{(s')}E_{i}^{(s)}, \qquad |i-j|>1, \ 1\le s,s'\le r, \text{ or } |i-j|=1,\ s+s'\le r,\\
&E_{i}^{(s)}E_{i+1}^{(s')}E_{i}^{(s'')}=
\begin{cases}
\tau^{r-s'}E_{i}^{(r)}, & \text{ if } s=s''=r, \\
\tau^{r-s'}E_{i+1}^{(r-s)}E_{i}^{(s'')}, & \text{ if } s\le s''<r, \text{ and  } s+s'\ge r,\\
\tau^{r-s'}E_{i}^{(s)}E_{i+1}^{(r-s'')}, & \text{ if } r>s\ge s'', \text{ and  } s''+s\ge r,
\end{cases}\\
&E_{i+1}^{(s)}E_{i}^{(s')}E_{i+1}^{(s'')}=
\begin{cases}
\tau^{r-s'}E_{i+1}^{(r)}, & \text{ if } s=s''=r,\\
\tau^{r-s'}E_{i}^{(r-s)}E_{i+1}^{(s'')}, & \text{ if } s\le s''<r, \text{ and } s+s'\ge r, \\
\tau^{r-s'}E_{i+1}^{(s)}E_{i}^{(r-s'')}, & \text{ if } r>s\ge s'', \text{ and } s''+s'\ge r,
\end{cases}
\end{align*}
where $E_{i}^{(0)}:=\mathbf{1}$.
\end{theorem}
\begin{proof}
It is a routine to check that the diagram (\ref{eq:picFis}) satisfies the relations 
of order up to three.
Note that if we have closed loops in a diagram, we remove it and 
give the overall factor $\tau$ for each closed loop.
\end{proof}

In general, we have relations which involve many generators in $\mathbb{TL}_{n}^{(r)}$.
For example, we consider the case $(n,r)=(3,2)$.
Then, we have 
\begin{align*}
E_2^{(2)}E_{3}^{(1)}E_{1}^{(2)}E_{2}^{(2)}
=E_2^{(2)}E_{3}^{(2)}E_{1}^{(1)}E_{2}^{(2)}
=\tau^{-1}E_2^{(2)}E_{3}^{(1)}E_{1}^{(1)}E_{2}^{(2)}.
\end{align*}
They are equivalent to the following diagram:
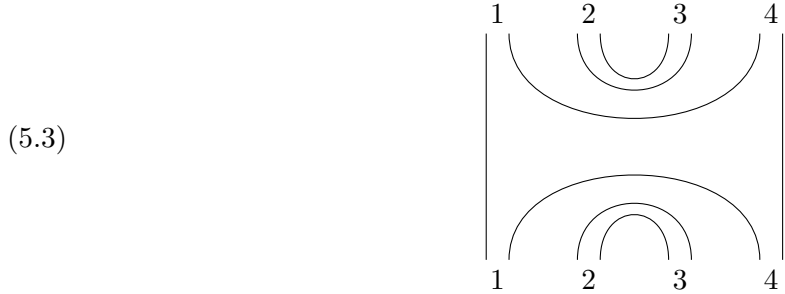
\begin{align}
\label{eq:2312}
\tikzpic{-0.5}{[xscale=0.6]
\draw(0,0)--(0,-3);
\draw(0.5,0)..controls(0.5,-1.5)and(6,-1.5)..(6,0);
\draw(2,0)..controls(2,-1)and(4.5,-1)..(4.5,0);
\draw(2.5,0)..controls(2.5,-0.8)and(4,-0.8)..(4,0);
\draw(6.5,0)--(6.5,-3);
\draw(0.5,-3)..controls(0.5,-1.5)and(6,-1.5)..(6,-3);
\draw(2,-3)..controls(2,-2)and(4.5,-2)..(4.5,-3);
\draw(2.5,-3)..controls(2.5,-2.2)and(4,-2.2)..(4,-3);
\draw(0.25,-3)node[anchor=north]{$1$}(2.25,-3)node[anchor=north]{$2$}
(4.25,-3)node[anchor=north]{$3$}(6.25,-3)node[anchor=north]{$4$};
\draw(0.25,0)node[anchor=south]{$1$}(2.25,0)node[anchor=south]{$2$}
(4.25,0)node[anchor=south]{$3$}(6.25,0)node[anchor=south]{$4$};
}
\end{align}
Any element in $\mathbb{TL}_{n}^{(r)}$ can be calculated and visualized through 
the diagrammatic representation (\ref{eq:picFis}).

An element of $\mathbb{TL}_{n}^{(r)}$ is depicted as a diagram 
consisting of $rn$ strands connecting $n$ top points and $n$ bottom points.
Each top or bottom point has $r$ strands which are non-crossing and satisfies 
the condition (\ref{eq:condA}) if we fold the diagram down to the right.  
Therefore, the dimension of $\mathbb{TL}_{n}^{(r)}$ is equal to the 
number of such diagrams.

\begin{prop}
\label{prop:TLP}
We have 
\begin{align}
\label{eq:TLP}
\dim(\mathbb{TL}_{n}^{(r)})=|\mathcal{P}_{n+1}^{(r)}|=\genfrac{}{}{}{}{1}{r(n+1)+1}\genfrac{(}{)}{0pt}{}{(r+1)(n+1)}{n+1}.
\end{align}
\end{prop}
\begin{proof}
Recall that an element $X$ in $\mathbb{TL}_{n}^{(r)}$ is a diagram of $r(n+1)$ strands connecting 
$n+1$ bottom points and $n+1$ top points. Each bottom or top point has $r$ strands. 
By folding the diagram for $X$ down to the right 
such that the $n+1$ top points are right to the $n+1$ bottom points. 
For example, the diagram (\ref{eq:2312}) corresponds to the following diagram:
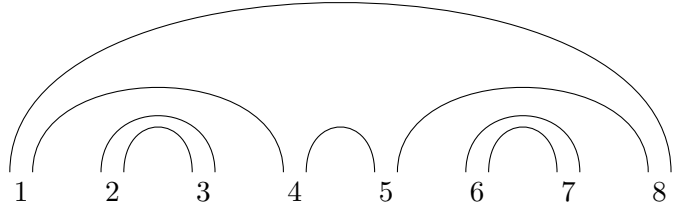
\begin{align*}
\tikzpic{-0.5}{[xscale=0.6]
\draw(0,-3)..controls(0,0)and(14.5,0)..(14.5,-3);
\draw(0.5,-3)..controls(0.5,-1.5)and(6,-1.5)..(6,-3);
\draw(2,-3)..controls(2,-2)and(4.5,-2)..(4.5,-3);
\draw(2.5,-3)..controls(2.5,-2.2)and(4,-2.2)..(4,-3);
\draw(8.5,-3)..controls(8.5,-1.5)and(14,-1.5)..(14,-3);
\draw(10,-3)..controls(10,-2)and(12.5,-2)..(12.5,-3);
\draw(10.5,-3)..controls(10.5,-2.2)and(12,-2.2)..(12,-3);
\draw(6.5,-3)..controls(6.5,-2.2)and(8,-2.2)..(8,-3);
\draw(0.25,-3)node[anchor=north]{$1$}(2.25,-3)node[anchor=north]{$2$}
(4.25,-3)node[anchor=north]{$3$}(6.25,-3)node[anchor=north]{$4$}
(8.25,-3)node[anchor=north]{$5$}(10.25,-3)node[anchor=north]{$6$}
(12.25,-3)node[anchor=north]{$7$}(14.25,-3)node[anchor=north]{$8$};
}
\end{align*}
The folded diagram for $X$ satisfies the property (\ref{eq:condA}).
The number of such folded diagrams is equal to the number of chord 
diagrams in $\mathcal{C}_{n+1}^{(r)}$.
From Proposition \ref{prop:NCP} and Proposition \ref{prop:NCC}, 
Eq. (\ref{eq:TLP}) follows.
\end{proof}

\begin{prop}
\label{prop:isoNCTL}
We have an isomorphism between $\mathbb{NC}_{n}^{(r)}$ and $\mathbb{TL}_{n}^{(r)}$, 
{\it i.e.}, $F_{i}^{(s)}\mapsto E_{i}^{(s)}$.
\end{prop}
\begin{proof}
From Proposition \ref{prop:Fe}, we have an isomorphism between 
$\mathbb{NC}_{n}^{(1)}$ and $\mathbb{TL}_{n}^{(1)}$.
Especially, we have $F_i=\Psi^{-1}e_{i}\Psi$ by Eq. (\ref{rel:Fe}).
For $r\ge2$, let $\pi^{(r)}=(\pi_1,\ldots,\pi_r)\in\mathcal{NC}_{n}^{(r)}$.
The generator $F_{i}^{(s)}$ acts on $\pi_1,\ldots,\pi_{r-s}$ as an identity, 
and acts on $\pi_{r-s+1},\ldots,\pi_r$ as $f_{i}$.
On the other hand, the generator $E_i^{(s)}$ also acts on $r-s$ strands 
as an identity.
By the construction of the bijection $\Psi^{(r)}$ between $\mathcal{NC}_{n}^{(r)}$ 
and $\mathcal{C}_{n}^{(r)}$, the action of $F_{i}^{(s)}$ on 
each non-crossing partition $\pi_{i}$, $1\le i\le r$, in an $r$-chain is equivalent
to the action of $F_{i}\in\mathbb{NC}_{n}^{(1)}$ or the identity.
Recall that if we regard an element in $\mathcal{C}_{n}^{(r)}$ as an element in $\mathcal{C}_{rn}$,
the rotation $\rho^{(r)}$ of $\mathcal{C}_{n}^{(r)}$ is given by $\rho^{r}$ where $\rho$ is the rotation
of $\mathcal{C}_{rn}$.
From these, we have $F_{i}^{(s)}={\Psi^{(r)}}^{-1}E_{i}^{(s)}\Psi^{(r)}$.
Note that $\Psi^{(r)}$ is a bijection between $\mathcal{NC}_{n}^{(r)}$ and $\mathcal{C}_{n}^{(r)}$.
This completes the proof.
\end{proof}

\section{Symmetric non-crossing partitions}
\label{sec:SNC}
\subsection{Symmetric non-crossing partitions}
A {\it symmetric Dyck path} of size $n$ is a Dyck path of size $n$ which 
is symmetric along the line $y=-x+n$.
In the case of $r=1$, we have an identification between a Dyck path 
and a chord diagram. 
This identification implies that the chord diagram of a symmetric Dyck path is
symmetric along the vertical line in the middle.
For $r\ge2$, the numbers of up and right steps in an $r$-Dyck path are different.
This prevents us to naively define symmetric $r$-Dyck paths.
There are two ways to define symmetric $r$-Dyck paths.
The first one is to define a symmetric $r$-Dyck path as a symmetric 
generalized chord diagram.
The second one is to define a symmetric $r$-Dyck path as a symmetric 
non-crossing partition.
In what follows, we introduce the notions of symmetric generalized chord diagrams
and symmetric non-crossing partitions, and 
study the relation between them.

A {\it symmetric generalized chord diagram} in $\mathcal{C}_{n}^{(r)}$ is a chord diagram 
in $\mathcal{C}_{n}^{(n)}$ such that it is symmetric along the vertical line in the middle.

\begin{defn}
\label{defn:Bn}
We denote by $\mathcal{SC}_{n}^{(r)}$ the set of symmetric generalized chord diagrams in 
$\mathcal{C}_{n}^{(r)}$.
We define $B_{n}^{(r)}:=|\mathcal{SC}_{n}^{(r)}|$.
\end{defn}

For example, we have $B_{2}^{(2)}=3$, $B_{3}^{(2)}=6$ and $B_{4}^{(2)}=17$.

Since we have a bijection between a generalized chord diagram and an $r$-Dyck path via 
non-crossing partitions, one can identify a symmetric $r$-Dyck path with a symmetric 
non-crossing partition which is bijective to a symmetric generalized chord diagram.
For this purpose, we first define a symmetric non-crossing partition.

We introduce the notion of symmetric non-crossing partitions in $\mathcal{NC}_{n}$.
Let $\epsilon\in\{0,1\}$.
A non-crossing partition $\pi\in\mathcal{NC}_{n}$ is said to be symmetric 
if the two integers $i$ and $j$ belong to a block $B_1$, then the integers 
$n+2-\epsilon-i$ and $n+2-\epsilon-j$ also belong to a block $B_2$.
Here, we identify the integer $n+1$ with the integer $1$, and the two blocks 
$B_1$ and $B_2$ may be the same.

\begin{defn}
We denote by $\mathcal{SNC}_{n}$ the set of symmetric non-crossing partitions of 
size $n$.
\end{defn}

In the case where $n$ is even, we have two sets of symmetric non-crossing partitions according 
to the choice of $\epsilon\in\{0,1\}$. 
Since we have a bijection between a non-crossing partition and a Dyck path,
this bijection defines a symmetric Dyck path induced from a symmetric 
non-crossing partition.
If $\epsilon=0$, we have the following bijection.

\begin{prop}
\label{prop:SNCtoSC}
Let $\epsilon=0$ and $\Psi:\mathcal{NC}_{n}\xrightarrow{\sim}\mathcal{C}_{n}$ be the bijection introduced 
in Section \ref{sec:TLNC}.
The restriction map $\Psi:\mathcal{SNC}_{n}\rightarrow\mathcal{C}_{n}$ is 
a bijection between $\mathcal{SNC}_{n}$ and $\mathcal{SC}_{n}$.
\end{prop}
\begin{proof}
We fist prove that $\Psi(\pi)\in\mathcal{SC}_{n}$ with $\pi\in\mathcal{SNC}_{n}$.
We consider the two cases: 1) $n$ even, and 2) $n$ odd.

\paragraph{Case 1)}
Since $n$ is even, the middle point of a chord diagram is between $(n/2)'$ and $(n/2+1)$. 
Suppose that a block $B$ of $\pi$ consists of integers $(p_1,p_2,\ldots,p_{m})$ with $m\ge1$. 
Since $\epsilon=0$ and $\pi$ is symmetric, $\Psi(\pi)$ has arches connecting $p_{i}$ and $(p_{i+1}-1)'$, and 
$n+2-p_{i+1}$ and $(n+2-p_{i}-1)'$. These two arches are symmetric under the vertical line 
in the middle. Since all arches are obtained in this way, $\Psi(\pi)$ is symmetric.

\paragraph{Case 2)}
Since $n$ is odd, the middle point of a chord diagram is between $(n+1)/2$ and $((n+1)/2)'$.
By the similar argument to Case 1), one can show that $\Psi(\pi)$ is symmetric.
 
From these, we have $\Psi(\pi)\in\mathcal{SC}_{n}$ if $\pi\in\mathcal{SNC}_{n}$.
Conversely, suppose that $C\in\mathcal{SC}_{n}$. By reversing the arguments in Case 1) and 2),
it is easy to see that $\pi=\Psi^{-1}(C)$ is a symmetric non-crossing partition.
Therefore, the restriction of the map $\Psi$ is a bijection between $\mathcal{SNC}_{n}$
and $\mathcal{SC}_{n}$. 
\end{proof}

The set $\mathcal{SNC}_{n}$ forms a subposet in $\mathcal{NC}_{n}$.
In the set $\mathcal{SNC}_{n}$, $\pi_2$ covers $\pi_1$ if we have
\begin{align*}
\pi_1=\nu_0\lessdot \nu_1\lessdot \cdots \lessdot\nu_{r}\lessdot \nu_{r+1}=\pi_{2},
\end{align*}
such that $\nu_{i}\in\mathcal{NC}_{n}\setminus \mathcal{SNC}_{n}$ for all $i\in[1,r]$.
Therefore, $\pi_1\lessdot\pi_2$ in $\mathcal{SNC}_{n}$ does not mean 
$\pi_1\lessdot\pi_2$ in $\mathcal{NC}_{n}$ in general.
Nevertheless, we have $\pi_1<\pi_2$ in $\mathcal{NC}_{n}$ if $\pi_1\lessdot\pi_2$ in $\mathcal{SNC}_{n}$.

We show the Hasse diagrams of the two posets $\mathcal{SNC}_{n}$ with $n=4$ ($\epsilon=1$) 
and $n=5$ ($\epsilon=0$) in Figure \ref{fig:symNC}.
\begin{figure}[ht]
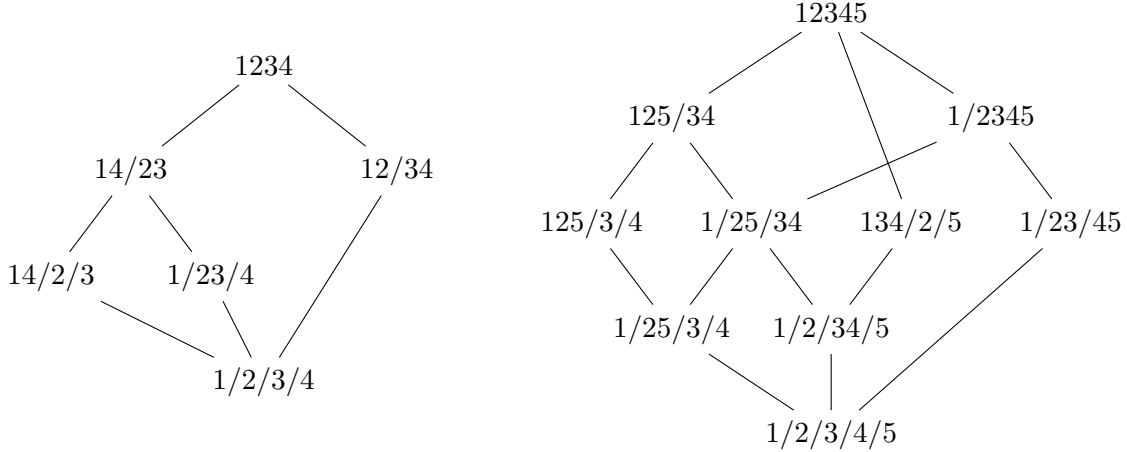

\tikzpic{-0.5}{[scale=0.7]
\node(0)at(0,0){$1/2/3/4$};
\node(1)at(-4,2){$14/2/3$};
\node(2)at(-1,2){$1/23/4$};
\node(3)at(-2.5,4){$14/23$};
\node(4)at(2.5,4){$12/34$};
\node(5)at(0,6){$1234$};
\draw(0)to(1)to(3)(0)to(2)to(3)to(5)(0)to(4)to(5);
}\qquad
\tikzpic{-0.5}{[scale=0.7]
\node(0)at(0,0){$1/2/3/4/5$};
\node(1)at(-3,2){$1/25/3/4$};
\node(2)at(0,2){$1/2/34/5$};
\node(3)at(-4.5,4){$125/3/4$};
\node(4)at(-1.5,4){$1/25/34$};
\node(5)at(1.5,4){$134/2/5$};
\node(6)at(4.5,4){$1/23/45$};
\node(7)at(-3,6){$125/34$};
\node(8)at(3,6){$1/2345$};
\node(9)at(0,8){$12345$};
\draw(0)to(1)to(3)to(7)to(9)(0)to(2)to(4)to(7)(2)to(5)
(0)to(6)to(8)to(9)(1)to(4)to(8)(5)to(9);
}
\caption{The Hasse diagrams of symmetric non-crossing partitions for $n=4$ ($\epsilon=1$) 
and $n=5$ ($\epsilon=0$).}
\label{fig:symNC}
\end{figure}
The non-crossing partition $\pi_1:=12/34$ covers $\pi_2:=1/2/3/4$ in $\mathcal{SNC}_{4}$, however,
$\pi_1$ does not cover $\pi_2$ in $\mathcal{NC}_{4}$.

\begin{prop}
The number $|\mathcal{SNC}_{n}|$ of symmetric non-crossing partitions 
is given by 
\begin{align}
\label{eq:SNCn}
|\mathcal{SNC}_{n}|=\genfrac{(}{)}{0pt}{}{n}{\lfloor n/2\rfloor},
\end{align}
where $\lfloor x \rfloor$ is the floor function.
\end{prop}
\begin{proof}
We consider the case where $\epsilon=1$ for $n$ even and $\epsilon=0$ for $n$ odd.
Let $A(n):=|\mathcal{SNC}_{n}|$, and $C(n)$ be the $n$-th Catalan number, 
{\it i.e.}, $C(n)=\genfrac{}{}{}{}{1}{n+1}\genfrac{(}{)}{0pt}{}{2n}{n}$.

We first consider the case where $n=2m+1$ with $m\ge0$.
If the integer $1$ forms a block consisting of only $1$, the total number of 
such non-crossing partitions are $A(2m)$.
Suppose that the integer $1$ is contained in a block $B$ such that $|B|\ge2$.
Let $j\ge2$ be the second smallest integer in $B$. Since we consider 
a symmetric non-crossing partition, the integer $n+2-j$ is also contained in $B$.
Since the integers in $[2,j-1]$ can form a non-crossing partitions, we have 
$C(j-2)$ such non-crossing partitions. Since the set $[j+1,n+1-j]\cup\{1\}$ forms 
a symmetric non-crossing partition, the number of such symmetric non-crossing partition 
is $A(n+2-2j)=A(2m-2j+3)$.
Combining these observations, we arrive at the following recurrence relation: 
\begin{align}
\label{eq:Aodd}
A(2m+1)=A(2m)+\sum_{j=2}^{m+1}C(j-2)A(2m-2j+3).
\end{align}

Secondly, we consider the case where $n=2m$ with $m\ge1$.
Suppose that the integer $1$ forms a block consisting of only $1$.
Then, the symmetry implies that the integer $n$ also forms a block 
consisting of only $n$. 
The number of such symmetric non-crossing partitions is given by $A(2m-2)$.
Suppose that $1$ and $n$ belongs to the same block $B$ such that $|B|\ge2$.
Let $j\ge2$ be the second smallest integer in $|B|$.
Since $1$ and $j$ belong to the same block, the integers in $[2,j-1]$ can 
form a non-crossing partition. The number of such partitions is given by 
$C(j-2)$. The set $[j+1,n-j]$ and $B$ form a symmetric non-crossing partition, and 
the number of such non-crossing partitions is $A(n-2j+1)=A(2m-2j+1)$.
From these, we have 
\begin{align}
\label{eq:Aev}
A(2m)=A(2m-1)+A(2m-2)+\sum_{j=2}^{m}C(j-2)A(2m-2j+1).
\end{align}
From Eqs. (\ref{eq:Aodd}) and (\ref{eq:Aev}), we have 
\begin{align}
\label{eq:Aev2}
A(2m)=2A(2m-1).
\end{align}

We will express $A(2m+1)$ in terms of $A(2m)$.
If $1$ forms a block consisting of only $1$, we have $A(2m)$ symmetric
non-crossing partitions.
Suppose $1$ is contained a block $B$ such that $|B|\ge2$. Then, if we delete $1$ 
from such a symmetric non-crossing partition, we have $A(2m)-C(m)$ symmetric 
non-crossing partitions. The $C(m)$ non-crossing partitions come from the fact 
that if $[2,m+1]$ and $[m+2,2m+1]$ form a non-crossing partitions of size $m$, then 
there is no arch connecting an integer in $[2,m+1]$ with an integer $[m+2,2m+1]$.
From these, we have 
\begin{align}
\label{eq:Aodd2}
A(2m+1)=2A(2m)-C(m).
\end{align}

One can easily show Eq. (\ref{eq:SNCn}) by use of Eqs. (\ref{eq:Aev2}) and (\ref{eq:Aodd2}),
and induction on $n$.
\end{proof}

\begin{remark}
\label{remark:epsilon}
When we define $\mathcal{SNC}_{n}$, we chose $\epsilon=1$ if $n$ is even.
We can choose $\epsilon=0$ as well. 
The number of symmetric non-crossing partitions for $\epsilon=0$ is equal 
to that for $\epsilon=1$ since the Kreweras endomorphism $\rho$ maps 
from a symmetric non-crossing partition for $\epsilon=0$ to 
that for $\epsilon=1$, and vice versa.
Therefore, $\rho$ is a bijection between the two sets.
\end{remark}

In what follows, we mainly consider the case of $\epsilon=0$.
In this case, a symmetric non-crossing partition corresponds to a symmetric chord diagram
through the bijection $\Psi$.
In the case of $\epsilon=1$, a chord diagram corresponding to a symmetric non-crossing 
partition is not symmetric along the vertical line. 
However, by Remark \ref{remark:epsilon}, the Kreweras endomorphism $\rho$ gives
a bijection between the sets of chord diagrams for $\epsilon=0$ and $\epsilon=1$.
One can translate the results for $\epsilon=0$ to the case of $\epsilon=1$.
Therefore, we also call a chord diagram for a symmetric non-crossing partition in the 
case of $\epsilon=1$ a symmetric chord diagram.

An increasing $r$-chain $\pi^{(r)}:=(\pi_1,\ldots,\pi_r)$ of $\mathcal{SNC}_{n}$ is 
a sequence of symmetric non-crossing partitions such that 
$\pi_1\le \pi_2\le\ldots\le\pi_{r}$ in $\mathcal{SNC}_{n}$.
We denote by $\mathcal{SNC}_{n}^{(r)}$ the set of $r$-chains of $\mathcal{SNC}_{n}$.

\begin{prop}
\label{prop:SNCBn}
We have $|\mathcal{SNC}_{n}^{(r)}|=B_{n}^{(r)}$.
\end{prop}
\begin{proof}
Recall that we have a bijection between a cover-exclusive Dyck tiling and a generalized 
chord diagram in $\mathcal{C}_{n}^{(r)}$.
By construction of a cover-exclusive Dyck tiling, a symmetric cover-exclusive Dyck 
tiling corresponds to a symmetric non-crossing partition.
Therefore, an increasing $r$-chain of $\mathcal{SNC}_{n}$ is bijective to a symmetric cover-exclusive
Dyck tiling, whose total number is $B_{n}^{(r)}$ by Definition \ref{defn:Bn}.
This completes the proof. 
\end{proof}

\begin{remark}
Suppose we have an increasing $2$-chain $\pi=(1/2/3/4,124/3)$. 
Since both $1/2/3/4$ and $124/3$ are in $\mathcal{SNC}_{4}$ with $\epsilon=0$, $\pi$ is a chain of symmetric non-crossing partitions.
The bijection $\kappa^{(2)}$ defined in Section \ref{sec:GDPic} maps $\pi$ to the $2$-Dyck path
$P:=URUR^2U^2R^5$.
Since the numbers of up and down steps are different, we have no obvious pictorial symmetry in $P$ as 
a symmetric Dyck path.
However, since $\pi$ and the corresponding generalized chord diagram are symmetric, 
we regard the path $P$ as a symmetric $2$-Dyck path.
\end{remark}

Let $C$ be a symmetric generalized chord diagram in $\mathcal{SC}_{n}^{(r)}$.
We cut $C$ by a vertical line in the middle and obtain a new diagram $\widetilde{C}$.
The diagram $\widetilde{C}$ consists of arches and half-arches.
We put a bullet ``$\bullet$"on the right end of a half-arch, and call it a right-end point.
We call a diagram $\widetilde{C}$ a reduced symmetric chord diagram. 

\begin{defn}
We denote by red-$\mathcal{SC}_{n}^{(r)}$ the set of reduced symmetric chord diagrams.
\end{defn}

For example, we have $\mathcal{SC}_{2}^{(2)}=\mathcal{C}_{2}^{(2)}$ and 
three reduced symmetric chord diagrams.
In Figure \ref{fig:SCtoredSC}, we give all reduced symmetric chord diagrams for 
$(n,r)=(2,2)$.
\begin{figure}[ht]
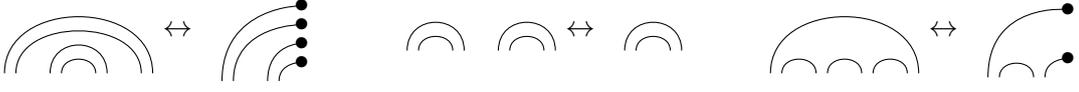

\tikzpic{-0.5}{[yscale=0.5,xscale=0.3]
\draw(0,0)..controls(0,2)and(6.5,2)..(6.5,0);
\draw(0.5,0)..controls(0.5,1.5)and(6,1.5)..(6,0);
\draw(2,0)..controls(2,1)and(4.5,1)..(4.5,0);
\draw(2.5,0)..controls(2.5,0.5)and(4,0.5)..(4,0);
}$\leftrightarrow$
\tikzpic{-0.5}{[yscale=0.5,xscale=0.3]
\draw(0,0)..controls(0,2)and(3.5,2)..(3.5,2)node{$\bullet$};
\draw(0.5,0)..controls(0.5,1.5)and(3.5,1.5)..(3.5,1.5)node{$\bullet$};
\draw(2,0)..controls(2,1)and(3.5,1)..(3.5,1)node{$\bullet$};
\draw(2.5,0)..controls(2.5,0.5)and(3.5,0.5)..(3.5,0.5)node{$\bullet$};
}\qquad
\tikzpic{-0.4}{[yscale=0.5,xscale=0.3]
\draw(0,0)..controls(0,1)and(2.5,1)..(2.5,0);
\draw(0.5,0)..controls(0.5,0.5)and(2,0.5)..(2,0);
\draw(4,0)..controls(4,1)and(6.5,1)..(6.5,0);
\draw(4.5,0)..controls(4.5,0.5)and(6,0.5)..(6,0);
}$\leftrightarrow$
\tikzpic{-0.4}{[yscale=0.5,xscale=0.3]
\draw(0,0)..controls(0,1)and(2.5,1)..(2.5,0);
\draw(0.5,0)..controls(0.5,0.5)and(2,0.5)..(2,0);
}\qquad
\tikzpic{-0.5}{[yscale=0.5,xscale=0.3]
\draw(0.5,0)..controls(0.5,0.5)and(2,0.5)..(2,0);
\draw(2.5,0)..controls(2.5,0.5)and(4,0.5)..(4,0);
\draw(4.5,0)..controls(4.5,0.5)and(6,0.5)..(6,0);
\draw(0,0)..controls(0,2)and(6.5,2)..(6.5,0);
}$\leftrightarrow$
\tikzpic{-0.5}{[yscale=0.5,xscale=0.3]
\draw(0,0)..controls(0,2)and(3.5,1.8)..(3.5,1.8)node{$\bullet$};
\draw(0.5,0)..controls(0.5,0.5)and(2,0.5)..(2,0);
\draw(2.5,0)..controls(2.5,0.5)and(3.5,0.5)..(3.5,0.5)node{$\bullet$};
}
\caption{Symmetric chord diagrams and reduced symmetric chord diagrams for $(n,r)=(2,2)$.}
\label{fig:SCtoredSC}
\end{figure}

The next lemma directly follows from the definitions of $\mathcal{SC}_{n}^{(r)}$ 
and red-$\mathcal{SC}_{n}^{(r)}$.
\begin{lemma}
\label{lemma:SCredSC}
We have a natural bijection between $\mathcal{SC}_{n}^{(r)}$ and red-$\mathcal{SC}_{n}^{(r)}$.
\end{lemma}

\subsection{An algebra \texorpdfstring{$\mathbb{SNC}_{n}$}{SNCn} on symmetric non-crossing partitions}
We first introduce an algebra $\mathbb{SNC}_{n}$ on symmetric non-crossing partitions.
Here, symmetric non-crossing partitions are characterized by $\epsilon=0$, that is, the integer $1$
is on the vertical line in the middle.
Then, we show that this algebra is isomorphic to the one-boundary Temperley--Lieb algebra
acting on symmetric chord diagrams.
In fact, we give the set of generators which act on reduced symmetric chord diagrams.
As in the case of Temperley--Lieb algebra, we construct one-boundary Fuss--Catalan algebra 
acting on increasing $r$-chains of symmetric non-crossing partitions 
by generalizing the case of increasing $r$-chains in $\mathcal{NC}_{n}^{(r)}$.
Let $\pi$ be a symmetric non-crossing partition in $\mathcal{SNC}_{n}$ 
corresponding to a symmetric chord diagram $C$.
We consider the map $\Psi$ rather than $\Phi$. Therefore, by Proposition \ref{prop:SNCtoSC}, we have $C=\Psi(\pi)$.
Let $B(i)$ be a block in $\pi$ such that the integer $i$ belongs to $B(i)$.
By definition, we may have $B(i)=B(j)$ for $i\neq j$ if $|B(i)|\ge2$.
We define a map $g_{p}:\mathcal{SNC}_{n}\rightarrow\mathcal{SNC}_{n}$, 
$1\le p\le n-1$, as follows.
We define 
\begin{align}
\label{eq:taup}
\tau_{p}:=
\begin{cases}
-(q_n+q_n), & \text{ if } p\equiv n \mod 2, \\
qq_n^{-1}+q^{-1}q_{n}, & \text{ otherwise }.
\end{cases}
\end{align}
We also define $B(n+1):=B(1)$.
Then, the action of $g_p$ on a symmetric non-crossing partition $\pi$ is defined to be 
\begin{align}
\label{eq:defgp}
g_{p}\pi:=\begin{cases}
\tau \pi, & \text{ if } B(1)=B(2) \text{ and } B(n+2-p)=B(n+1-p), \\
\tau_{p}\pi', & \text{ if } B(1)=B(n+2-p)\neq B(2)=B(n+1-p),  \\ 
\pi', & \text{ otherwise },
\end{cases}
\end{align}
where  $\tau=-(q+q^{-1})$, and $\pi'$ is a symmetric non-crossing partition obtained from $\pi$
by merging the four blocks into larger blocks $B(1)\cup B(2)$ and $B(n+1-p)\cup B(n+2-p)$.
Note that two blocks $B(1)\cup B(2)$ and $B(n+1-p)\cup B(n+2-p)$ may coincide with each other.

Similarly, for $\epsilon=0$, we define the action of $g_{n}$ on $\pi$ by 
\begin{align}
\label{eq:gn}
g_{n}\pi:=
\begin{cases}
\pi'', & \text{ if  } n\equiv 0 \pmod{2} \text{ and } |B(n/2+1)|\neq1, \\
\pi''', & \text{ if } n\equiv1\pmod{2} \text{ and } B((n+1)/2)\neq B((n+3)/2), \\
-(q_n+q_n^{-1})\pi, & \text{ otherwise },
\end{cases}
\end{align}
where $\pi''$ is a symmetric non-crossing partition obtained from $\pi$ in such a way that
we divide the block $B(n/2+1)$ into two smaller blocks 
$B_1=\{n/2+1\}$ and $B(n/2+1)\setminus B_1$, and 
$\pi'''$ is a symmetric non-crossing partition obtained from $\pi$ by merging two 
blocks $B((n+1)/2)$ and $B((n+3)/2)$.

The action of $g_{n}$ on $\mathcal{SNC}_{n}$ in the case of $\epsilon=1$ for $n$ even 
can be similarly defined by using the action of the Kreweras endomorphism $\rho$ on $\mathcal{SNC}_{n}$
in the case of $\epsilon=0$ as in Remark \ref{remark:epsilon}.

We define the generators $G_{i}$, $1\le i\le n$, by
\begin{align}
\label{eq:gp}
G_{i}:=\begin{cases}
\rho^{i-1}g_{i}\rho^{-(i-1)}, & \text{ if } 1\le i\le n-1, \\
g_{n}, & \text{ if } i=n,
\end{cases}
\end{align}
where $\rho$ is the Kreweras endomorphism defined in Section \ref{sec:Kreendo}.

As a consequence, we define the set of generators $\{G_{i}: 1\le i\le n\}$ acting
on $\mathcal{SNC}_{n}$.

Since we have a natural inclusion $\mathcal{SNC}_{n}\hookrightarrow\mathcal{NC}_{n}$,
one can define the action of $F_{i}$, $1\le i\le 2n-1$, on $\mathcal{SNC}_{n}$
by Eq. (\ref{def:Fi}).
Then, the generator $G_{i}$, $1\le i\le n-1$, can be expressed in terms of $F_{i}$
by $G_{i}=F_{i}F_{2n-i}$.

\begin{lemma}
The action of the generator $G_{i}$ is well-defined, {\it i.e.}, we have 
$G_{i}\pi\in\mathcal{SNC}_{n}$ if $\pi\in\mathcal{SNC}_{n}$.
\end{lemma}
\begin{proof}
We prove only the case $i=1$ since other cases $2\le i\le n-1$ can be shown by a similar argument.
We first show that $G_{i}\pi\in\mathcal{SNC}_{n}$ for $\pi\in\mathcal{SNC}_{n}$ and $i=1$.
We have three cases as in Eq. (\ref{eq:defgp}).
We make use of the bijection $\Psi:\mathcal{SNC}_{n}\rightarrow\mathcal{SC}_{n}$.
Let $C=\Psi(\pi)$ for $\pi\in\mathcal{SNC}_{n}$.
Recall that the bottom points of $C\in\mathcal{SC}_{n}$ has labels $1,1',\ldots,n,n'$. 
First, if $B(1)=B(2)$ and $B(n+1)=B(n)$, then $C$ has two arches  connecting $1$ and $1'$, and 
$n$ and $n'$. In this case, since $G_{1}\pi=\tau\pi$, $G_{1}\pi$ is a symmetric non-crossing partition.
Secondly, if $B(1)=B(n+1)\neq B(2)=B(n)$, $\pi$ has two arches connecting $1$ and $n'$, and 
$1'$ and $n$. Suppose that the action of $G_{1}$ on $\pi$ gives $\pi':=\Psi^{-1}(C')$. 
Then, $C'$ has two arches connecting $1$ and $1'$, and $n$ and $n'$. 
Other arches in $C'$ are the same as $C$.
From these, $C'$ is also symmetric if $C$ is.
In the third case, $\pi'$ is obtained by merging four blocks into larger blocks.
By the same argument as the second case, it is clear that $\pi'$ is symmetric if $\pi$ is.

We show that $G_{n}\pi\in\mathcal{SNC}_{n}$.
We have three cases in Eq. (\ref{eq:gn}). 
We first consider the third case. If $n\equiv1\pmod2$, $\pi\in\mathcal{SNC}_{n}$ implies that 
$C$ has an arch connecting $\lfloor(n+1)/2\rfloor$ and $(\lfloor(n+1)/2\rfloor)'$.
Similarly, if $n\equiv0\pmod2$ and $|B(n/2+1)|=1$, then $C$ contains an arch connecting 
$(n/2)'$ and $n/2+1$. 
These arches are symmetric along the vertical line in the middle.
This means that the third case is trivial.
Secondly, we consider the first case. 
The non-crossing partition $\pi''=\Psi(C'')$ is obtained from $\pi$ by dividing 
the block $B(n/2+1)$ into two smaller blocks. 
Since $\pi''$ has a block consisting of a single element $\{n/2+1\}$, $C''$ has an arch connecting 
$n/2+1$ and $(n/2)'$. This arch is symmetric along the vertical line in the middle.
Suppose that $C$ has two arches connecting $i$ and $(n/2)'$, and $n/2+1$ and $(n+1-i)'$.
The action of $G_{n}$ on $\pi$ implies that $C''$ has two arches connecting $i$ and $(n+1-i)'$, 
and $(n/2)'$ and $n/2+1$. 
These two arches are symmetric. All the other arches in $C$ and $C''$ coincide with each other.
Therefore, $\pi''$ is symmetric.  
Finally, we consider the second case.
The non-crossing partition $\pi\in\mathcal{SNC}_{n}$ has two distinguished 
arches in $C$. They are arches connecting $(n+1)/2$ and $(\min B((n+1)/2)-1)'$, and 
$\max B((n+3)/2)$ and $((n+1)/2)'$.
Since $B((n+1)/2)\neq B((n+3)/2)$ and $\pi\in\mathcal{SNC}_{n}$, if we merge $B((n+1)/2)$ and $B((n+3)/2)$ into a larger 
block, we replace the above mentioned two arches with two arches connecting $(n+1)/2$ and $((n+1)/2)'$, and 
$n_1'$ and $n_2$ where $n_1=\min B((n+1)/2)-1$ and $n_2=\max B((n+3)/2)$.
It is clear that these two new arches are symmetric, and other arches remain the same.
This implies that $\pi'''$ is symmetric.

From these, the non-crossing partition $G_{i}\pi$ is symmetric along the vertical line in the middle. 
This completes the proof.
\end{proof}

\begin{defn}
The algebra $\mathbb{SNC}_{n}$ is a unital associative $\mathbb{C}[q,q^{-1},q_n,q_{n}^{-1}]$-algebra
generated by the set of generators $\{G_{i}: 1\le i\le n\}$.
We call the algebra $\mathbb{SNC}_{n}$ the one-boundary Temperley--Lieb algebra on symmetric 
non-crossing partitions.
\end{defn}

\subsection{The algebra \texorpdfstring{$\mathbb{SNC}_{n}^{(r)}$}{SNCnr}}
Let $\pi^{(r)}:=(\pi_1,\ldots,\pi_r)$ be an increasing $r$-chain in $\mathcal{SNC}_{n}$.
We first define $G_i^{(s)}$, $1\le i\le n$, $1\le s\le r$, on 
$\mathcal{SNC}_{n}^{(r)}$ by 
\begin{align*}
G_i^{(s)}\pi^{(r)}
:=
\begin{cases}
(\pi_1,\ldots,\pi_{r-s},G_{i}(\pi_{r-s+1}),\dots,G_{i}(\pi_{r})), & 
\text{ if } i\equiv 1\pmod2, \\
(G_i(\pi_1),\ldots,G_i(\pi_s),\pi_{s+1},\ldots,\pi_{r}), & 
\text{ if } i\equiv 0\pmod2,
\end{cases}
\end{align*}
where the operator $G_{i}$, $1\le i\le n$, is defined in Eq. (\ref{eq:gp}).

As a consequence, we have the set of generators $\{G_{i}^{(s)}: 1\le i\le n, 1\le s\le r\}$
acting on an increasing $r$-chain in $\mathcal{SNC}_{n}^{(r)}$.

The algebra $\mathbb{SNC}_{n}$ is generalized to an algebra acting on $r$-chains
in $\mathcal{SNC}_{n}^{(r)}$.
\begin{defn}
The algebra $\mathbb{SNC}_{n}^{(r)}$ is a unital associative algebra over 
$\mathbb{C}[q,q^{-1},q_n,q_n^{-1}]$ generated by 
$\{G_{i}^{(s)} : 1\le i\le n, 1\le s\le r\}$.
We call the algebra $\mathbb{SNC}_{n}^{(r)}$ the one-boundary Fuss--Catalan
algebra on an increasing $r$-chain in $\mathcal{SNC}_{n}^{(r)}$.
\end{defn}

\subsection{An algebra \texorpdfstring{$2$-$\mathbb{SNC}_{n}$}{2SNCn} on symmetric crossing partitions}
\label{sec:2SNCn}
We introduce an algebra $2$-$\mathbb{SNC}_{n}$ on symmetric non-crossing partitions.
Let $\pi$ be a symmetric non-crossing partition $\mathcal{SNC}_{n}$. 
We generalize the notion of a symmetric non-crossing partition by adding a notion
of a primed integer. A symmetric non-crossing partition with primed integers carries the information 
about a partition, and some extra information on it.

Let $\pi\in\mathcal{SNC}_{n}$ and $B_i$ be a block of $\pi$.
Recall $\mathcal{E}(\pi)$ in Eq. (\ref{eq:calE}) is the set of pairs of integers.
We regard the pair of integers $(1,1)$ as the pair $(1,n+1)$.
Let $\mathcal{E}^{sym}(\pi)$ be the set of integers such that 
\begin{align*}
\mathcal{E}^{sym}(\pi)
&:=\{
b\in(b,c): (b,c)\in\mathcal{E}(\pi), b+c=n+2\},
\end{align*}
We introduce a linear order $\le_{1}$ for the set $[1,n]$:
\begin{align}
\label{eq:lo1}
1>_{1}n>_{1}2>_{1}n-1>_{1}\ldots >_{1}\lfloor n/2+1\rfloor.
\end{align}
We say that the integers $1,2,3\ldots,\lfloor n/2\rfloor$ have an odd parity and 
$n,n-1,\ldots,\lfloor n/2+1\rfloor$ have an even parity.
We denote by $S^{sym}(\pi)$  a sequence of decreasing 
integers obtained from the set $\mathcal{E}^{sym}(\pi)$ by 
the linear order $>_1$.

Suppose we have $S^{sym}(\pi):=(s_1,\ldots,s_{m})$ such that $s_1>_{1}s_2>_{1}\ldots>_{1}s_{m}$. 
We construct $m+1$ symmetric non-crossing partitions with primed integers as follows.
We start from a symmetric non-crossing partition $\pi$ which has no primed integers. 
We put a prime on integers $s_1,s_2,\ldots, s_{k}$ in $\pi$ for $0\le k\le m$. 
For example, if $s_1=1$, then we put a prime on the integer $1$ in a symmetric 
non-crossing partition. 

For example, consider the symmetric non-crossing partition $\pi=1/2/34/5$.
We have the linear order $1>_{1}5>_{1}2>_{1}4>_{1}3$.
Then, we have $\mathcal{E}(\pi)=\{(1,6),(2,2),(3,4),(4,3),(5,5)\}$, 
$\mathcal{E}^{sym}(\pi)=\{1,3,4\}$, and $S^{sym}(\pi)=(1,4,3)$.
We have four elements from $\pi$:
\begin{align*}
1/2/34/5, \qquad 1'/2/34/5, \qquad 1'/2/34'/5 \qquad 1'/2/3'4'/5.
\end{align*}
Note that a partition $1/2/3'4'/5$ is not admissible since the integer $1$ is the maximal element 
in the linear order and we have no prime on the integer $1$.

\begin{defn}
We denote by $\mathcal{SNC}'_{n}$ the set of symmetric non-crossing partitions consisting 
of unprimed integers and primed integers obtained as above	.
\end{defn}

Let $\pi\in\mathcal{SNC'}_{n}$ be a symmetric non-crossing partition with primed integers.
Suppose we have a symmetric chord diagram $C$ corresponding to $\pi$ by forgetting 
the primes.
When an integer $i$ is primed in $\pi$, we put a dot on a chord connecting $i$ and $j'$
with some $j$.
The linear order $\ge_1$ insures that if there is a dotted symmetric chord $c_1$ in $C$, 
then any symmetric chord $c_2$ which is outer than $c_1$ has a dot.

\begin{example}
We consider $\pi=1'/2/34'/5\in\mathcal{SNC'}_{5}$.
Then, the chord diagram for $\pi$ is depicted as 
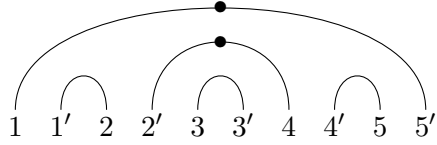
\begin{center}
\tikzpic{-0.5}{[scale=0.6]
\draw(0,0)..controls(0,3)and(9,3)..(9,0);
\draw(1,0)..controls(1,1)and(2,1)..(2,0);
\draw(3,0)..controls(3,2)and(6,2)..(6,0);
\draw(4,0)..controls(4,1)and(5,1)..(5,0);
\draw(7,0)..controls(7,1)and(8,1)..(8,0);
\foreach \x/\y in {0/1,2/2,4/3,6/4,8/5}
\draw(\x,-0.8)node[anchor=south]{$\y$};
\foreach \x/\y in {1/1',3/2',5/3',7/4',9/5'}
\draw(\x,-0.8)node[anchor=south]{$\y$};	
\draw(4.5,2.27)node{$\bullet$}(4.5,1.49)node{$\bullet$};
}
\end{center} 
The chords consisting of $1$ and $5'$, and $4$ and $2'$ have dots.
The dots in the chord diagram correspond to primes in $\pi$.
\end{example}

\begin{example}
Set $n=3$. We have three elements in $\mathcal{SNC}_3$: $1/23$, $1/2/3$, and $123$.
We have eight elements in $\mathcal{SNC}'_{3}$:
\begin{align*}
1/23 \quad 1'/23 \quad 1'/23' \quad 1'/2'3' \quad 1/2/3 \quad 1'/2/3 \quad 123 \quad 12'3
\end{align*}
since we have a linear order $1>_1 3>_1 2$.
\end{example}

\begin{remark}
The number of elements in $\mathcal{SNC}'_{n}$ is given by $2^n$. An element $C$ in $\mathcal{SNC}'_n$
is bijective to a sequence of $\{\pm\}^{n}$. This is realized by cutting $C$ along the vertical line 
in the middle, and consider the left-half of $C$. 
We put $+$ (resp. $-$) on a bottom point $p$ if $p$ or $p'$ is connected to $q'$ or $q$ 
by an arch with $q>p$ (resp. $q<p$), or $p$ has a right-end (resp. left-end) point.
For example, we have 
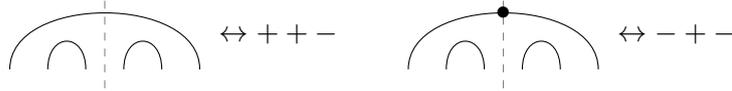
\begin{align*}
\tikzpic{-0.5}{[scale=0.5]
\draw(0,0)..controls(0,2)and(5,2)..(5,0)(1,0)..controls(1,1)and(2,1)..(2,0)
(3,0)..controls(3,1)and(4,1)..(4,0);
\draw[gray,dashed](2.5,-0.5)--(2.5,2);
}\leftrightarrow ++-
\qquad
\tikzpic{-0.5}{[scale=0.5]
\draw(0,0)..controls(0,2)and(5,2)..(5,0)(1,0)..controls(1,1)and(2,1)..(2,0)
(3,0)..controls(3,1)and(4,1)..(4,0);
\draw[gray,dashed](2.5,-0.5)--(2.5,2);
\draw(2.5,1.5)node{$\bullet$};
}\leftrightarrow -+-
\end{align*} 
\end{remark}

\begin{remark}
The linear order $>_{1}$ is compatible with the order of arches in a symmetric chord diagram.
For example, consider the symmetric chord diagram $1/24/3$:
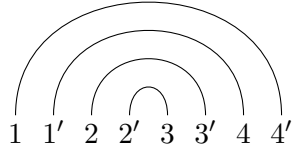
\begin{center}
\tikzpic{-0.5}{[scale=0.5]
\draw(0,0)..controls(0,4)and(7,4)..(7,0);
\draw(1,0)..controls(1,3)and(6,3)..(6,0);
\draw(2,0)..controls(2,2)and(5,2)..(5,0);
\draw(3,0)..controls(3,1)and(4,1)..(4,0);
\foreach \x/\y in {0/1,2/2,4/3,6/4}
\draw(\x,-1)node[anchor=south]{$\y$};
\foreach \x/\y in {1/1',3/2',5/3',7/4'}
\draw(\x,-1)node[anchor=south]{$\y$};
}
\end{center}
Let $a,b\in\mathcal{E}^{sym}(\pi)$. Then, we have two symmetric arches 
containing $a$ and $b$ respectively.
We define $a>b$ if the symmetric arch containing $a$ is outer than 
that containing $b$.
In the case of $1/24/3$, we have four symmetric arches, and the order 
of the arches is given by $1>4>2>3$ from top to bottom.
This order is nothing but the order by $>_{1}$ on the set $[1,4]$.
\end{remark}

We study the action of the Kreweras endomorphism $\rho$ on $\mathcal{SNC}'_{n}$.
The action of $\rho$ on $\pi\in\mathcal{SNC}_{n}$ is given by $\rho(\pi)$.
Suppose that $\pi'\in\mathcal{SNC}'_{n}$ has primed integers, and $\pi\in\mathcal{SNC}_{n}$ 
is obtained from $\pi'$ by forgetting primes.
Further, suppose that the decreasing sequence of possible primed integers in $\pi'$ is 
given by $Q(\pi'):=(q_1,q_2,\ldots,q_{r})$ with $q_1>_1q_2>_{1}\ldots >_{1}q_r$. We define a sequence  
$\overline{Q(\pi')}$ by 
\begin{align*}
\overline{Q(\pi')}:=(\overline{q_1},\ldots,\overline{q_r}),
\end{align*}
where $\overline{i}=n+1-i$.
The non-crossing partition $\rho^{-1}(\pi')$ can have a primed integer $i$ if  
$i\in\overline{Q(\pi')}$.
The sequence $Q(\pi')$ allows us to trace the primed integers in $\pi'$ after the 
action of $\rho^{-1}$.
Note that $\pi'$ is symmetric, however, $\rho^{-1}(\pi')$ is not symmetric in general.
More generally, we define 
\begin{align*}
Q(\pi';t)&:=(p_{1}-t,p_2-t,\ldots,p_{r}-t), \\
\overline{Q(\pi';t)}&:=(\overline{p_1}-t,\overline{p_2}-t,\ldots,\overline{p_r}-t),
\end{align*}
where an element in $Q(\pi';t)$ or $\overline{Q(\pi';t)}$ is modulo $n$.
Then, $\rho^{-2t}(\pi')$ is $\rho^{-2t}(\pi)$ as a partition, and the integers 
$i\in Q(\pi';t)$ are primed.
Similarly, $\rho^{-2t-1}(\pi')$ is $\rho^{-2t-1}(\pi)$ as a partition, and the integers
$i\in \overline{Q(\pi';t)}$ can be primed.

In what follows, we introduce the set of generators $G_{i}$, $0\le i\le n$, acting 
on $\mathcal{SNC}'_{n}$.
Recall $\tau_{p}$ is defined in Eq. (\ref{eq:taup}). 
Similarly, we define
\begin{align*}
\tau'_{p}:=
\begin{cases}
-(q_0+q_0^{-1}), & \text{ if } p\equiv0\pmod2, \\
qq_{0}^{-1}+q^{-1}q_{0}, & \text{ otherwise }.
\end{cases}
\end{align*}
We introduce a new variable $\theta$ to relate the new algebra (given in Definition \ref{defn:2SNC}) with 
the finite dimensional two-boundary Temperley--Lieb algebra.

We introduce the $n+1$ generators $g_{i}$, $0\le i\le n$, on a non-crossing partition 
with primed integers.
The action of $g_{p}$, $1\le p\le n-1$, on $\pi$ is given by 
\begin{align}
\label{eq:gi2}
g_{p}\pi:=
\begin{cases}
\tau \pi, & \text{ if } B(1)=B(2), B(n+2-p)=B(n+1-p), \\
\tau_{p}\pi^{(1)}, &
\begin{aligned}
&\text{ if } B(1)=B(n+2-p)\neq B(2)=B(n+1-p), \\ 
&\quad\text{ and } p, n+1-p \text{ are unprimed},
\end{aligned}\\
\theta \pi^{(1)}, & 
\begin{aligned}
&\text{ if } B(1)=B(n+2-p)\neq B(2)=B(n+1-p), \\ 
&\quad p \text{ is primed and } n+1-p \text{ is unprimed},
\end{aligned}\\ 
\tau'_{p}\pi^{(1)}, &
\begin{aligned}
&\text{ if } B(1)=B(n+2-p)\neq B(2)=B(n+1-p), \\ 
&\quad\text{and }p, n+1-p \text{ are primed},
\end{aligned}\\
\pi^{(1)}, & \text{ otherwise},	
\end{cases}
\end{align}
where $\pi^{(1)}$ is a non-crossing partition obtained from $\pi$ by merging 
the four blocks into larger blocks $B(1)\cup B(2)$ and $B(n+1-p)\cup B(n+2-p)$.

The action of $g_{n}$ on $\pi$ is given by 
\begin{align}
\label{eq:gn2}
g_{n}\pi:=
\begin{cases}
\pi^{(2)}, & 
\text{ if } n\equiv0\pmod2 \text{ and } |B(n/2+1)|\neq1,  \\
\pi^{(3)}, & \text{ if } n\equiv1\pmod2 \text{ and } B((n+1)/2)\neq B((n+3)/2), \\
\theta \pi^{(4)}, & \text{ if } \lfloor n/2+1\rfloor \text{ is primed}, \\
-(q_n+q_n^{-1})\pi, & \text{ otherwise }.
\end{cases}
\end{align}
Here, $\pi^{(2)}$ is the non-crossing partition obtained from $\pi$ by dividing 
the block $B(n/2+1)$ into two smaller blocks $B_1=\{n/2+1\}$ and $B(n/2+1)\setminus B_1$.
The non-crossing partition $\pi^{(3)}$ is obtained from $\pi$ by merging the two blocks
$B((n+1)/2)$ and $B((n+3)/2)$.
The non-crossing partition $\pi^{(4)}$ is obtained from $\pi$ by deleting the prime on 
the integer $\lfloor n/2+1\rfloor$.

Suppose $(1,n_1)\in\mathcal{E}(\pi)$.
The action of $g_0$ on $\pi$ is given by 
\begin{align}
\label{eq:g0}
g_{0}\pi:=
\begin{cases}
-(q_0+q_0^{-1})\pi, & \text{ if } n_1=1 \text{ and } 1 \text{ is primed}, \\ 
\theta\pi^{(5)}, & \text{ if } n_1=1 \text{ and } 1 \text{ is not primed}, \\ 
\pi^{(6)}, & \text{ if } n_1\neq1.
\end{cases}
\end{align}
Here, $\pi^{(5)}$ is obtained from $\pi$ by adding a prime on $1$.
In the third case, since $(1,n_1)\in\mathcal{E}(\pi)$, the integer $1$ is not primed in $\pi$, and 
$B(1)=B(n_1)=B(n+2-n_1)$.
The non-crossing partition $\pi^{(6)}$ is obtained from $\pi$ by dividing the block $B(1)$
into smaller blocks $B_1=\{1\}$ and $B(1)\setminus B_1$, and further by adding primes on $1$ and $n+2-n_1$.

\begin{example}
Let $\pi=13456/2/7$. Then the action of $g_0$ on $\pi$ gives 
$1'/2/3456'/7$.
This corresponds to the third case in Eq. (\ref{eq:g0}).
In terms of chord diagrams, we have
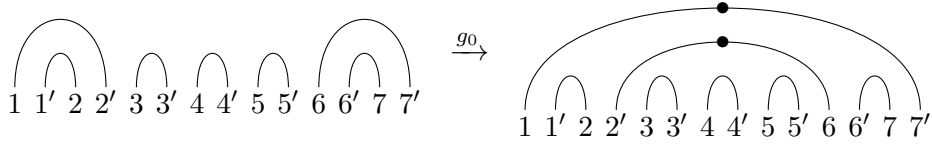
\begin{align*}
\tikzpic{-0.5}{[xscale=0.4,yscale=0.6]
\draw(0,0)..controls(0,2)and(3,2)..(3,0);
\draw(10,0)..controls(10,2)and(13,2)..(13,0);
\foreach \x in {1,4,6,8,11}
\draw(\x,0)..controls(\x,1)and(\x+1,1)..(\x+1,0);
\foreach \x/\y in {0/1,2/2,4/3,6/4,8/5,10/6,12/7}
\draw(\x,-0.8)node[anchor=south]{$\y$};
\foreach \x/\y in {1/1',3/2',5/3',7/4',9/5',11/6',13/7'}
\draw(\x,-0.8)node[anchor=south]{$\y$};
}
\xrightarrow{g_0}
\tikzpic{-0.5}{[xscale=0.4,yscale=0.6]
\draw(0,0)..controls(0,3)and(13,3)..(13,0);
\draw(3,0)..controls(3,2)and(10,2)..(10,0);
\foreach \x in {1,4,6,8,11}
\draw(\x,0)..controls(\x,1)and(\x+1,1)..(\x+1,0);
\foreach \x/\y in {0/1,2/2,4/3,6/4,8/5,10/6,12/7}
\draw(\x,-0.8)node[anchor=south]{$\y$};
\foreach \x/\y in {1/1',3/2',5/3',7/4',9/5',11/6',13/7'}
\draw(\x,-0.8)node[anchor=south]{$\y$};
\draw(6.5,2.24)node{$\bullet$}(6.5,1.48)node{$\bullet$};
}
\end{align*}
The two dotted symmetric chords correspond to the primed 
integers $1$ and $6$ in $\pi$.
\end{example}

We define the generators $G_{i}$, $0\le i\le n$, by 
\begin{align}
\label{eq:defnGis}
G_{i}:=\begin{cases}
g_{i}, & \text{ if } i=0 \text{ or } n, \\
\rho^{i-1}g_{i}\rho^{-(i-1)}, & \text{ if } 1\le i\le n-1.
\end{cases}
\end{align}

As a consequence, we define the set of generators $\{G_{i}: 0\le i\le n\}$
acting on a symmetric non-crossing partition in $\mathcal{SNC}'_{n}$.

\begin{defn}
\label{defn:2SNC}
An algebra $2$-$\mathbb{SNC}_{n}$ is a unital associative algebra over 
$\mathbb{C}[q,q^{-1},q_n,q_n^{-1},q_0,q_0^{-1},\theta]$ generated by the set of 
generators $\{G_{i}: 0\le i\le n\}$.
We call the algebra $2$-$\mathbb{SNC}_{n}$ the two-boundary Temperley--Lieb 
algebra on $\mathcal{SNC}'_{n}$.
\end{defn}

\subsection{The algebra \texorpdfstring{$2$-$\mathbb{SNC}_{n}^{(r)}$}{2SNCnr}}
Let $\pi^{(r)}:=(\pi_1,\ldots,\pi_r)\in\mathcal{SNC}_{n}^{(r)}$.
We consider an increasing $r$-chain with primed integers in the poset of $\mathcal{SNC}'_{n}$.
The $r$-chain $\pi^{(r)}$ satisfies the following conditions:
\begin{enumerate}[(a)]
\item $\pi_{i}\le \pi_{j}$ as a symmetric non-crossing partition if we forget primes on integers.
\item 
Let $m$ be a maximal integer in the linear order (\ref{eq:lo1}) such that $m$ has no prime in $\pi_{i}$
but can have a prime in $\pi_{i}$.
If the parity of $m$ is odd, then the integers $t>_1m$ have a prime in $\pi_{i+1}$ if possible.
If the parity of $m$ is even, then the integers $t>_1m$ have a prime if possible and the integer $m$ may have 
a prime in $\pi_{i+1}$.
\end{enumerate}

\begin{remark}
The condition (b) can be rephrased in terms of a generalized chord diagram as follows.
The integer $m$ and the primed integer $(n+1-m)'$ form an outer-most symmetric chord $c$ without a dot 
in a cord diagram corresponding to $\pi_{i}$.
Then, the condition (b) simply implies that a generalized chord diagram corresponding to 
$\pi^{(r)}$ has dots on the symmetric chords above $c$.
\end{remark}

\begin{defn}
\label{defn:SNCdash}
We denote by $\mathcal{SNC'}_{n}^{(r)}$ the set of $r$-chains of symmetric non-crossing 
partitions with primed integers which satisfy the conditions (a) and (b).
\end{defn}

\begin{example}
Consider the $3$-chain $\pi=(1/2,1/2,1/2)$. We have the linear order $1>_{1}2$. 
Then, we obtain seven $3$-chains with primed integers from $\pi$:
\begin{align*}
&(1/2,1/2,1/2)\ \  (1'/2,1/2,1/2)\ \  (1'/2,1'/2,1/2) \ \ (1'/2,1'/2,1'/2)\\
&\qquad(1'/2,1'/2,1'/2')\ \  (1'/2,1'/2',1'/2') \ \ (1'/2',1'/2',1'/2')
\end{align*}
Other $3$-chains with primed integers such as $(1'/2',1'/2,1'/2)$ are not allowed
by the condition (b). 
\end{example}

We define $G_{i}^{(s)}$, $0\le i\le n$, $1\le s\le r$, on $\mathcal{SNC'}_{n}^{(r)}$ by
\begin{align}
\label{eq:defngis}
G_{i}^{(s)}\pi^{(r)}:=
\begin{cases}
(\pi_1,\ldots,\pi_{r-s},G_i(\pi_{r-s+1}),\ldots, G_{i}(\pi_{r})), & 
i\equiv 1\pmod2,\\
(G_i(\pi_1),\ldots,G_i(\pi_s),\pi_{s+1},\ldots,\pi_{r}), & 
i\equiv 0\pmod2,
\end{cases}
\end{align}
where the operator $G_i$ is defined in Eq. (\ref{eq:gi2}), (\ref{eq:gn2}), or(\ref{eq:g0}).

As a consequence, 
we define the set of generators $\{G_{i}^{(s)}: 0\le i\le n, 1\le s\le r\}$
acting on $\mathcal{SNC'}_{n}^{(r)}$.

The algebra $2$-$\mathbb{SNC}_{n}$ is generalized to an algebra 
acting on $r$-chains in $\mathcal{SNC'}_{n}^{(r)}$.

\begin{defn}
The algebra $2$-$\mathbb{SNC}_{n}^{(r)}$ is a unital associative algebra over 
$\mathbb{C}[q^{\pm1},q_n^{\pm1},q_0^{\pm1},\theta]$ generated by the set 
$\{G_{i}^{(s)}: 0\le i\le n, 1\le s\le r\}$.
\end{defn}

\section{One-boundary Fuss--Catalan algebra and symmetric non-crossing partitions}
\label{sec:1bFC}
\subsection{One-boundary Fuss-Catalan algebra}
In this section, we introduce the notion of one-boundary Fuss--Catalan algebra on 
a symmetric generalized chord diagram in $\mathcal{SC}_{n}^{(r)}$ by using 
diagrams. This diagram algebra is a natural generalization of 
the one-boundary Temperley--Lieb algebra.

Let $E_{i}^{(s)}$, $1\le i\le n-1$, $1\le s\le r$, be a diagram depicted 
in Eq. (\ref{eq:picFis}).
Similarly, we depict the generator $E_{n}^{(s)}$, $1\le s\le r$, by
\begin{align}
\label{eq:picFn}
E_{n}^{(s)}:=
\tikzpic{-0.5}{
\node(1)at(0,-1){\framebox{$s'$}};
\draw(0,0.3)--(0,-0.74)(0,-1.26)--(0,-2.3);
\node(2)at(1.2,-0.7){\framebox{$s$}};
\draw(0.45,0.3)..controls(0.45,-0.2)and(0.69,-0.75)..(0.99,-0.75);
\draw(1.4,-0.75)--(1.8,-0.75)node{$\bullet$};
\node(3)at(1.2,-1.3){\framebox{$s$}};
\draw(0.45,-2.3)..controls(0.45,-1.8)and(0.7,-1.3)..(1,-1.3);
\draw(1.4,-1.3)--(1.8,-1.3)node{$\bullet$};
\draw(0.225,-2.3)node[anchor=north]{$n$};
}
\end{align}
where $s'=r-s$. The action of $E_{n}^{(s)}$ is local, that is, we have 
vertical bundled strands at the positions except $n$.
We call a bullet $\bullet$ a right-end point. 
We have $2s$ right-end points in $E_{n}^{(s)}$. We call the bottom (resp. top) $s$ right-end points 
odd (resp. even) right-end points.

\begin{defn}
\label{defn:1BTL}
The one-boundary Fuss--Catalan algebra $1$-$\mathbb{BFC}_{n}^{(r)}$ is a unital associative algebra
over $\mathbb{C}[q,q^{-1},q_n,q_n^{-1}]$ 
generated by the set of generators $\{E_{i}^{(s)}: 1\le i\le n, 1\le s\le r\}$.
The product of $X,Y\in1\text{-}\mathbb{BFC}_{n}^{(r)}$ is calculated by putting the diagram of $Y$ on top 
of the diagram of $X$. If we have a closed loop, we remove it and give a factor $-(q+q^{-1})$.
If we have a strand from a lower even right-end point to an upper odd right-end point, 
we remove it and give a factor $-(q_n+q_n^{-1})$.
Similarly, if we have a strand from a lower odd right-end point to an upper right-end point, we remove 
it and give a factor $(qq_{n}^{-1}+q^{-1}q_{n})$.
\end{defn}

\begin{remark}
When $r=1$, $1$-$\mathbb{BFC}_{n}^{(1)}$ is the well-studied generalization of the Temperley--Lieb
algebra called one-boundary Temperley--Lieb algebra 
(or sometimes called the blob algebra) \cite{deGNicPyaRit05,MarSal93,MarSal94,MarWoo00,MarWoo03}.
Then, $1$-$\mathbb{BFC}_{n}^{(1)}$ is generated by the set of generators $\{e_i: 1\le i\le n\}$.
The set $\{e_i: 1\le i\le n-1\}$ generates the Temperley--Lieb algebra $\mathbb{TL}_{n}$, 
and we have 
\begin{align*}
&e_n^{2}=-(q_n+q_{n}^{-1})e_{n}, \\
&e_{n-1}e_{n}e_{n-1}=\tau' e_{n-1}, \\
&e_{i}e_{n}=e_{n}e_{i}, \qquad i\neq n-1,
\end{align*}
where $\tau'=(qq_{n}^{-1}+q^{-1}q_{n})$.
\end{remark}

Recall that we have several relations of order up to three and many relations of order larger than four
for $\mathbb{TL}_{n}^{(r)}$.
In the case of $1$-$\mathbb{BFC}_{n}^{(r)}$ with $r\ge2$, we have many relations as in the case of 
$\mathbb{TL}_{n}^{(r)}$.
However, any relation can be calculated by use of the diagram representations (\ref{eq:picFis}) and 
(\ref{eq:picFn}) of the generators $E_{i}^{(s)}$.

The next proposition is the boundary analogue of Proposition \ref{prop:TLP}.
Recall that $B_{n}^{(r)}$ is the number of elements in $\mathcal{SNC}_{n}^{(r)}$ 
by Definition \ref{defn:Bn}.
\begin{prop}
We have 
\begin{align}
\label{eq:dimBTL}
\dim(1\text{-}\mathbb{BFC}_{n}^{(r)})=|B_{2n}^{(r)}|.
\end{align}
\end{prop}
\begin{proof}
Let $X$ be an element in $1\text{-}\mathbb{BFC}_{n}^{(r)}$. 
Graphically, $X$ has $2p$ right-end points and $rn-p$ strands in 
its diagram where $0\le p\le rn$.
More precisely, $X$ has $n$ points on the top and the bottom respectively, and 
each points are connected to another points or a right-end point. 
We number the top $n$ points from right to left from $1$ to $n$, and 
the bottom $n$ points from left to right from $n+1$ to $2n$.
We fold the diagram of $X$ down to the left such that the $n$ top points are left 
to the $n$ bottom points.  
We denote by $\widetilde{X}$ this new diagram, and by $\widetilde{X}^{'}$ the mirror
image of $\widetilde{X}$ by a vertical line.
We call a bullet in $\widetilde{X}^{'}$ a left-end point.
We put the diagram $\widetilde{X}^{'}$ right to $\widetilde{X}$, and 
connect the strands connected to a right-end point of $\widetilde{X}$ 
with the strands connected to a left-end point of $\widetilde{X}^{'}$ 
from bottom to top. We denote by $D(X)$ the generalized chord diagram with $4n$ points.
By construction, $D(X)$ is symmetric along the vertical line in the middle since 
$\widetilde{X}^{'}$ is the mirror image of $\widetilde{X}$.
Further, the chord diagram $D(X)$ satisfies the condition (\ref{eq:condA}).
It is obvious that an element $X$ gives a symmetric generalized chord diagram $D(X)$, 
and $D(X)$ gives a unique element $X$ in $1$-$\mathbb{BFC}_{n}^{(r)}$.
Therefore, we have a bijection between $X$ and $D(X)$, which implies Eq. (\ref{eq:dimBTL}).
\end{proof}

\begin{example}
Let $(n,r)=(3,2)$.
We consider the element $X:=E_{2}^{(2)}E_{3}^{(2)}E_{1}^{(2)}E_{2}^{(1)}\in1\text{-}\mathbb{BFC}_{3}^{(2)}$. 
\begin{figure}[ht]
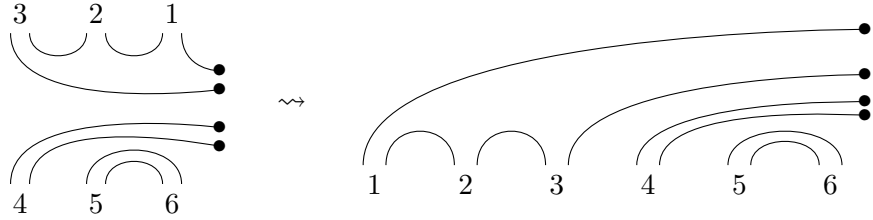

\tikzpic{-0.5}{[scale=0.5]
\draw(2,0)..controls(2,1.2)and(4.5,1.2)..(4.5,0);
\draw(2.5,0)..controls(2.5,0.8)and(4,0.8)..(4,0);
\draw(0,0)..controls(0,2.2)and(5.5,1.5)..(5.5,1.5)node{$\bullet$};
\draw(0.5,0)..controls(0.5,2)and(5.5,1)..(5.5,1)node{$\bullet$};
\draw(0,4)..controls(0,1.8)and(5.5,2.5)..(5.5,2.5)node{$\bullet$};
\draw(0.5,4)..controls(0.5,3.2)and(2,3.2)..(2,4);
\draw(2.5,4)..controls(2.5,3.2)and(4,3.2)..(4,4);
\draw(4.5,4)..controls(4.5,3)and(5.5,3)..(5.5,3)node{$\bullet$};
\draw(0.25,0)node[anchor=north]{$4$}(2.25,0)node[anchor=north]{$5$}
(4.25,0)node[anchor=north]{$6$};
\draw(0.25,4)node[anchor=south]{$3$}(2.25,4)node[anchor=south]{$2$}
(4.25,4)node[anchor=south]{$1$};
}\quad$\leadsto$\quad
\tikzpic{-0.5}{[scale=0.6]
\draw(0,0)..controls(0,3)and(11,3)..(11,3)node{$\bullet$};
\draw(0.5,0)..controls(0.5,1)and(2,1)..(2,0);
\draw(2.5,0)..controls(2.5,1)and(4,1)..(4,0);
\draw(4.5,0)..controls(4.5,2)and(11,2)..(11,2)node{$\bullet$};
\draw(6,0)..controls(6,1.5)and(11,1.4)..(11,1.4)node{$\bullet$};
\draw(6.5,0)..controls(6.5,1.4)and(11,1.1)..(11,1.1)node{$\bullet$};
\draw(8,0)..controls(8,1)and(10.5,1)..(10.5,0);
\draw(8.5,0)..controls(8.5,0.7)and(10,0.7)..(10,0);
\draw(0.25,0)node[anchor=north]{$1$}(2.25,0)node[anchor=north]{$2$}
(4.25,0)node[anchor=north]{$3$}(6.25,0)node[anchor=north]{$4$}
(8.25,0)node[anchor=north]{$5$}(10.25,0)node[anchor=north]{$6$};
}

\caption{A folding of an element in $1$-$\mathbb{BFC}_{3}^{(2)}$.}
\label{fig:dimBTLtoSCD}
\end{figure}
The element $X$ is depicted as the left picture in Figure \ref{fig:dimBTLtoSCD}.
The diagram $\widetilde{X}$ obtained from $X$ is depicted in the right picture.
\end{example}

The number $B_{n+1}^{(r)}$ also counts the number of diagrams defined below.
Let $\Gamma_{n}^{(r)}$ be the set of elements $X$ in $1$-$\mathbb{BFC}_{n}^{(r)}$
such that the diagram presentation $D(X)$ of an element $X$ satisfies the following conditions:
\begin{enumerate}
\item The diagram $D(X)$ is symmetric along the horizontal line in the middle.
\item The diagram $D(X)$ has at most $2r$ right-end points.
\item A right-end point in $D(X)$ is connected to the $i$-th top or bottom point from left 
where $i\equiv n \mod 2$.
\end{enumerate}

The next proposition gives the number of elements in $\Gamma_{n}^{(r)}$ in terms of $B_{n+1}^{(r)}$.
\begin{prop}
We have $|\Gamma_{n}^{(r)}|=B_{n+1}^{(r)}$.
\end{prop}
\begin{proof}
Suppose that $X$ is in $\Gamma_{n}^{(r)}$, and let $D(X)$ be its diagram representation.
Denote by $s\le 2r$ the number of right-end points in $D(X)$.
We construct a bijection between $D(X)$ and an element in $B_{n+1}^{(r)}$.
We fold the diagram $D(X)$ down to the right such that the $n$ top points are right to 
the $n$ bottom points. 
We insert two points between the top and bottom points.
From the condition (2) of $\Gamma_{n}^{(r)}$, we connect the $s$ strands with right-end points
with the inserted two points. 
We connect the inserted two points by $r-s$ strands. 
Note that we have a unique way of connecting points by strands.
In this way, we have a generalized chord diagram $C(X)$ in $\mathcal{C}_{r(n+1)}$.
We will show that it lies in $\mathcal{SC}_{n+1}^{(r)}$.
The condition (1) implies that the chord diagram $C(X)$ is symmetric along the vertical 
line in the middle since we fold $X$ to the right.
Further, the condition (3) implies that the chord diagram $C(X)$ satisfies 
the condition (\ref{eq:condA}).
From these, we have $C(X)\in\mathcal{SC}_{n+1}^{(r)}$.
Note that the construction of $C(X)\in\mathcal{SC}_{n+1}^{(r)}$ from $X\in\Gamma_{n}^{(r)}$ 
is invertible.
By Definition \ref{defn:Bn}, we have $|\Gamma_{n}^{(r)}|=B_{n+1}^{(r)}$.
\end{proof}

\begin{example}
\begin{figure}[ht]
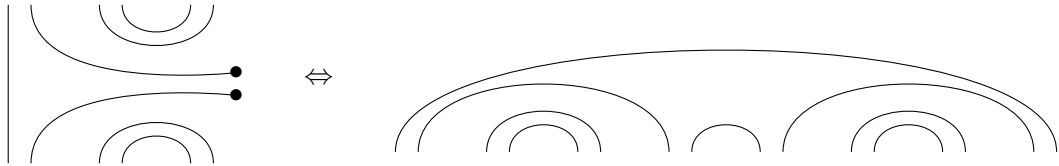

\tikzpic{-0.5}{[scale=0.6]
\draw(0.5,0)..controls(0.5,2)and(5,1.5)..(5,1.5)node{$\bullet$};
\draw(0.5,3.5)..controls(0.5,1.5)and(5,2)..(5,2)node{$\bullet$};
\draw(0,0)--(0,3.5);
\draw(2,0)..controls(2,1.2)and(4.5,1.2)..(4.5,0);
\draw(2.5,0)..controls(2.5,0.8)and(4,0.8)..(4,0);
\draw(2,3.5)..controls(2,2.3)and(4.5,2.3)..(4.5,3.5);
\draw(2.5,3.5)..controls(2.5,2.7)and(4,2.7)..(4,3.5);
}
\quad $\Leftrightarrow$ \quad
\tikzpic{-0.5}{[scale=0.6]
\draw(0,0)..controls(0,3)and(14.5,3)..(14.5,0);
\draw(0.5,0)..controls(0.5,2)and(6,2)..(6,0);
\draw(2,0)..controls(2,1.2)and(4.5,1.2)..(4.5,0);
\draw(2.5,0)..controls(2.5,0.8)and(4,0.8)..(4,0);
\draw(10,0)..controls(10,1.2)and(12.5,1.2)..(12.5,0);
\draw(10.5,0)..controls(10.5,0.8)and(12,0.8)..(12,0);
\draw(6.5,0)..controls(6.5,0.8)and(8,0.8)..(8,0);
\draw(8.5,0)..controls(8.5,2)and(14,2)..(14,0);
}
\caption{A bijection between an element in $\Gamma_{3}^{(2)}$ and an element in $\mathcal{SC}_{4}^{(2)}$.}
\label{fig:GammaBn}
\end{figure}
We consider an element in $\Gamma_3^{(2)}$ depicted as the left picture in Figure \ref{fig:GammaBn}.
The corresponding element in $\mathcal{SC}_{4}^{(2)}$ is depicted right in Figure \ref{fig:GammaBn}.

For $n=2$, the following diagram violates the condition (3) since the first bundled point from left 
is connected to a right-end point:
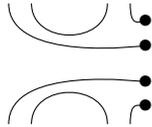
\begin{center}
\tikzpic{-0.5}{[yscale=0.8]
\draw(0,0)..controls(0,1)and(1.8,0.7)..(1.8,0.7)node{$\bullet$};
\draw(0.3,0)..controls(0.3,0.7)and(1.3,0.7)..(1.3,0);
\draw(1.6,0)..controls(1.6,0.4)and(1.7,0.3)..(1.8,0.3)node{$\bullet$};
\draw(0,2)..controls(0,1)and(1.8,1.3)..(1.8,1.3)node{$\bullet$};
\draw(0.3,2)..controls(0.3,1.3)and(1.3,1.3)..(1.3,2);
\draw(1.6,2)..controls(1.6,1.6)and(1.7,1.7)..(1.8,1.7)node{$\bullet$};
}
\end{center}
This diagram corresponds to the diagram in Eq.(\ref{eq:dnadm}), and 
it is not contained in the set $\Gamma_{2}^{(2)}$.
\end{example}

\subsection{Two algebras \texorpdfstring{$\mathbb{SNC}_{n}$}{SNCn} 
and \texorpdfstring{$1$}{1}-\texorpdfstring{$\mathbb{BFC}_{n}^{(1)}$}{BFCn}}
\label{sec:isoSNC1BTL}
In this section, we study the relation between the two algebras $\mathbb{SNC}_{n}$
and $\mathbb{BFC}_{n}^{(1)}$. 
We will see that the action of $G_{i}$ on a symmetric non-crossing partition 
is compatible with that of $E_{i}$ on a symmetric chord diagram.

The algebra $\mathbb{SNC}_{n}$ acts on a symmetric non-crossing partition $\pi(C)$ 
corresponding to an element $C$ in $\mathcal{SC}_{n}$, and the algebra 
$1$-$\mathbb{BFC}_{n}^{(1)}$ acts on an element $\widetilde{C}$ in red-$\mathcal{SC}_{n}$.
From Lemma \ref{lemma:SCredSC}, we have a bijection between $\mathcal{SC}_{n}$ and 
red-$\mathcal{SC}_{n}$.
We will show that the action of $G_{i}$ on $\pi(C)$ coincides with the action of
$E_{i}$ on $\widetilde{C}$.
We have used the convention that the bottom points of $C$ are labeled by $1,1',2,2',\ldots,n,n'$,
however, in this section, we label the bottom points of $C$  (resp. $\widetilde{C}$) from left to right 
from $1$ to $2n$ (resp. $n$).
We denote $\pi':=G_{i}\pi$ for $\pi\in\mathcal{SNC}_{n}$.

We consider the action of $G_1$ on a symmetric non-crossing partition and 
that of $E_1$ on the corresponding chord diagram.
We have three cases in Eq. (\ref{eq:defgp}). The first case on $\pi(C)$ implies that 
we have an arch connecting the bottom points $1$ and $2$ in $\widetilde{C}$.
The action of $E_1$ on $\widetilde{C}$ gives a closed loop, whose weight is $\tau$.
In the second case, $C$ contains two arches connecting $1$ and $2n$, and $2$ and $2n-1$.
In terms of $\widetilde{C}$, we have two half-arches from the two points $1$ and $2$.
The action of $E_1$ on $\widetilde{C}$ gives an arch connecting $1$ and $2$, and 
a strand which starts from a lower odd right-end point to an upper even right-end point.
From Definition \ref{defn:1BTL}, this strand gives a weight $\tau_{1}$.
The existence of an arch connecting $1$ and $2$ corresponds to the merged blocks 
$B(1)\cup B(2)$ and $B(n)\cup B(n+1)$ in $\pi'$.
In the third case, we have two cases:
\begin{enumerate}
\item A symmetric non-crossing partition $\pi$ has a single block $\{1\}$.
Let $B(2)=\{2,\ldots, n_1\}$ and $B(n)=\{n+2-n_1,\ldots, n\}$ be two blocks.
Note that these two blocks do not coincide with each other.
Therefore, we have $n_1<n/2+1$.
The action of $G_1$ on $\pi$ gives a large block $\{1,2,\ldots,n_1,n+2-n_1,\ldots,n\}$.
On the other hand, $\widetilde{C}(\pi)$ contains a strand connecting the point $1$ to a righ-end
point, and an arch connecting $2$ and $n_1\le n$.
The action of $E_1$ on $\widetilde{C}(\pi)$ gives an arch connecting $1$ and $2$, and a strand 
connecting the point $n_1$ to a right-end point. This strand means that $n_1$ and $n+2-n_1$ belong 
to the same block. The chord diagram $\widetilde{C}(\pi')$ is compatible with $\pi'$. 
\item A symmetric non-crossing partition $\pi$ contains a block $B(1)$ such that $|B(1)|\ge2$.
We have $B(1)\neq B(2)$.
Let $n_1$ be the maximal integer in $B(2)$, and $n_2$ be the minimal integer in $B(1)\setminus\{1\}$.
The integers $n_1$ and $n_2$ are well-defined since we have $|B(1)|\ge2$ and $|B(2)|\ge1$.
In terms of $\widetilde{C}$, we have two distinguished arches connecting $1$ and $2n_2-2$, and $2$ and $2n_1-1$ with 
$n_1<n_2\le n$. These two arches are not connected to right-end points.
The action of $E_1$ on $\widetilde{C}$ gives the reconnection of arches. 
We connect the bottom points $1$ and $2$, and $2n_1-1$ and $2n_2-2$.
This reconnection does not involve a strand 
connecting a point to a right-end point. 
It is a routine to check that the action of $E_1$ gives the diagram which is compatible with $\pi'$.
\end{enumerate}
Form these, the action of $G_1$ on $\pi$ is compatible with the action of $E_1$ on $C(\pi)$.
One can show that we have the compatibility for $G_i$ and $E_i$ for $2\le i\le n$ in a similar 
manner.
The next theorem is a direct consequence of the above.
\begin{theorem}
\label{thrm:isoSNC1BTL}
The two algebras $\mathbb{SNC}_n$ and $1$-$\mathbb{BFC}_{n}^{(1)}$ are isomorphic. 
We have $G_{i}\mapsto E_{i}$ for $1\le i\le n$.
\end{theorem}

\subsection{Realization of \texorpdfstring{$1$}{1}-\texorpdfstring{$\mathbb{BFC}_{n}^{(r)}$}{BFCnr} 
by \texorpdfstring{$\mathbb{SNC}_{n}^{(r)}$}{SNCnr}}
In Section \ref{sec:isoSNC1BTL}, we have the isomorphism between the two algebras 
$\mathbb{SNC}_{n}^{(r)}$ and $1$-$\mathbb{BFC}_{n}^{(1)}$.
Below, we see the isomorphism between $\mathbb{SNC}_{n}^{(r)}$ and $1$-$\mathbb{BFC}_{n}^{(r)}$.

\begin{theorem}
The two algebras $\mathbb{SNC}_{n}^{(r)}$ and $1$-$\mathbb{BFC}_{n}^{(r)}$ are isomorphic.
Especially, we have $G_{i}^{(s)}\mapsto E_{i}^{(s)}$. 
\end{theorem}
\begin{proof}
From Theorem \ref{thrm:isoSNC1BTL}, we have an isomorphism between $\mathbb{SNC}_{n}$ and 
$1$-$\mathbb{BFC}_{n}^{(1)}$. Especially, we have $G_{i}\mapsto E_{i}^{(1)}$.
Suppose $r\ge2$.
The generators $G_{i}^{(s)}$ acts on $\pi^{(r)}:=(\pi_1,\ldots,\pi_r)\in\mathcal{SCN}_{n}^{(r)}$
as the identity on the $r-s$ symmetric non-crossing partitions, and as $G_{i}$ on the $s$ 
symmetric non-crossing partitions. 
This corresponds to the diagram of $E_{i}^{(s)}$, that is, $E_{i}^{(s)}$ acts on $r-s$ strands 
as the identity, and on $s$ strands as the cap-cup operator.
Therefore, the isomorphism for the $r=1$ case insures that  we have an isomorphism between 
$G_{i}^{(s)}$ and $E_{i}^{(s)}$.
This completes the proof.
\end{proof}

\section{Two-boundary Fuss--Catalan algebra}
\label{sec:2BFC}
\subsection{Definition}
\label{sec:defn2BTL}
We depict the generator $E_{0}^{(s)}$, $1\le s\le r$, as
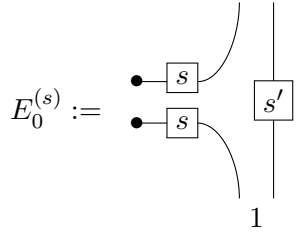
\begin{align*}
E_{0}^{(s)}:=
\tikzpic{-0.5}{
\node(1)at(0,-1){\framebox{$s'$}};
\draw(0,0.3)--(0,-0.74)(0,-1.26)--(0,-2.3);
\node(2)at(-1.2,-0.7){\framebox{$s$}};
\draw(-0.45,0.3)..controls(-0.45,-0.2)and(-0.69,-0.75)..(-0.99,-0.75);
\draw(-1.4,-0.75)--(-1.8,-0.75)node{$\bullet$};
\node(3)at(-1.2,-1.3){\framebox{$s$}};
\draw(-0.45,-2.3)..controls(-0.45,-1.8)and(-0.7,-1.3)..(-1,-1.3);
\draw(-1.4,-1.3)--(-1.8,-1.3)node{$\bullet$};
\draw(-0.225,-2.3)node[anchor=north]{$1$};
}
\end{align*}
where $s'=r-s$. The action of $E_{0}^{(s)}$ is local, which means that 
we have vertical bundled strands at the positions except $1$.
We call a bullet $\bullet$ a left-end point.
The diagram  $E_{0}^{(s)}$ contains $2s$ left-end points. 

We call the $i$-th left-end point from the bottom an odd (resp. even)
left-end point if $i\equiv1\pmod2$ (resp. $i\equiv0\pmod2$).

\begin{defn}
The two-boundary Fuss--Catalan algebra $2$-$\mathbb{BFC}_{n}^{(r)}$ is a unital 
associative algebra over $\mathbb{C}[q,q^{-1},q_n,q_n^{-1},q_0,q_0^{-1}]$
generated by the set of generators $\{E_{i}^{(s)}: 0\le i\le n, 1\le s\le r\}$.
The set $\{E_i^{(s)}: 1\le i\le n, 1\le s\le r\}$ generates $1$-$\mathbb{BFC}_{n}^{(r)}$.
The product of $X,Y\in2\text{-}\mathbb{BFC}_{n}^{(r)}$ is calculated by putting 
the diagram of $Y$ on top of the diagram of $X$.
If we have a strand from an even (resp. odd) left-end point to an odd (resp. even) 
left-end point, we remove it and give a factor $-(q_0+q_{0}^{-1})$ (resp. $qq_{0}^{-1}+q^{-1}q_0$).
\end{defn}

The algebras $\mathbb{TL}_{n}^{(r)}$ and $1$-$\mathbb{BFC}_{n}^{(r)}$ are finite dimensional.
In contrast, the algebra $2$-$\mathbb{BFC}_{n}^{(r)}$ is infinite dimensional.
For example, for $n=2$, the elements $(E_{1}^{(r)}E_{0}^{(r)}E_{2}^{(r)})^{m}$, $m\ge2$ contain 
several strands which connect left-end points with right-end points.
Due to the existence of such strands, these elements cannot be reduced.
To have a finite dimensional representation of two-boundary Temperley--Lieb algebra, 
an additional relation is introduced in \cite{deGNic09}.

Following \cite{deGNic09}, we introduce an additional relations in $2$-$\mathbb{BFC}_{n}^{(r)}$.
\begin{enumerate}[({$\star$}1)]
\item
If we have a strand from a left-end point to a right-end point, we remove it and give 
a factor $\theta$.
\end{enumerate}
Below, we consider the finite dimensional $2$-$\mathbb{BFC}_{n}^{(r)}$ which satisfies 
the condition ($\star$1).

\begin{remark}
When $r=1$, $2$-$\mathbb{BFC}_{n}^{(1)}$ is a generalization of the Temperley--Lieb 
algebra called two-boundary Temperley--Lieb algebra \cite{deGNic09}.
Then, $2$-$\mathbb{BFC}_{n}^{(1)}$ is generated by the set of generators 
$\{e_i: 0\le i\le n\}$. The subset $\{e_{i}:1\le i\le n\}$ generate the 
one-boundary Temperley--Lieb algebra $1$-$\mathbb{BFC}_{n}^{(1)}$, and 
we have 
\begin{align*}
&e_0^{2}=-(q_0+q_0^{-1})e_{0}, \\
&e_1e_0e_1=\tau''e_{1}, \\
&e_0e_i=e_ie_0, \qquad i\neq1,
\end{align*}
where $\tau''=(qq_0^{-1}+q^{-1}q_{0})$.
\end{remark}

Let $2$-$\mathcal{SC}_{n}^{(r)}$ be the set of generalized chord diagram with left-end 
and right-end points satisfying the condition (\ref{eq:condA}).

\begin{example}
We have nine diagrams in $2$-$\mathcal{SC}_{2}^{(2)}$ as shown in Figure \ref{fig:2SNC}.

\begin{figure}[ht]
\tikzpic{-0.5}{[yscale=0.4,xscale=0.5]
\draw(0,0)..controls(0,2)and(2.5,2)..(2.5,0);
\draw(0.5,0)..controls(0.5,1.5)and(2,1.5)..(2,0);
\draw(0.25,0)node[anchor=north]{$1$}(2.25,0)node[anchor=north]{$2$};
}\qquad
\tikzpic{-0.5}{[yscale=0.4,xscale=0.5]
\draw(0,0)..controls(0,2.2)and(3,1.5)..(3,1.5)node{$\bullet$};
\draw(2.5,0)..controls(2.5,1)and(3,0.8)..(3,0.8)node{$\bullet$};
\draw(0.5,0)..controls(0.5,1.5)and(2,1.5)..(2,0);
\draw(0.25,0)node[anchor=north]{$1$}(2.25,0)node[anchor=north]{$2$};
}\qquad
\tikzpic{-0.5}{[yscale=0.4,xscale=0.5]
\draw(0,0)..controls(0,1)and(-0.5,0.8)..(-0.5,0.8)node{$\bullet$};
\draw(2.5,0)..controls(2.5,1)and(3,0.8)..(3,0.8)node{$\bullet$};
\draw(0.5,0)..controls(0.5,1.5)and(2,1.5)..(2,0);
\draw(0.25,0)node[anchor=north]{$1$}(2.25,0)node[anchor=north]{$2$};
}\qquad
\tikzpic{-0.5}{[yscale=0.4,xscale=0.5]
\draw(0,0)..controls(0,1)and(-0.5,0.8)..(-0.5,0.8)node{$\bullet$};
\draw(2.5,0)..controls(2.5,2.2)and(-0.5,1.5)..(-0.5,1.5)node{$\bullet$};
\draw(0.5,0)..controls(0.5,1.5)and(2,1.5)..(2,0);
\draw(0.25,0)node[anchor=north]{$1$}(2.25,0)node[anchor=north]{$2$};
}
\\[12pt]
\tikzpic{-0.5}{[yscale=0.4,xscale=0.5]
\draw(0,0)..controls(0,3)and(3,2.5)..(3,2.5)node{$\bullet$};
\draw(0.5,0)..controls(0.5,2.5)and(3,2)..(3,2)node{$\bullet$};
\draw(2,0)..controls(2,1.4)and(3,1.3)..(3,1.3)node{$\bullet$};
\draw(2.5,0)..controls(2.5,1)and(3,0.8)..(3,0.8)node{$\bullet$};
\draw(0.25,0)node[anchor=north]{$1$}(2.25,0)node[anchor=north]{$2$};
}\qquad
\tikzpic{-0.5}{[yscale=0.4,xscale=0.5]
\draw(0,0)..controls(0,1)and(-0.5,0.8)..(-0.5,0.8)node{$\bullet$};
\draw(0.5,0)..controls(0.5,2.5)and(3,2)..(3,2)node{$\bullet$};
\draw(2,0)..controls(2,1.4)and(3,1.3)..(3,1.3)node{$\bullet$};
\draw(2.5,0)..controls(2.5,1)and(3,0.8)..(3,0.8)node{$\bullet$};
\draw(0.25,0)node[anchor=north]{$1$}(2.25,0)node[anchor=north]{$2$};
}\qquad
\tikzpic{-0.5}{[yscale=0.4,xscale=0.5]
\draw(0,0)..controls(0,1)and(-0.5,0.8)..(-0.5,0.8)node{$\bullet$};
\draw(0.5,0)..controls(0.5,1.4)and(-0.5,1.3)..(-0.5,1.3)node{$\bullet$};
\draw(2,0)..controls(2,1.4)and(3,1.3)..(3,1.3)node{$\bullet$};
\draw(2.5,0)..controls(2.5,1)and(3,0.8)..(3,0.8)node{$\bullet$};
\draw(0.25,0)node[anchor=north]{$1$}(2.25,0)node[anchor=north]{$2$};
}\qquad
\tikzpic{-0.5}{[yscale=0.4,xscale=0.5]
\draw(0,0)..controls(0,1)and(-0.5,0.8)..(-0.5,0.8)node{$\bullet$};
\draw(0.5,0)..controls(0.5,1.4)and(-0.5,1.3)..(-0.5,1.3)node{$\bullet$};
\draw(2,0)..controls(2,2.5)and(-0.5,2)..(-0.5,2)node{$\bullet$};
\draw(2.5,0)..controls(2.5,1)and(3,0.8)..(3,0.8)node{$\bullet$};
\draw(0.25,0)node[anchor=north]{$1$}(2.25,0)node[anchor=north]{$2$};
}\qquad
\tikzpic{-0.5}{[yscale=0.4,xscale=0.5]
\draw(0,0)..controls(0,1)and(-0.5,0.8)..(-0.5,0.8)node{$\bullet$};
\draw(0.5,0)..controls(0.5,1.4)and(-0.5,1.3)..(-0.5,1.3)node{$\bullet$};
\draw(2,0)..controls(2,2.5)and(-0.5,2)..(-0.5,2)node{$\bullet$};
\draw(2.5,0)..controls(2.5,3)and(-0.5,2.5)..(-0.5,2.5)node{$\bullet$};
\draw(0.25,0)node[anchor=north]{$1$}(2.25,0)node[anchor=north]{$2$};
}
\caption{Nine diagrams in $2$-$\mathcal{SC}_{2}^{(2)}$.}
\label{fig:2SNC}
\end{figure}
\end{example}

We first enumerate the generalized chord diagrams in $2$-$\mathcal{SC}_{n}^{(r)}$.
Let $C\in\mathcal{SC}_{n}^{(r)}$ be a symmetric generalized chord diagram.
Given a diagram $C$, we denote by $v^{\downarrow}(C)$ one plus the number of arches which cross 
the vertical line in the middle.
We define the number $V_{n}^{(r)}$ by
\begin{align*}
V_{n}^{(r)}:=\sum_{C\in\mathcal{SC}_{n}^{(r)}}v^{\downarrow}(C).
\end{align*}

\begin{prop}
\label{prop:2SCVnr}
We have $|2\text{-}\mathcal{SC}_{n}^{(r)}|=V_{n}^{(r)}$.
\end{prop}
\begin{proof}
Let $C$ be a diagram in $2$-$\mathcal{SC}_{n}^{(r)}$, and $C^{\vee}$ be 
a mirror image of $C$ along a vertical line.
If $C$ has $p$ right-end points and $q$ left-end points, then $C^{\vee}$
has $q$ right-end points and $p$ left-end points.
We place $C^{\vee}$ right to $C$ and connect the right-end points of $C$ with 
the left-end points of $C^{\vee}$.
The new diagram $C'$ is symmetric and it has $q$ right-end and $q$ left-end points.
We connect the $i$-th right-end point with the $i$-th left-end point by an arch 
for $1\le i\le q$.
The new diagram $C''$ is a symmetric along the vertical line in the middle.
Note that the number of arches which cross the vertical line in the middle is 
$p+q$.

Conversely, suppose that a symmetric diagram $C''$ has $p+q$ arches which cross the 
vertical line in the middle. 
We cut the diagram $C''$ in the middle and it has $p+q$ right-end points.
We bend the top $q$ strands which have right-end points in such a way that a new diagram has 
$q$ left-end points.
This implies $C''$ gives $v^{\downarrow}(C'')=p+q+1$ diagrams in $2$-$\mathcal{SC}_{n}^{(r)}$.

From these, we have $|2\text{-}\mathcal{SC}_{n}^{(r)}|=V_{n}^{(r)}$.
\end{proof}

Let $K$ be a diagram obtained from a diagram $C\in\mathcal{C}_{2n}^{(r)}$ by folding 
$C$ in the middle such that $K$ has $2n$ top and $2n$ bottom points.
Let $\mathcal{K}_{n}^{(r)}$ be the set of diagrams $K$ such that it is symmetric 
along the vertical line in the middle.
We denote by $v(K)$ half the number of vertical strand in $K$.
If $v(K)=0$, there is no arch which crosses the horizontal line in the middle. 
In this case, we define $v^{\uparrow}(K)$ (resp. $v^{\downarrow}(K)$) be one plus the number of arches which are 
above (resp. below) the horizontal line in the middle, and cross the vertical line in the middle. 
We define the weight $\mathrm{wt}(K)$ by 
\begin{align*}
\mathrm{wt}(K):=
\begin{cases}
v(K), & \text{ if } v(K)\ge1, \\
v^{\uparrow}(K)v^{\downarrow}(K), & \text{ if } v(K)=0. 
\end{cases}
\end{align*}
Then, we define the number $K_{n}^{(r)}$ by 
\begin{align*}
K_{n}^{(r)}:=\sum_{K\in\mathcal{K}_{n}^{(r)}}\mathrm{wt}(K).
\end{align*}

\begin{prop}
Suppose that $2$-$\mathbb{BFC}_{n}^{(r)}$ satisfies ($\star1$).
We have 
\begin{align}
\label{eq:dim2BTL}
\dim(2\text{-}\mathbb{BFC}_{n}^{(r)})=K_{n}^{(r)}.
\end{align}
\end{prop}
\begin{proof}
An element $U$ in $2$-$\mathbb{BFC}_{n}^{(r)}$ is a diagram such that it has 
$n$ top and $n$ bottom points, each point has $r$ strands, and several right-end 
and left-end points. 

Let $K$ be a diagram in $\mathcal{K}_{n}^{(r)}$. By definition, $K$ is symmetric 
along a vertical line in the middle.
First, suppose $v(K)\ge1$. Then, we cut $K$ by a vertical line in the middle. 
A new diagram has several right-end points, and $v(K)$ vertical strands.
We cut $i$, $0\le i\le v(K)-1$, vertical lines from left in the middle, and make them 
$2i$ left-end points. This gives an element $U$.
Secondly, suppose $v(K)=0$.
We cut $K$ along the vertical line in the middle. By definition, we have 
$v^{\uparrow}(K)+v^{\downarrow}(K)-2$ arches which cross the vertical line.
As in the proof of Proposition \ref{prop:2SCVnr}, we have $v^{\uparrow}(K)v^{\downarrow}(K)$
ways to have right-end and left-end points.
Each gives an element $U$.

Conversely, one can construct $K$ from an element $U$.
Let $U^{\vee}$ be the mirror image of $U$ along a vertical line.
We place $U^{\vee}$ right to $U$ and connect right-end points of $U$ and left-end points of $U^{\vee}$.
We denote by $K'$ this new diagram.
First, suppose that $K'$ has a vertical strand.
In this case, the number of left-end points of $K'$ is even, and we connect the $i$-th left-end point 
and the $N+1-i$-the left-end point by a strand, where $N$ is the number of left-end points. 
We do the same procedure for the right-end points of $K'$.
The new diagram $K$ is symmetric and in $\mathcal{K}_{n}^{(r)}$.
Secondly, suppose that $K'$ has no vertical strand.
Then, each left-end point starts from a bottom point or from a top point.
We connect the $i$-th left-end point from bottom and the $i$-th right-end point from bottom, both of which start from bottom points,
by an arch. Similarly, we do the same procedure for right-end and left-end points which start from top points.
The new diagram $K$ is symmetric and in $\mathcal{K}_{n}^{(r)}$.

We have one-to-$\mathrm{wt}(K)$ correspondence between a diagram $K$ and elements $U$'s.
This implies Eq. (\ref{eq:dim2BTL}). 
\end{proof}

\begin{example}
We consider the case $(n,r)=(4,2)$.
Let $K$ be a diagram in $\mathcal{K}_{4}^{(2)}$:
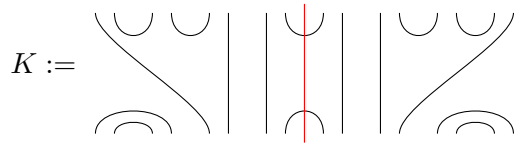
\begin{align*}
K:=
\tikzpic{-0.5}{[xscale=0.5,yscale=0.4]
\draw(0,0)..controls(0,1)and(2,1)..(2,0);
\draw(0.5,0)..controls(0.5,0.5)and(1.5,0.5)..(1.5,0);
\draw(9,0)..controls(9,1)and(11,1)..(11,0);
\draw(9.5,0)..controls(9.5,0.5)and(10.5,0.5)..(10.5,0);
\draw(3,0)..controls(3,1)and(0,3)..(0,4);
\draw(8,0)..controls(8,1)and(11,3)..(11,4);
\draw(3.5,0)--(3.5,4)(7.5,0)--(7.5,4);
\draw(0.5,4)..controls(0.5,3)and(1.5,3)..(1.5,4);
\draw(2,4)..controls(2,3)and(3,3)..(3,4);
\draw(8,4)..controls(8,3)and(9,3)..(9,4);
\draw(9.5,4)..controls(9.5,3)and(10.5,3)..(10.5,4);
\draw(4.5,0)--(4.5,4)(6.5,0)--(6.5,4);
\draw(5,0)..controls(5,1)and(6,1)..(6,0);
\draw(5,4)..controls(5,3)and(6,3)..(6,4);
\draw[red](5.5,-0.3)--(5.5,4.3);
}
\end{align*}
The diagram $K$ is symmetric along the vertical red line in the middle. 
Since the diagram $K$ has $v(K)=3$, $K$ yields the following three 
diagrams in $2\mathbb{BFL}_{4}^{(2)}$:
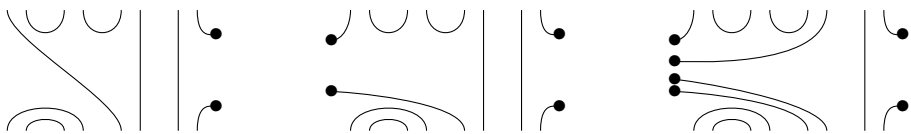
\begin{align*}
\tikzpic{-0.5}{[xscale=0.5,yscale=0.4]
\draw(0,0)..controls(0,1)and(2,1)..(2,0);
\draw(0.5,0)..controls(0.5,0.5)and(1.5,0.5)..(1.5,0);
\draw(3,0)..controls(3,1)and(0,3)..(0,4);
\draw(3.5,0)--(3.5,4);
\draw(0.5,4)..controls(0.5,3)and(1.5,3)..(1.5,4);
\draw(2,4)..controls(2,3)and(3,3)..(3,4);
\draw(4.5,0)--(4.5,4);
\draw(5,0)..controls(5,1)and(5.4,0.8)..(5.5,0.8)node{$\bullet$};
\draw(5,4)..controls(5,3)and(5.4,3.2)..(5.5,3.2)node{$\bullet$};
}
\qquad
\tikzpic{-0.5}{[xscale=0.5,yscale=0.4]
\draw(0,0)..controls(0,1)and(2,1)..(2,0);
\draw(0.5,0)..controls(0.5,0.5)and(1.5,0.5)..(1.5,0);
\draw(3,0)..controls(3,1)and(-0.2,1.3)..(-0.5,1.3)node{$\bullet$};
\draw(0,4)..controls(0,3.5)and(-0.2,3)..(-0.5,3)node{$\bullet$};
\draw(3.5,0)--(3.5,4);
\draw(0.5,4)..controls(0.5,3)and(1.5,3)..(1.5,4);
\draw(2,4)..controls(2,3)and(3,3)..(3,4);
\draw(4.5,0)--(4.5,4);
\draw(5,0)..controls(5,1)and(5.4,0.8)..(5.5,0.8)node{$\bullet$};
\draw(5,4)..controls(5,3)and(5.4,3.2)..(5.5,3.2)node{$\bullet$};
}
\qquad
\tikzpic{-0.5}{[xscale=0.5,yscale=0.4]
\draw(0,0)..controls(0,1)and(2,1)..(2,0);
\draw(0.5,0)..controls(0.5,0.5)and(1.5,0.5)..(1.5,0);
\draw(3,0)..controls(3,1)and(-0.2,1.3)..(-0.5,1.3)node{$\bullet$};
\draw(0,4)..controls(0,3.5)and(-0.2,3)..(-0.5,3)node{$\bullet$};
\draw(4.5,0)--(4.5,4);
\draw(0.5,4)..controls(0.5,3)and(1.5,3)..(1.5,4);
\draw(2,4)..controls(2,3)and(3,3)..(3,4);
\draw(3.5,0)..controls(3.5,1)and(-0.3,1.7)..(-0.5,1.7)node{$\bullet$};
\draw(3.5,4)..controls(3.5,2)and(-0.3,2.3)..(-0.5,2.3)node{$\bullet$};
\draw(5,0)..controls(5,1)and(5.4,0.8)..(5.5,0.8)node{$\bullet$};
\draw(5,4)..controls(5,3)and(5.4,3.2)..(5.5,3.2)node{$\bullet$};
}
\end{align*}
Note that these three diagrams have at least one vertical strands, and 
the right-end (resp. left-end) points are right (resp. left) to the vertical strands.  
The three diagrams correspond to the elements $E_1^{(2)}E_2^{(1)}E_4^{(1)}$, 
$E_{1}^{(2)}E_2^{(1)}E_{0}^{(1)}E_4^{(1)}$, and $E_1^{(2)}E_{0}^{(2)}E_2^{(2)}E_{1}^{(1)}E_4^{(1)}$ 
from left to right.
\end{example}

The $2$-$\mathbb{BFC}_{n}^{(r)}$ contains $1$-$\mathbb{BFC}_{n}^{(r)}$ and $\mathbb{TL}_{n}^{(r)}$
as subalgebras.
More precisely, let $\mathcal{A}_{n}^{(r)}$ be the set of chord diagrams $C\in2\text{-}\mathcal{SC}_{n}^{(r)}$ 
such that $C$ is a diagram without neither right-end and left-end points. 
Note that $\mathcal{A}_{n}^{(r)}$ is bijective to the set $\mathcal{C}_{n}^{(r)}$.
We have $\mathbb{TL}_{n}^{(r)}$ if we take the set of generators 
$\{E_{i}^{(s)} :1\le i\le n-1, 1\le s\le r\}$ and the set $\mathcal{A}_{n}^{(r)}$. 

Let $\mathcal{B}_{n}^{(r)}$ be the set of chord diagrams $C\in2$-$\mathcal{SC}_{n}^{(r)}$ such 
that $C$ is a diagram without left-end points.
Then, the set $\mathcal{B}_{n}^{(r)}$ is bijective to the set red-$\mathcal{SC}_{n}^{(r)}$.
We have $1$-$\mathbb{BFC}_{n}^{(r)}$ if we take the set of generators $\{E_{i}^{(s)}:1\le i\le n, 1\le s\le r\}$
and the set $\mathcal{B}_{n}^{(r)}$.
Similarly, suppose that $\mathcal{B'}_{n}^{(r)}$ is the set of chord diagrams without right-end points.
Then, if we take the set of generators $\{E_{i}^{(s)}: 0\le i\le n-1, 1\le s\le r\}$ and 
the set $\mathcal{B'}_{n}^{(r)}$, we have another $1$-$\mathbb{BFC}_{n}^{(r)}$.

\subsection{Two algebras \texorpdfstring{$2$-$\mathbb{SNC}_{n}$}{2SNCn} 
and \texorpdfstring{$2$-$\mathbb{BFC}_{n}^{(1)}$}{2BFCn}}
In Section \ref{sec:2SNCn}, we have introduced the algebra $2$-$\mathbb{SNC}_{n}$ 
acting on symmetric non-crossing partitions with primed integers.
In Section \ref{sec:defn2BTL}, we have introduced the two-boundary Fuss--Catalan 
algebra $2$-$\mathbb{BFC}_{n}^{(r)}$, $r\ge1$, whose defining relations are 
obtained by diagrammatic calculations.

We first consider the two-boundary Temperley--Lieb algebra and $2$-$\mathbb{SNC}_{n}$.
The next theorem connects $2$-$\mathbb{SNC}_{n}$ with $2$-$\mathbb{BFC}_{n}^{(1)}$.
 
\begin{theorem}
\label{thrm:2SNCiso2BTL1}
The two algebras $2$-$\mathbb{SNC}_{n}$ and $2$-$\mathbb{BFC}_{n}^{(1)}$ are 
isomorphic. We have $G_i\mapsto E_{i}$.
\end{theorem}
To prove Theorem \ref{thrm:2SNCiso2BTL1}, we first show that 
the set $\mathcal{SNC'}_{n}$ is bijective to the set $2$-$\mathcal{SC}_{n}^{(1)}$.

\begin{lemma}
\label{lemma:SNCiso2SC}
There exists a bijection $\mathcal{SNC'}_{n}\xrightarrow{\sim}2$-$\mathcal{SC}_{n}^{(1)}$.
\end{lemma}
\begin{proof}
Suppose we have a non-crossing partition $\pi\in\mathcal{SNC'}_{n}$.
Recall that a chord diagram $C\in2\text{-}\mathcal{SC}_{n}^{(1)}$ 
consists of $n$ points labeled $1,2,\ldots,n$, arches connecting labeled points,
and strands with a left-end and right-end point.
We will construct $C$ from $\pi$.
By definition, if we forget the primes on integers in $\pi$, we have a symmetric chord diagram $\overline{C}$
corresponding to $\pi$. If an integer $i$ in $\pi$ has a prime, we put a ``dot" ($\bullet$) on the 
arch connecting $i$ and $j'$ for some $j'$.   
By the condition on primes, we have a dot only on symmetric arches which intersect the vertical line 
in the middle. Further, the linear order (\ref{eq:lo1}) and the set $S^{sym}(\pi)$ 
imply that we have dots on several outer-most symmetric arches.
We cut the diagram $\overline{C}$ by a vertical line in the middle, and obtain a new diagram $C'$.
The diagram has arches and strands with a right-end point. 
Note that there is no strand with a left-end point in $C'$.
Some strands with a right-end point have a dot. Therefore, we bend strands with a dot leftward such that 
they become strands with a left-end point. In this way, we obtain a diagram $C\in2\text{-}\mathcal{SC}_{n}^{(1)}$ 
from $C'$.
These operations are obviously invertible. We have a natural bijection between $\mathcal{SNC'}_{n}$ 
and $2$-$\mathcal{SC}_{n}^{(1)}$.
\end{proof}

\begin{proof}[Proof of Theorem \ref{thrm:2SNCiso2BTL1}]
From Lemma \ref{lemma:SNCiso2SC}, the vector spaces on which the two algebra act are isomorphic.
It is enough to show that the action of $G_{i}$ on $\mathcal{SNC'}_n$ is 
compatible with the action of $E_{i}$ on $2\text{-}\mathcal{SC}_{n}^{(1)}$.

We first compare the actions of $G_{0}$ on $\pi$ and $E_{0}$ on $C$.
We have three cases from Eq. (\ref{eq:g0}). 
In the first case, the conditions that $n_1=1$ and $1$ is primed in $\pi$ imply 
that $C$ has a strand starting from the point $1$ and with a left-end point. 
The diagram calculation implies that the action of $E_0$ on $C$ gives  $-(q_0+q_0^{-1})C$ 
since we have a strand from an even left-end point to an odd left-end point. 
In the second case, the condition on $\pi$ implies that $C$ has a strand starting from the point $1$ and 
with a right-end point. By the diagram calculation, the action of $E_{0}$ on $C$ gives $\theta C'$ 
where $C'$ is obtained from $C$ by changing the strand from $1$ with a right-end point to a strand 
from $1$ with a left-end point. This is compatible with $\pi^{(5)}$. The factor $\theta$ comes from 
a horizontal strand from a left-end point to a right-end point.
In the third case, the condition on $\pi$ implies that $C$ has an arch from the point $1$ to a primed point 
$j'$ with $j\le n/2$. Thus, $C$ has two non-symmetric arches connecting $1$ and $j'$, and $n+1-j$ and $n'$.
By the diagram calculation, the action of $E_{0}$ on $C$ gives a new diagram $C'$ having 
two strands with a left-end points. In terms of arches, $C'$ has two symmetric arches 
with a dot which connect $1$ and $n'$, and $n+1-j$ and $j'$. 
By the bijection in Lemma \ref{lemma:SNCiso2SC}, the diagram $C'$ gives 
the partition $\pi^{(6)}$.

Secondly, we compare the action of $G_n$ on $\pi$ with that of $E_{n}$ on $C$ as above.
We have four cases from Eq. (\ref{eq:gn2}).
In the first case, the condition on $\pi$ implies that $C$ does not have an arch 
connecting the primed point $(n/2)'$ and the point $n/2+1$ since $|B(n/2+1)|\neq1$.
By the diagram calculation, the action of $E_{n}$ gives a new diagram $C'$ which has
two strands with a right-end point.
It is easy to verify that these two strands correspond to the two blocks $B_1=\{n/2+1\}$ and 
$B(n/2+1)\setminus B_1$. The diagram $C'$ corresponds to $\pi^{(2)}$ by the bijection in Lemma \ref{lemma:SNCiso2SC}.
The second case can be verified in a similar manner to the first case. The main difference is that 
we merge two blocks rather than dividing a block.
In the third case, the condition $\pi$ implies that $C$ has a strand from the point $\lfloor n/2+1\rfloor$ with 
a left-end point. By a diagram calculation, the action of $E_{n}$ on $C$ gives a new diagram $C'$ which 
has a strand from the right-most point with a right-end point and a strand $s$ which has both left-end and right-end 
points.  
The strand $s$ gives the factor $\theta$ by the condition ($\star1$). 
The diagram $C'$ corresponds to $\pi^{(4)}$ by the bijection in Lemma \ref{lemma:SNCiso2SC}.
In the fourth case, the condition on $\pi$ implies that $C$ has a strand from the right-most point with 
a right-end point. It is easy to verify that the action of $E_n$ on $C$ gives $-(q_n+q_{n}^{-1})C$.

Finally, we consider $G_{i}$ and $E_{i}$ with $1\le i\le n-1$.
We have five cases from Eq. (\ref{eq:gi2}).
In the first case, the condition on $\pi$ implies that we have an arch connecting $1$ and $1'$ in $C$.
It is obvious that $E_{i}C=-(q+q^{-1})C$.
In the second case, the condition on $\pi$ implies that $C$ has two strands from the two left-most points 
with a right-end point. By a diagram calculation, the action of $E_i$ on $C$ gives a new diagram $C'$
such that it has an arch connecting the two left-most points, and a strand from an odd (resp. even) right-end point 
to an even (resp. odd) right-end point if $i\equiv1\pmod2$ (resp. $i\equiv0\pmod2$).
Then, it is clear that $C'$ corresponds to $\pi^{(1)}$ and has an overall factor $\tau_{i}$. 
In the third case, the condition on $\pi$ implies that $C$ has a strand from the point $1$ with a left-end point, and 
a strand from the point $2$ with a right-end point.
By a diagram calculation, $E_{i}C$ is a new diagram $C'$ which has an arch connecting the two left-most points, and 
a strand from a left-end point to a right-end point. From the condition ($\star1$), we have a factor $\theta$.
By the bijection in Lemma \ref{lemma:SNCiso2SC}, $C'$ corresponds to $\pi^{(1)}$. 
In the fourth case, the condition on $\pi$ implies that $C$ has two strands from the left-most points 
with a left-end point. By a diagram calculation, $E_{i}C$ gives  a new diagram $C'$ which has an arch connecting 
the two left-most points and a strand from an odd (resp. even) left-end point to an even (resp. odd) 
left-end point	if $i\equiv1\pmod2$ (resp. $i\equiv0\pmod2$). The strand from a left-end point to another left-end point
gives the factor $\tau'_{i}$. The diagram $C'$ corresponds to $\pi^{(1)}$ with the factor $\tau'_{i}$.
In the fifth case, the condition on $\pi$ implies that $C$ has an arch from the point $1$ to another point $j>2$.
The diagram $C$ has another arch from $2$ to $i$ with $3<i<j$. 
Then, the action of $E_{i}$ on $C$ gives a new diagram $C'$ which has arches from $1$ to $2$ and from $i$ to $j$.
It is clear that $C'$ corresponds to $\pi^{(1)}$ without an extra overall factor.

In all cases, the action of $E_i$ on $C$ coincides with that of $G_{i}$ on $\pi$. This completes the proof.
\end{proof}

\subsection{Realization of \texorpdfstring{$2$-$\mathbb{BFC}_{n}^{(r)}$}{2BFC} 
by \texorpdfstring{$2$-$\mathbb{SNC}_{n}^{(r)}$}{2SNC}}
We generalize Theorem \ref{thrm:2SNCiso2BTL1} to the case of $r\ge2$.

Let $\pi^{(r)}:=(\pi_1,\ldots,\pi_{r})$ be a symmetric $r$-chain of $\mathcal{SNC}_{n}$.
Let $C$ be a chord diagram in $\mathcal{SC}_{n}^{(r)}$ corresponding to $\pi^{(r)}$.
Below, we consider an $r$-chain of symmetric non-crossing partitions $\pi'^{(r)}$ with primed integers.
A primed integer corresponds to an arch with a ``dot" $\bullet$ in $C$.
Namely, if the integer $i$ has a prime in $\pi'^{(r)}$, then the corresponding arch connecting $i$ and $j'$
has a dot on it.  
Suppose that $C\in\mathcal{SC}_{n}^{(r)}$ has $m$ symmetric arches which intersect the vertical line in the middle.
Then, we say that the chord diagram $C$ is admissible, if the $k\le m$ outer-most symmetric arches have a dot.

\begin{defn}
\label{defn:SCdash}
We denote by $\mathcal{SC'}_{n}^{(r)}$ the set of admissible symmetric chord diagrams with dots.
\end{defn}

\begin{remark}
Two remarks are in order.
\begin{enumerate}
\item
We have two types of symmetric arches in a diagram in $\mathcal{SC'}_{n}^{(r)}$. 
Since the underlying chord diagram is 
in $\mathcal{SC}_{n}^{(r)}$, {\it i.e.}, the diagram is symmetric, a symmetric arch without a dot 
corresponds to a strand which has a right-end point.
Similarly, a symmetric arch with a dot corresponds to a strand which has a left-end point.
\item
The admissibility condition of a chord diagram implies that there is no strand with 
right-end (resp. left-end) point which is between two left-end (resp. right-end) points. 
\end{enumerate}
\end{remark}

\begin{theorem}
\label{thrm:2SNC2BTL}
The two algebras $2$-$\mathbb{SNC}_{n}^{(r)}$ and $2$-$\mathbb{BFC}_{n}^{(r)}$
are isomorphic. We have $G_{i}^{(s)}\mapsto E_{i}^{(s)}$.		
\end{theorem}

To prove Theorem \ref{thrm:2SNC2BTL}, we study the two vector spaces 
$\mathcal{SNC'}_{n}^{(r)}$ and $\mathcal{SC'}_{n}^{(r)}$.
\begin{lemma}
\label{lemma:SNCdisoSCd}
We have a bijection $\mathcal{SNC'}_{n}^{(r)}\xrightarrow{\sim}\mathcal{SC'}_{n}^{(r)}$.
\end{lemma}
\begin{proof}
From Definition \ref{defn:Bn} and Proposition \ref{prop:SNCBn}, we have a bijection 
between $\mathcal{SNC}_{n}^{(r)}$ and $\mathcal{SC}_{n}^{(r)}$.
It is enough to show that a primed integer in a symmetric non-crossing partition
is bijective to an arch with a dot in a symmetric chord diagram.
However, this correspondence is clear by comparing Definition \ref{defn:SNCdash} 
with Definition \ref{defn:SCdash}.
This completes the proof.
\end{proof}

\begin{proof}[Proof of Theorem \ref{thrm:2SNC2BTL}]
From Lemma \ref{lemma:SNCdisoSCd}, we have a bijection between the vector spaces 
on which two algebras act. It is enough to show that the action of $G_{i}^{(s)}$ on 
$\mathcal{SNC'}_{n}^{(r)}$ is compatible with the action of $E_{i}^{(s)}$ on 
$\mathcal{SC'}_{n}^{(r)}$.
From Theorem \ref{thrm:2SNCiso2BTL1}, we have an isomorphism in the case of $r=1$.
A diagram $C$ in $\mathcal{SC'}_{n}^{(r)}$ is a superposition of $r$ diagrams in $\mathcal{SC'}_{n}^{(1)}$.
The definition (\ref{eq:defnGis}) of $G_{i}^{(s)}$ implies that the action $g_{i}^{(s)}$ on $\pi^{(r)}$ 
is the identity on the $r-s$ non-crossing partitions, and $g_{i}$ on the $s$ non-crossing partition.
Further, the definition (\ref{eq:defngis}) implies that $G_{i}$ acts on $s$ outer-most arches of $C$.
This coincides with the definition of $E_{i}^{(s)}$, {\it i.e.}, the action of $E_{i}^{(s)}$ changes 
the connectivity of $s$ inner-most arches in $C$. These mean that the isomorphism in the case of $r=1$ is preserved 
by the superposition of the $r$ diagrams. 
Therefore, we have an isomorphism $G_{i}^{(s)}\mapsto E_{i}^{(s)}$.
This completes the proof.
\end{proof}

\section{Integrability}
\label{sec:RE}
\subsection{The Yang--Baxter equation and the reflection equation}
In what follows, we consider a two-dimensional statistical model. 
In this paper, we consider only the square lattice vertex models.
Then, the Yang--Baxter equation \cite{Bax82} is a sufficient condition 
for the existence of an infinite set of commuting transfer matrices 
for the model. 
The Yang--Baxter equation is explicitly given by 
\begin{align}
\label{eq:YBE}
R_{i}(w)R_{i+1}(wz)R_{i}(z)=R_{i+1}(z)R_{i}(wz)R_{i+1}(w),
\end{align}
with the normalization condition
\begin{align*}
R_{i}(w)R_{i}(1/w)=1.
\end{align*}
We impose a further condition 
\begin{align*}
R_{i}(1)=1,
\end{align*}
for the normalization of $R_{i}(w)$.

Similarly, the reflection equation together with the Yang--Baxter equation 
is a sufficient condition for the existence of an infinite set of commuting 
transfer matrices for the model with boundaries.
The reflection equation \cite{Skl88} is explicitly given by
\begin{align}
\label{eq:RKRK}
K_2(w)R_{1}(1/(wz))K_2(z)R_{1}(w/z)=R_{1}(w/z)K_2(z)R_{1}(1/(wz))K_2(w),
\end{align}
with the normalization condition
\begin{align*}
K_2(w)K_{2}(1/w)=1.
\end{align*}
As in the case of $R_{i}(w)$, we further impose the condition
\begin{align*}
K_2(1)=1,
\end{align*}
for the normalization of $K_2(w)$.

Similarly, the reflection equation for the left boundary is given by 
\begin{align}
\label{eq:RKRK2}
K_{0}(z)R_{1}(zw)K_{0}(w)R_{1}(w/z)=R_{1}(w/z)K_{0}(w)R_{1}(wz)K_0(z),
\end{align}
with the normalization conditions
\begin{align*}
&K_{0}(w)K_{0}(1/w)=1, \\
&K_{0}(1)=1.
\end{align*}
We focus on the reflection equation (\ref{eq:RKRK}) since Eq. (\ref{eq:RKRK2}) can 
be solved by a similar way.

\subsection{Solution for \texorpdfstring{$r=2$}{r=2}}
In \cite{Dif98}, Di Francesco obtained a solution of the Yang--Baxter equation (\ref{eq:YBE})
for $r\ge1$.
As in \cite{Dif98}, we consider the solution for $r=2$ which has the following form:
\begin{align}
\label{eq:R1}
R_{i}(w):=\mathbf{1}_{i}+r_{1}(w)E_i^{(1)}+r_2(w)E_{i}^{(2)},
\end{align}
where $\mathbf{1}_{i}$ is the identity. 
By solving the functional relations obtained from the Yang--Baxter equation (\ref{eq:YBE}),
one can obtain 
\begin{align}
\label{eq:rsol1}
r_1(w)=\tau^{-1}(w-1), \qquad r_{2}(w)=\genfrac{}{}{}{}{w(w-1)}{\tau^2-1-w}.
\end{align}
This solution was explicitly given in \cite{Dif98}.

Given the solution (\ref{eq:rsol1}) of the Yang--Baxter equation, a
solution of the reflection equation can be given as follows.
\begin{prop}
\label{prop:RE}
Let $r=2$.
If $\tau \tau_e\neq \tau_o$, $\tau \tau_o\neq\tau_e$ and $\tau^2\neq1$, 
the solution of the reflection equation (\ref{eq:RKRK}) is given by
\begin{align}
\label{eq:K1}
K_{2}(w)=\mathbf{1}+k_1(w)E_2^{(1)}+k_2(w)E_2^{(2)},
\end{align}
where 
\begin{align}
\label{eq:k1k2sol1}
\begin{aligned}
k_1(w)=-\genfrac{}{}{}{}{C_2(w^2-1)(w-\tau_{o}C_1)}{w(1-2\tau_{o}\tau_{e}C_1C_2 w+\tau\tau_{e}^2C_1C_2w+\tau_{e}C_{2}w^2)}, \\
k_2(w)=-\genfrac{}{}{}{}{\tau C_1C_2(w^2-1)}{w(1-2\tau_{o}\tau_{e}C_1C_2 w+\tau\tau_{e}^2C_1C_2w+\tau_{e}C_{2}w^2)},
\end{aligned}
\end{align}
and
\begin{align}
C_1=\pm\sqrt{\genfrac{}{}{}{}{\tau\tau_{o}-\tau_{e}}{(\tau^2-1)\tau_{e}\tau_{o}(\tau\tau_{e}-\tau_{o})}},
\quad C_2=-\genfrac{}{}{}{}{\tau^2-1}{\tau\tau_{o}-\tau_{e}}.
\end{align}
Similarly, if $\tau\tau_e=\tau_o$ or $\tau\tau_o=\tau_e$, then we have 
\begin{align}
\label{eq:REsol2}
k_1(w)=-\genfrac{}{}{}{}{w^2-1}{\tau_e w^2}, \quad
k_2(w)=\genfrac{}{}{}{}{\tau(w^2-1)}{\tau_o\tau_e w^2}.
\end{align}
\end{prop}
\begin{proof}
Given a solution of the Yang--Baxter equation, we look for a solution of the reflection
equation (\ref{eq:RKRK}).
Since we consider the right boundary of the system, it is enough to 
consider $1$-$\mathbb{BFC}_{2}^{(2)}$ for simplicity. 
This algebra has only four generators 
$\{E_{i}^{(s)}: 1\le i\le 2, 1\le s\le2\}$.
They satisfy the following relations:
\begin{align*}
E_1^{(s)}E_{1}^{(s')}&=\tau^{\min\{s,s'\}}E_{1}^{(\max\{s,s'\})}, \\
E_2^{(s)}E_{2}^{(s')}&=\tau_{e}^{\min\{s,s'\}}E_2^{(\max\{s,s'\})},
\end{align*}
and 
\begin{align*}
&E_1^{(2)}E_2^{(2)}E_{1}^{(2)}=\tau_o^{2}E_1^{(2)}, 
\quad E_1^{(2)}E_2^{(2)}E_{1}^{(1)}=\tau_o E_1^{(2)}E_2^{(1)},
\quad E_{1}^{(2)}E_{2}^{(1)}E_1^{(2)}=\tau\tau_o E_{1}^{(2)}, \\
&E_1^{(2)}E_{2}^{(1)}E_{1}^{(1)}=\tau E_1^{(2)}E_{2}^{(1)}, 
\quad E_1^{(1)}E_{2}^{(2)}E_{1}^{(2)}=\tau_o E_2^{(1)}E_1^{(2)}, 
\quad E_1^{(1)}E_{2}^{(2)}E_{1}^{(1)}=\tau_o E_2^{(1)}E_1^{(1)}, \\
&E_1^{(1)}E_{2}^{(1)}E_{1}^{(2)}=\tau E_{2}^{(1)}E_1^{(2)}, 
\quad E_1^{(1)}E_{2}^{(1)}E_{1}^{(1)}=\tau E_2^{(1)}E_1^{(1)},
\end{align*}
where $\tau_{0}:=qq_2^{-1}+q^{-1}q_{2}$ and $\tau_{e}:=-(q_2+q_2^{-1})$.

We will find a solution of the reflection equation (\ref{eq:RKRK})
of the form (\ref{eq:K1}). 
By substituting the expressions (\ref{eq:R1}) and (\ref{eq:K1}) into the 
reflection equation (\ref{eq:RKRK}), we find that the coefficients $k_i(w)$ and $r_i(w)$
should satisfy the relations given below. 
From the coefficients of $E_1^{(1)}E_2^{(2)}$ or $E_2^{(2)}E_1^{(1)}$, we have 
\begin{align}
\label{eq:condK1}
\begin{aligned}
&r_1(w/z) k_2(w) + r_1(1/(zw)) k_2(w) + \tau\ r_1(w/z) r_1(1/(zw)) k_2(w) + \tau_e\ r_1(w/z) k_1(z) k_2(w)\\
&\ + \tau_e\ r_1(1/(zw)) k_1(z) k_2(w) + \tau\tau_e\  r_1(w/z) r_1(1/(zw)) k_1(z) k_2(w) + r_1(w/z) k_2(z) \\ 
&\ -r_1(1/(zw)) k_2(z)+ \tau_e\ r_1(w/z) k_1(w) k_2(z) - \tau_e\ r_1(1/(zw)) k_1(w) k_2(z) \\ 
&\ + \tau_e^2\ r_1(w/z) k_2(w)k_2(z) + \tau_{o}\tau_e\ r_1(w/z) r_1(1/(zw)) k_2(w) k_2(z)=0.
\end{aligned}
\end{align}
From the coefficients of $E_2^{(1)}E_1^{(2)}$ or $E_1^{(2)}E_2^{(1)}$, we have
\begin{align}
\label{eq:condK2}
\begin{aligned}
&r_2(w/z)k_1(w) + \tau\ r_1(1/(zw)) r_2(w/z) k_1(w) + r_2(1/(zw)) k_1(w) + \tau\ r_1(w/z) r_2(1/(zw)) k_1(w)\\
&\ + \tau^2\ r_2(w/z) r_2(1/(zw)) k_1(w) + r_2(w/z) k_1(z) + \tau\ r_1(1/(zw)) r_2(w/z) k_1(z)\\ 
&\ - r_2(1/(zw)) k_1(z) - \tau \ r_1(w/z) r_2(1/(zw)) k_1(z) + \tau_e\ r_2(w/z) k_1(w) k_1(z)\\ 
&\ + \tau \tau_e\ r_1(1/(zw)) r_2(w/z) k_1(w) k_1(z) + \tau \tau_o\ r_2(w/z) r_2(1/(zw)) k_1(w) k_1(z) \\
&\ + \tau_o\ r_1(1/(zw)) r_2(w/z) k_2(z) - \tau_o\ r_1(w/z) r_2(1/(zw)) k_2(z) \\
&\ + \tau_o\tau_e\ r_1(1/(zw)) r_2(w/z) k_1(w) k_2(z) + \tau_o^2\ r_2(w/z) r_2(1/(zw)) k_1(w) k_2(z)=0
\end{aligned}
\end{align}

From the coefficients of $E_2^{(2)}E_1^{(2)}$ or $E_1^{(2)}E_2^{(2)}$, we have 
\begin{align}
\label{eq:condK3}
\begin{aligned}
&r_2(w/z) k_2(w) + \tau\ r_1(1/(zw)) r_2(w/z) k_2(w) + r_2(1/(zw)) k_2(w) + \tau\ r_1(w/z) r_2(1/(zw)) k_2(w)\\
&\ + \tau^2\ r_2(w/z) r_2(1/(zw)) k_2(w) + \tau_e\ r_2(w/z) k_1(z) k_2(w) + \tau\tau_{e}\ r_1(1/(zw)) r_2(w/z) k_1(z) k_2(w)\\
&\ + \tau \tau_o\ r_2(wz) r_2(1/(zw)) k_1(z) k_2(w) + r_2(w/z) k_2(z) - r_2(1/(zw)) k_2(z)\\ 
&\ + \tau_e\ r_2(w/z) k_1(w) k_2(z) + \tau_e^2\ r_2(w/z) k_2(w) k_2(z) + \tau_o\tau_e\ r_1(1/(zw)) r_2(w/z) k_2(w) k_2(z) \\ 
&\ + \tau_o^2\ r_2(w/z) r_2(1/(zw)) k_2(w) k_2(z)=0.
\end{aligned}
\end{align}
From the coefficients of $E_2^{(2)}E_1^{(2)}E_2^{(1)}$ or $E_2^{(1)}E_1^{(2)}E_{2}^{(2)}$, we have 
\begin{align}
\label{eq:condK4}
\begin{aligned}
&r_2(1/(zw)) k_1(z) k_2(w)+\tau\ r_1(w/z) r_2(1/(zw)) k_1(z) k_2(w) \\
&\ - r_2(1/(zw)) k_1(w) k_2(z) + \tau_o\ r_1(w/z) r_2(1/(zw)) k_2(w) k_2(z)=0.
\end{aligned}
\end{align}
Further, from the normalization condition, we have 
\begin{align}
\label{eq:normK}
\begin{aligned}
&k_1(w) + k_1(1/w) + \tau_e\ k_1(w) k_1(1/w)=0, \\
&k_2(w) + \tau_e\ k_1(w)k_2(1/w) +k_2(1/w) + \tau_e\ k_1(1/w) k_2(w) + \tau_e^2 k_2(w)k_2(1/w)=0.
\end{aligned}
\end{align}

By substituting the expression (\ref{eq:rsol1}) into Eq. (\ref{eq:condK4}), we obtain
\begin{align*}
\tau w k_1(z)k_2(w)-\tau z k_1(w)k_2(z)+\tau_{o}w k_2(w)k_2(z)-\tau_{o} zk_2(w)k_2(z)=0.
\end{align*}
This equation is equivalent to 
\begin{align}
\label{eq:k1k21}
\tau k_1(w)+\tau_{o}k_2(w)=0,
\end{align}
or
\begin{align}
\label{eq:condK5}
\genfrac{}{}{}{}{wk_2(w)}{\tau k_1(w)+\tau_{o}k_2(w)}=\genfrac{}{}{}{}{zk_2(z)}{\tau k_1(z)+\tau_{o}k_2(z)}=C_{1},
\end{align}
where $C_1$ is a constant since the l.h.s. (resp. r.h.s) is a function of only $w$ (resp. $z$).

We first consider Eq. (\ref{eq:k1k21}). We have $k_2(w)=-\tau\tau_{o}^{-1}k_1(w)$. By substituting this relation 
into Eq. (\ref{eq:condK1}), we obtain 
\begin{align}
\tau_o+\tau\tau^2_ek_1(w)w^2-\tau_oz^2-2k_1(w)\tau_e\tau_ow^2=0,
\end{align}
or 
\begin{align}
C'_2 k_1(w)w=\tau_o+\tau\tau^2_ek_1(w)w^2-\tau_oz^2-2k_1(w)\tau_e\tau_ow^2,
\end{align}
where $C'_2$ is a constant. In both cases, by solving the equations with respect to $k_{1}(w)$, and 
substituting it into Eqs. (\ref{eq:condK2}), (\ref{eq:condK3}) and (\ref{eq:normK}),
we obtain Eq. (\ref{eq:REsol2}) with $C'_2=0$ and $\tau\tau_e=\tau_o$.
 
Secondly, we consider Eq. (\ref{eq:condK5}).
By solving Eq. (\ref{eq:condK5}) with respect to $k_2(w)$, we obtain 
\begin{align}
\label{eq:k2byk1}
k_2(w)=\genfrac{}{}{}{}{\tau C_1 k_1(w)}{w-\tau_{o}C_{1}}.
\end{align}
By substituting Eq. (\ref{eq:rsol1}) and Eq. (\ref{eq:k2byk1}) into Eq. (\ref{eq:condK1}), we have 
\begin{align}
\label{eq:condk1}
\begin{aligned}
&-\tau_{o} C_1 w k_1(w) + w z k_1(w) + \tau_{o} C_1 w z^2 k_1(w) - w z^3 k_1(w) \\
&\quad+ \tau_{o} C_1 z k_1(z) -  w z k_1(z) - \tau_{o} C_1 w^2 z k_1(z) 
+ w^3 z k_1(z) - 2 \tau_{o} \tau_{e} C_1 w^2 z k_1(w) k_1(z) \\ 
&\quad+ \tau \tau_{e}^2 C_1 w^2 z k_1(w) k_1(z) + \tau_{e} w^3 z k_1(w) k_1(z) 
+ 2 \tau_{o}\tau_{e}C_1 w z^2 k_1(w) k_1(z) \\ 
&\quad- \tau \tau_{e}^2 C_1 w z^2 k_1(w) k_1(z) - \tau_{e} w z^3 k_1(w) k_1(z)=0.
\end{aligned}
\end{align}
By a separation of variables, Eq. (\ref{eq:condk1}) is equivalent to 
either
\begin{align}
\label{eq:condk1null}
-\tau_{o}C_1+w +\tau_{o}C_1w^2-w^3+2\tau_{o}\tau_{e}C_1w^2k_1(w)-\tau\tau_{e}^{2}C_1w^2k_1(w)-\tau_{e}w^3k_1(w)=0
\end{align}
or
\begin{align}
\label{eq:condk12}
\genfrac{}{}{}{}{wk_1(w)}{-\tau_{o}C_1+w +\tau_{o}C_1w^2-w^3+2\tau_{o}\tau_{e}C_1w^2k_1(w)-\tau\tau_{e}^{2}C_1w^2k_1(w)-\tau_{e}w^3k_1(w)}
=C_2,
\end{align}
where $C_2$ is a constant.

We fist consider Eq. (\ref{eq:condk1null}). By solving Eq. (\ref{eq:condk1null}) for $k_1(w)$, and 
substituting the expression into Eqs. (\ref{eq:condK2}), (\ref{eq:condK3}) and (\ref{eq:normK}),
we obtain Eq. (\ref{eq:REsol2}) with $C_1=0$ and $\tau\tau_o=\tau_e$.

Finally, we consider Eq. (\ref{eq:condk12}).
From Eq. (\ref{eq:k2byk1}) and Eq. (\ref{eq:condk12}), we obtain 
\begin{align}
\label{eq:k1k2sol2}
\begin{aligned}
k_1(w)=-\genfrac{}{}{}{}{C_2(w^2-1)(w-\tau_{o}C_1)}{w(1-2\tau_{o}\tau_{e}C_1C_2 w+\tau\tau_{e}^2C_1C_2w+\tau_{e}C_{2}w^2)}, \\
k_2(w)=-\genfrac{}{}{}{}{\tau C_1C_2(w^2-1)}{w(1-2\tau_{o}\tau_{e}C_1C_2 w+\tau\tau_{e}^2C_1C_2w+\tau_{e}C_{2}w^2)}.
\end{aligned}
\end{align}
To fix constants $C_1$ and $C_2$, we substitute Eq. (\ref{eq:k1k2sol2}) into 
Eqs. (\ref{eq:condK2}), (\ref{eq:condK3}) and (\ref{eq:normK}).
These uniquely fix the values of $C_1$ and $C_2$ as 
\begin{align}
C_1&=\pm\sqrt{\genfrac{}{}{}{}{\tau\tau_{o}-\tau_{e}}{(\tau^2-1)\tau_{e}\tau_{o}(\tau\tau_{e}-\tau_{o})}}, 
\qquad C_2=-\genfrac{}{}{}{}{\tau^2-1}{\tau\tau_{o}-\tau_{e}}.
\end{align}
\end{proof}

\bibliographystyle{amsplainhyper} 
\bibliography{biblio}

\end{document}